\documentclass[a4paper, 11pt, twoside, reqno]{amsart}

\usepackage[top=3cm, bottom=3cm, left=3cm, right=3cm]{geometry}
\usepackage[pdftex]{graphicx}
\usepackage{tikz}
\usetikzlibrary{shapes, arrows, arrows.meta, decorations.markings}

\usepackage{amsmath}
\usepackage{amsthm}
\usepackage{amssymb}
\usepackage{mathtools}
\usepackage[integrals]{wasysym}
\usepackage{stmaryrd}

\usepackage{bbm}
\usepackage[mathcal]{euscript}

\usepackage[utf8]{inputenc}
\usepackage[T1]{fontenc}
\usepackage[english]{babel}
\usepackage{cite}
\usepackage{lmodern}

\usepackage{calc}
\usepackage[inline]{enumitem}
\usepackage[normalem]{ulem}
\usepackage{url}

\usepackage[hidelinks, bookmarks, bookmarksnumbered, pdfstartview={XYZ null null 1.00}]{hyperref}

\newcommand{\theoname}{Theorem}
\newcommand{\lemmname}{Lemma}
\newcommand{\coroname}{Corollary}
\newcommand{\propname}{Proposition}
\newcommand{\definame}{Definition}

\newcommand{\remkname}{Remark}
\newcommand{\explname}{Example}

\theoremstyle{plain}
\newtheorem{theorem}{\theoname}[section]
\newtheorem{lemma}[theorem]{\lemmname}
\newtheorem{corollary}[theorem]{\coroname}
\newtheorem{proposition}[theorem]{\propname}
\newtheorem*{claim}{Claim}
\theoremstyle{definition}
\newtheorem{definition}[theorem]{\definame}
\newtheorem{remark}[theorem]{\remkname}

\DeclareMathOperator{\Lip}{Lip}

\DeclareMathOperator{\conv}{co}
\DeclareMathOperator{\diverg}{div}

\DeclareMathOperator{\supp}{supp}
\DeclareMathOperator{\Proj}{Pr}

\DeclareMathOperator{\diff}{d\!}

\DeclarePairedDelimiter{\abs}{\lvert}{\rvert}

\newcommand{\suchthat}{\ifnum\currentgrouptype=16 \mathrel{}\middle|\mathrel{}\else\mid\fi}

\DeclareMathOperator{\MFG}{MFG}
\DeclareMathOperator{\OCP}{OCP}
\DeclareMathOperator{\Adm}{Adm}
\DeclareMathOperator{\Opt}{Opt}
\DeclareMathOperator{\In}{In}
\DeclareMathOperator{\OOpt}{\mathbf{Opt}}

\numberwithin{figure}{section}

\makeatletter
\newcommand{\customlabel}[2]{%
   \protected@write \@auxout {}{\string \newlabel {#1}{{#2}{\thepage}{#2}{#1}{}} }%
   \hypertarget{#1}{#2}
}
\makeatother

\begin{document}

\setlength{\parskip}{1pt plus 1pt minus 1pt} 
\newlist{hypotheses}{enumerate}{1}

\setlist[enumerate, 1]{label={\textnormal{(\alph*)}}, ref={(\alph*)}, leftmargin=*}
\setlist[enumerate, 2]{label={\textnormal{(\roman*)}}, ref={(\roman*)}}
\setlist[description, 1]{leftmargin=0pt, itemindent=*}
\setlist[itemize, 1]{label={\textbullet}, leftmargin=0pt, itemindent=*}
\setlist[hypotheses]{label={\textup{(H\arabic*)}}, leftmargin=*}

\title{Nonsmooth mean field games with state constraints}

\author{Saeed Sadeghi Arjmand}
\address{CMLS, \'Ecole Polytechnique, CNRS, Universit\'e Paris-Saclay, 91128, Palaiseau, France, \& Universit\'e Paris-Saclay, CNRS, CentraleSup\'elec, Inria, Laboratoire des signaux et syst\`emes, 91190, Gif-sur-Yvette, France.}
\email{saeed.sadeghi-arjmand@polytechnique.edu}

\author{Guilherme Mazanti}
\address{Universit\'e Paris-Saclay, CNRS, CentraleSup\'elec, Inria, Laboratoire des signaux et syst\`emes, 91190, Gif-sur-Yvette, France.}
\email{guilherme.mazanti@inria.fr}

\begin{abstract}
In this paper, we consider a mean field game model inspired by crowd motion where agents aim to reach a closed set, called target set, in minimal time. Congestion phenomena are modeled through a constraint on the velocity of an agent that depends on the average density of agents around their position. The model is considered in the presence of state constraints: roughly speaking, these constraints may model walls, columns, fences, hedges, or other kinds of obstacles at the boundary of the domain which agents cannot cross. After providing a more detailed description of the model, the paper recalls some previous results on the existence of equilibria for such games and presents the main difficulties that arise due to the presence of state constraints. Our main contribution is to show that equilibria of the game satisfy a system of coupled partial differential equations, known mean field game system, thanks to recent techniques to characterize optimal controls in the presence of state constraints. These techniques not only allow to deal with state constraints but also require very few regularity assumptions on the dynamics of the agents.
\end{abstract}

\keywords{Mean field games, optimal control, minimal time, nonsmooth analysis, state constraints, MFG system, congestion games.}

\subjclass[2020]{49N80, 35Q89, 49K15, 49N60, 35A01}

\maketitle
\hypersetup{pdftitle={Nonsmooth mean field games with state constraints}, pdfauthor={Saeed Sadeghi Arjmand and Guilherme Mazanti}}
\tableofcontents

\newcommand\blfootnote[1]{%
  \begingroup
  \renewcommand\thefootnote{}\footnote{#1}%
  \addtocounter{footnote}{-1}%
  \endgroup
}

\blfootnote{G.M.\ was partially supported by ANR PIA funding: ANR-20-IDEES-0002.}

\section{Introduction}

Mean field games (MFGs for short) were first introduced around 2006 by Lasry and Lions \cite{LasryLionen, LasryLionsfr1, LasryLionsfr2} and independently by Caines, Huang, and Malhamé \cite{HuangMalhame, 2HuangMalhame, Huang2003Individual}, motivated by problems in economics and engineering and based on some previous works on games with infinitely many players, such as those from \cite{Aumann1964Markets, Aumann1974Values, Jovanovic1988Anonymous}. MFGs are differential games with a continuum of players, assumed to be rational, indistinguishable, individually negligible, and influenced only by some ``average'' behavior of other players through a ``mean-field type'' interaction. Since their introduction, MFGs have been studied both in connection with several applications and from a theoretical point of view, in which the main goals are typically proving the existence of equilibria, characterizing such equilibria as solutions to a system of partial differential equations, called \emph{MFG system}, or studying the connections between MFGs and games with a large (but finite) number of symmetric players. We refer to \cite{Achdou2020Mean, Carmona2018ProbabilisticI, CardaliaguetNotes, Carmona2018ProbabilisticII, Cardaliaguet2019Master, Gomes2016Regularity} for more details and further references on mean field games. In this paper, we use the words ``players'' and ``agents'' interchangeably to refer to people taking part in the game.

Most works on mean field games consider either first- or second-order mean field games. First-order MFGs, also known as deterministic mean field games, usually assume that players' dynamics are described by a deterministic control system, and their equilibria are characterized by a system of first-order partial differential equations, whereas second-order MFGs, also known as probabilistic mean field games, consider that players' dynamics are determined by a stochastic control system, typically with an additive Brownian motion modeling a random drift, and their equilibria are typically described by a system of second-order partial differential equations.

This paper considers a class of first-order mean field games inspired by crowd motion in which agents are assumed to remain inside of a specific domain while their goal is to arrive at a given target set in minimal time. In order to model congestion, the speed of each agent is constrained by a function depending on the position of the agent and on the distribution of all agents.

Crowd motion has been the subject of a very large number of works in the literature from different points of view, motivated not only by understanding but also by controlling and optimizing the crowd behavior (see, e.g., \cite{CristaniPiccoli, Gibelli2018Crowd, DhelbingIfarkas, Helbing1995Social, LFHenderson, Maury2019Crowds, Muntean2014Collective, PiccoliBenedettoTosin, Rosini2013Macroscopic, Hughes2002Continuum, Hughes2003Flow}). The natural framework for a mean field game modeling crowd motion is to adopt a \emph{macroscopic} modeling of crowd, i.e., to describe the crowd at a given time $t$ as infinitely many agents represented by a measure $m_t$ on the space of possible positions, which evolves according to some conservation law, typically a continuity equation of the form $\partial_t m + \diverg(m V) = 0$, where $V$ is the velocity field followed by the agents. While most macroscopic crowd motion models consider a given velocity field $V$ constructed from modeling assumptions, the mean field game approach consists instead in considering that each agent will choose their trajectory by solving some optimal control problem, and the velocity field $V$ is a consequence of the optimal choices of the agents.

Up to the authors' knowledge, the first work to be fully dedicated to a mean field game model for crowd motion is \cite{Lachapelle2011Mean}, which proposes an MFG model for a two-population crowd with trajectories perturbed by additive Brownian motion and considers both their stationary distributions and their evolution on a prescribed time interval. Other works have later proposed MFG models for crowd motion taking into account different characteristics, such as \cite{Burger2013Mean}, which considers the fast exit of a crowd and proposes a mean field game model which is studied numerically; \cite{CardaliaguetPierre2}, which is not originally motivated by the modeling of crowd motion but considers a MFG model with a density constraint, which is a natural assumption in some crowd motion models; \cite{BenamouCarlierSantam}, which presents numerical simulations for some variational mean field games related to crowd motion; or also \cite{CristianiGeneralized}, which provides a generalized MFG model for pedestrians with limited predictive abilities.

The present paper considers the MFG model inspired by crowd motion introduced in \cite{Mazanti2019Minimal} (described in details in Section~\ref{SecModel} below), which contains two major differences with respect to previous MFGs for crowd motions and also to most MFG models in general. Firstly, the model from \cite{Mazanti2019Minimal} assumes that each agent solves an optimal control problem with free final time, and actually that the optimization criterion of each agent is to minimize their arrival time at a certain target set, after which they quit the game. This is in contrast with most of the MFG literature, which usually considers either optimization criteria with a given and known finite final time or infinite-horizon optimization criteria. Secondly, congestion is modeled in \cite{Mazanti2019Minimal} by imposing a maximal speed for each agent which depends on their position and the distribution of other agents, while, in most of other MFG works, agents are allowed to choose their speed without constraints, but high speeds and congestion are instead penalized in the cost function. The motivation of \cite{Mazanti2019Minimal} to impose a constraint on the speed of agents is to model high-congestion situations in which an agent may be unable to move faster since other agents in front of them may work as a physical barrier which cannot be crossed by simply paying a higher ``cost''.

The main results of \cite{Mazanti2019Minimal} are the existence of equilibria of the proposed MFG model and the characterization of equilibria through an MFG system. Even though existence of equilibria is proved in \cite{Mazanti2019Minimal} under rather general assumptions, their characterization through an MFG system is only shown in the case where the target set of the agents is the whole boundary of the compact domain in which they evolve, which avoids the presence of state constraints in the minimal-time optimal control problem solved by each agent. This is a very restrictive assumption for a MFG model for crowd motion, since crowds often evolve in domains with boundaries which cannot be crossed by the agents, such as walls, columns, fences, hedges, or other obstacles. However, from a technical point of view, the major difficulty in analyzing optimal control problems with state constraints is that their value functions may fail to be semiconcave (see, e.g., \cite[Example~4.4]{CannarsaPiermarcoCastelpietra}), the latter property being important in the characterization of optimal controls (see, e.g., \cite{CannarsaPiermarcoSinestrari}), which is a key step in obtaining the MFG system in \cite{Mazanti2019Minimal}. Other works, such as \cite{Dweik2020Sharp, DucasseSecond, SadeghiMulti}, have further explored the model from \cite{Mazanti2019Minimal} and related models, but none of those works consider the case of MFGs with state constraints.

Studying optimal control problems and mean field games with state constraints turns out to be a challenging problem due to the possible lack of semiconcavity of the value function. The series of papers \cite{CannarsaPiermarco, Cannarsa2019C11, CannarsaMean} represent an important step in the study of MFGs with state constraints: in those works, the authors prove that, for the optimal control problem they consider, even though the corresponding value function may fail to be semiconcave in the classical sense, i.e., with a \emph{linear} modulus of semiconcavity, the value function is still semiconcave with a \emph{fractional} modulus of semiconcavity, which is sufficient to obtain additional properties of optimal trajectories allowing to characterize optimal controls. The fractional semiconcavity proved in \cite{CannarsaPiermarco} is global in the space variable, which is known not to be the case for some minimal-time optimal control problems. This paper follows a different strategy and proves the required additional properties of optimal trajectories without relying on the semiconcavity of the value function, which has also the additional advantage of requiring fewer regularity properties on the dynamics of the control system satisfied by each agent.

The main contribution of the present paper is thus to characterize optimal controls and deduce the MFG system for the model of \cite{Mazanti2019Minimal} under state constraints, which is an important step towards the study of MFG models for crowd motion with more realistic assumptions. A first step in the characterization of optimal controls is to apply Pontryagin Maximum Principle in order to deduce additional information on optimal trajectories. To do so, we adapt the penalization technique used in \cite{CannarsaCastelpietraCardaliaguet}: we transform the minimal-time optimal control problem with state constraints into a penalized optimal control problem without state constraints and show that, similarly to \cite{CannarsaCastelpietraCardaliaguet}, if the penalization parameter is small enough, optimal trajectories of the penalized problem coincide with optimal trajectories of the original problem (see Theorem~\ref{thm:Opt-eps-equals-Opt}), which allows us to deduce properties of the optimal trajectories of the original optimal control problem by applying Pontryagin Maximum Principle to the penalized problem (see Corollary~\ref{coro:smooth}).

With those properties of optimal trajectories, we then proceed to the study of optimal controls by following the arguments used in the recent paper \cite{SadeghiMulti}: our main results are Theorem~\ref{Thm viscosity boundary cond}, which provides a boundary condition for the value function of the optimal control problem solved by each agent; Theorem~\ref{thm Ut_0,x_0}, which characterizes the optimal control at a given point as the direction of greatest decrease of the value function by adapting the arguments of \cite[Theorem~4.14]{SadeghiMulti} to the case with state constraints; and Theorem~\ref{Thm MFG system}, which shows that equilibria of our MFG model satisfy the MFG system.

The present paper is an extended version of \cite{Sadeghi2021Characterization}, which announced most of the results from this paper and provided some details of the strategy of the proofs of the main results. With respect to that reference, this paper works under weaker assumptions on the dynamics of the control system (compare \ref{HypoOCP-k-Bound}--\ref{HypoMFG-K-Lip} below with \cite[(H2)--(H3)]{Sadeghi2021Characterization}), which renders some proofs more technically involved and require different proof strategies at some points, such as in Lemma~\ref{lemm-varphi-Lipschitz} and for most of the results of Section~\ref{sec:penalized}. This paper also provides a more detailed introduction to the subject, details proofs that had been omitted or only sketched in \cite{Sadeghi2021Characterization} (and in particular the results in Section~\ref{sec:charact-opt-controls}), and provides additional remarks and comments on the results.

Our paper is organized as follows. Some notations and standard definitions are provided in Section~\ref{SecModel}, together with the precise description of the optimal control problem and the mean field game model considered in this paper, and the list of hypotheses used here. Section~\ref{sec:prelim} presents preliminary results on the optimal control problem and the mean field game, most of which are either easy to prove or already present in the literature. The major new contributions in Section~\ref{sec:prelim} are the proof of Lipschitz continuity of the value function of the optimal control problem under the weaker assumptions \ref{HypoOCP-k-Bound} and \ref{HypoOCP-k-Lip} as well as the alternative characterization of equilibria from Proposition~\ref{prop:equilibrium-support}, which also holds for other mean field game models (see, e.g., \cite[Lemma~3.4]{SantambrogioCuckerSmale}).

The main results of this paper are provided in Sections~\ref{More on OPT} and \ref{sec MFG system}. Section~\ref{sec:penalized} uses the strategy of \cite{CannarsaCastelpietraCardaliaguet} to study an optimal control problem with state constraints through a penalized optimal control problem without state constraints, adapting the techniques of that reference to our setting in order to show that optimal trajectories of the original problem coincide with that of the penalized problem if the penalization parameter is small enough (Theorem~\ref{thm:Opt-eps-equals-Opt}). Section~\ref{sec:boundary-varphi} provides a boundary condition for the value function of the optimal control problem in the part of the boundary which is not in the target set. We then provide, in Section~\ref{sec:charact-opt-controls}, the characterization of optimal controls as directions of maximal descent of the value function, and we use this characterization in Section~\ref{sec:normalized-grad} to introduce the notion of normalized gradient and provide its main properties. The main result concerning the mean field game model of interest is the fact that its equilibria satisfy a system of PDEs, shown in Section~\ref{sec MFG system}.

\section{Notations and definitions}
\label{SecModel}

\subsection{General notations}

In this paper, $\mathbbm N$ denotes the set of positive integers, $d$ is a fixed positive integer, the sets of nonnegative and positive real numbers are denoted respectively by $\mathbbm R_+$ and $\mathbbm R^\ast_+$, $\mathbbm R^d$ is endowed with the usual Euclidean norm $\abs{\cdot}$, and the unit sphere is denoted by $\mathbbm S^{d-1}$. For $A \subset \mathbbm R^d$, $\bar A$ denotes its closure, $\partial A$ denotes its boundary, and $\conv A$ denotes its convex hull. For any $x \in \mathbbm R^d$ and $r \geq 0$, $B(x, r)$ (resp., $\bar B(x, r)$) denotes\footnote{In accordance with the previous notation, $\bar B(x, r)$ is the closure of $B(x, r)$ whenever $r > 0$. However, our notation is slightly ambiguous for $r = 0$: $B(x, 0) = \varnothing$, but $\bar B(x, 0) = \{x\}$ is not the closure of $B(x, 0)$. All along the paper, we use the convention that, when referring to balls, the notation $\bar B(x, r)$ is the closed ball centered at $x$ and with radius $r$. The same convention applies to the notation $\bar B_r$, and we remark that $\bar B$ is unambiguous.} the open (resp., closed) ball centered at $x$ and with radius $r$. We denote this ball simply by $B_r$ (resp., $\bar B_r$) if $x = 0$ and by $B$ (resp., $\bar B$) if $x = 0$ and $r = 1$, i.e., $B$ (resp., $\bar B$) is the open (resp., closed) unit ball centered at the origin.

For two sets $A$ and $B$, a set-valued map from $A$ to $B$ is a map $F$ that associates, with each $x \in A$, a (possibly empty) set $F(x) \subset B$. We use the notation $F: A \rightrightarrows B$ to indicate that $F$ is a set-valued map from $A$ to $B$.

Given two metric spaces $X$ and $Y$ and a constant $M > 0$, $\mathbf C(X; Y)$, $\Lip(X; Y)$, and $\Lip_M(X; Y)$ denote, respectively, the set of all continuous functions from $X$ to $Y$, the set of all Lipschitz continuous functions from $X$ to $Y$, and the subset of $\Lip(X; Y)$ containing only those functions whose Lipschitz constant is at most $M$. We will often use these spaces to represent time-dependent functions defined for nonnegative times, in which case $X = \mathbbm R_+$ and, for simplicity, we omit $X$ from the previous notations, writing simply $\mathbf C(Y)$, $\Lip(Y)$, and $\Lip_M(Y)$, respectively. When $X \subset \mathbbm R^k$ and $Y \subset \mathbbm R^m$ for some positive integers $k$ and $m$, we also consider the set $\mathbf C^{1}(X; Y)$ of continuously differentiable functions, the set $\mathbf C^{1, 1}(X; Y)$ of differentiable functions with a Lipschitz continuous differential, and the set $\mathbf C^\infty_c(X; Y)$ of infinitely differentiable functions with compact support on $X$. If $f: X \to Y$ and $Z \subset X$, we write $f \in \mathbf C(Z; Y)$ to denote that the restriction of $f$ to $Z$ is continuous (even though $f$ itself may fail to be continuous in the boundary of $Z$), with similar notations for the other functional spaces defined above.

For compact $A \subset \mathbbm R^d$, the space $\mathbf{C}(A)$ is assumed to be endowed with the topology of uniform convergence on compact sets, with respect to which $\mathbf{C}(A)$ is a Polish space (see, e.g., \cite[Chapter~X]{Bourbaki2007Topologie}). For $t\in \mathbbm R_+$, we denote by $e_t:\mathbf{C}(A) \to A$ the evaluation map at time $t$, defined by $e_t(\gamma) = \gamma(t)$ for every $\gamma \in \mathbf C(A)$.

Recall that, for two metric spaces $X$ and $Y$ endowed with their Borel $\sigma$-algebras and a Borel-measurable map $f:X \to Y$, the pushforward of a measure $\mu$ on $X$ through $f$ is the measure $f_{\#} \mu$ on $Y$ defined by $f_{\#} \mu (B) = \mu(f^{-1}(B))$ for every Borel subset $B$ of $Y$. Given a Polish space $X$, the set of all Borel probability measures on $X$ is denoted by $\mathcal{P}(X)$ and is endowed with the topology of weak convergence of measures (which is also sometimes referred to as narrow convergence). The support of a measure $\mu \in \mathcal P(X)$ is denoted by $\supp(\mu)$, and is defined as the set of all points $x \in X$ such that $\mu(N_x) > 0$ for every open neighborhood $N_x$ of $x$. When $X$ is endowed with a complete metric $\mathbf d$ with respect to which $X$ is bounded, we endow $\mathcal P(X)$ with the Wasserstein distance $\mathbf W_1$, defined for $\mu, \nu \in \mathcal P(X)$ by
\[
\mathbf{W}_1(\mu,\nu) = \inf \left \{ \int_{X \times X} \mathbf d(x, y) \diff \lambda(x,y) \suchthat \lambda \in \Pi(\mu,\nu) \right\},
\]
where $\Pi(\mu,\nu)=\{\lambda \in \mathcal{P}(X \times X) \suchthat \pi_{1\#} \lambda=\mu,\; \pi_{2\#} \lambda=\nu \}$ and $\pi_1,\, \pi_2 : X \times X \to X$ denote the canonical projections on to the first and second factors of the product $X \times X$, respectively. 
Recall (see, e.g., \cite[Chapter~7]{Ambrosio1} and \cite[Chapter~5]{FilippoSantam}) that $\mathbf W_1$ is compatible with the topology of weak convergence in $\mathcal P(X)$ and that it admits the dual formulation
\[
\mathbf{W}_1(\mu,\nu)=\sup\left \{\int_{X} \Phi(x) \diff \,(\mu-\nu) \suchthat \Phi \in  \Lip_1(X; \mathbbm R) \right \}.
\]

We shall also need in this paper the notion of signed distance to the boundary of a set, whose definition we now provide, together with the classical definition of distance from a point to a set.

\begin{definition}
Let $A\subset \mathbbm R^d$. If $A \neq \varnothing$, we denote by $d_A$ the Euclidean distance to $A$, defined by $d_A(x) = \inf_{y\in A} \abs*{x-y}$. If $A\neq \varnothing$ and $A\neq \mathbbm R^d$, the signed distance to the boundary of $A$, $d^{\pm}_{\partial A}:\mathbbm R^d\to \mathbbm R$, is defined by
\[
d^{\pm}_{\partial A}(x) = d_A(x)-d_{\mathbbm R^d\setminus A}(x).
\]
\end{definition}

Finally, we will need, in Section~\ref{More on OPT}, the notion of normal cones of nonsmooth sets. We recall here the definitions provided in \cite[Section~4.2]{Vinter2010Optimal}.

\begin{definition}
\label{defi:normal-cones}
Let $k$ be a positive integer, $C \subset \mathbbm R^k$ be a nonempty closed set, and $x \in C$.
\begin{enumerate}
\item\label{defi:proximal-normal-cone} The \emph{proximal normal cone} to $C$ at $x$ is the set $N_C^{\mathrm{P}}(x)$ defined by
\[
N_C^{\mathrm P}(x) = \{p \in \mathbbm R^k \suchthat \exists M > 0 \text{ such that } p \cdot (y - x) \leq M \abs{y - x}^2 \text{ for every } y \in C\}.
\]
\item\label{defi:limiting-normal-cone} The \emph{limiting normal cone} to $C$ at $x$ is the set $N_C(x)$ of vectors $p \in \mathbbm R^k$ for which there exist a sequence $(x_n)_{n \in \mathbbm N}$ in $C$ and a sequence $(p_n)_{n \in \mathbbm N}$ in $\mathbbm R^k$ such that $x_n \to x$ and $p_n \to p$ as $n \to +\infty$ and $p_n \in N_C^{\mathrm P}(x_n)$ for every $n \in \mathbbm N$.
\end{enumerate}
\end{definition}

\subsection{The minimal-time optimal control problem}
\label{sec:OCP}

Before introducing the minimal-time mean field game of interest to this paper, let us first consider an auxiliary minimal-time optimal control problem, which is a classical kind of problem in control theory. Let $\Omega \subset \mathbbm R^d$ be a nonempty open bounded set and $\Gamma \subset \bar\Omega$ be a nonempty closed set. We consider the control system
\begin{equation}\label{General control sys}
    \left \{
    \begin{aligned}
    \Dot{\gamma}(t)&=k(t,\gamma(t)) u(t),\qquad u(t)\in \bar{B},  \\
    \gamma(0)&=x,    
    \end{aligned} \right.
\end{equation}
where $x \in \bar\Omega$, $k: \mathbbm R_+ \times \bar\Omega \to \mathbbm R_+$, and, for $t \geq 0$, $\gamma(t) \in \bar\Omega$ is the state and $u(t) \in \bar B$ is the control. The function $k$ describes the dynamics of the system and, due to the constraint that $u(t) \in \bar B$, $k(t, x)$ can be interpreted as the maximal speed at which an agent at position $x$ at time $t$ can move. The fact of requiring that $\gamma(t) \in \bar\Omega$ for every time $t \geq 0$ is a constraint we impose in the state of \eqref{General control sys}.

We are interested in this paper in the optimal control problem consisting in minimizing the time a trajectory of \eqref{General control sys} takes to reach the set $\Gamma$, called the \emph{target set}. This optimal control problem, to which we refer in the sequel as $\OCP(k)$ making explicit its dependence on the dynamics $k$, is made more precise in the following definition.

\begin{definition}
\label{def:OCP}
Let $\Omega \subset \mathbbm R^d$ be a nonempty open bounded set, $\Gamma \subset \bar\Omega$ be a nonempty closed set, and $k: \mathbbm R_+ \times \bar{\Omega} \to \mathbbm R_+$.
\begin{enumerate}[label={(\alph*)}]
    \item A curve $\gamma \in \Lip(\bar{\Omega})$ is said to be \emph{admissible} for $\OCP(k)$ if there exists a measurable function $u: \mathbbm R_+ \to \bar B$, called the \emph{control} associated with $\gamma$, such that the first equation of \eqref{General control sys} is satisfied for almost every $t \in \mathbbm R_+$. The set of all admissible curves for $\OCP(k)$ is denoted by $\Adm(k)$.

    \item\label{def:OCP-tau} Let $t_0 \in \mathbbm R_+$. The \emph{first exit time} after $t_0$ of a curve $\gamma \in \Lip(\bar{\Omega})$ is the number $\tau(t_0, \gamma)$ defined by
    \[
    \tau(t_0, \gamma) = \inf \{t \ge 0 \suchthat \gamma(t + t_0) \in \Gamma\},
    \]
		with the convention that $\tau(t_0, \gamma) = +\infty$ if $\gamma(t + t_0) \not \in \Gamma$ for every $t \geq 0$.
				
    \item\label{item:def:optimal-trajectory} Let $t_0 \in \mathbbm R_+$ and $x_0 \in \bar{\Omega}$. A curve $\gamma \in \Lip(\bar{\Omega})$ is said to be an \emph{optimal trajectory} for $\OCP(k)$ from $(t_0, x_0)$ if $\gamma \in \Adm(k)$, $\gamma(t) = x_0$ for every $t \in [0, t_0]$,
    \begin{equation}\label{min ex time}
        \tau(t_0, \gamma) = \inf_{\substack{\beta \in \Adm(k) \\ \beta(t_0) = x_0}} \tau(t_0, \beta),
    \end{equation}
		and $\gamma(t) = \gamma(t_0 + \tau(t_0, \gamma))$ for every $t \in [t_0 + \tau(t_0, \gamma), +\infty)$. In this case, the associated control $u$ of $\gamma$ is called an \emph{optimal control} associated with $\gamma$. The set of all optimal trajectories for $\OCP(k)$ from $(t_0, x_0)$ is denoted by $\Opt(k, t_0, x_0)$.
		
		\item\label{def:value-function} The \emph{value function} of the optimal control problem $\OCP(k)$ is the function $\varphi: \mathbbm R_+ \times \bar{\Omega} \to \mathbbm R_+$ defined for $(t_0, x_0) \in \mathbbm R_+ \times \bar{\Omega}$ by
\begin{equation}\label{value function}
\varphi(t_0, x_0) = \inf_{\substack{\gamma \in \Adm(k) \\ \gamma(t_0) = x_0}} \tau(t_0,\gamma).
\end{equation}
\end{enumerate}
\end{definition}

We highlight the fact that, for a curve $\gamma$ to be admissible, it must not only satisfy the first equation of \eqref{General control sys} for some control $u$ and almost every time $t \geq 0$, but it must also remain inside $\bar\Omega$ for all times.

\begin{remark}
Since we are considering in Definition~\ref{def:OCP}\ref{item:def:optimal-trajectory} optimal trajectories $\gamma$ starting at time $t_0$ from $x_0$ and minimizing the time to reach $\Gamma$, these trajectories could have been defined only in the interval $[t_0, t_0 + \tau(t_0, \gamma)]$. However, in order to simplify the notations and deal in this paper only with trajectories and controls defined on $\mathbbm R_+$, we extend an optimal trajectory to $\mathbbm R_+$ by requiring it to be constant on $[0, t_0]$ and on $[t_0 + \tau(t_0, \gamma), +\infty)$, as done in Definition~\ref{def:OCP}\ref{item:def:optimal-trajectory}.
\end{remark}

\subsection{The minimal-time mean field game and its equilibria}

As in Section~\ref{sec:OCP}, we fix a nonempty open bounded set $\Omega \subset \mathbbm R^d$ and a nonempty closed set $\Gamma \subset \bar\Omega$. We consider a minimal-time mean field game in which a population of agents evolves on $\bar\Omega$ and the goal of each agent is to reach the target set $\Gamma$ in minimal time. The population is described by a time-dependent probability measure $m_t \in \mathcal P(\bar\Omega)$ for $t \geq 0$, and $m_0$ is assumed to be known. The trajectory $\gamma$ of an agent starting its movement at a position $x \in \bar\Omega$ is assumed to satisfy the control system
\begin{equation}
\label{Control sys MFG}
    \left \{
    \begin{aligned}
    \Dot{\gamma}(t)&=K(m_t,\gamma(t)) u(t),\qquad u(t)\in \bar{B},  \\
    \gamma(0)&=x,    
    \end{aligned} \right.
\end{equation}
where $K: \mathcal P(\bar\Omega) \times \bar\Omega \to \mathbbm R_+$ and $u: \mathbbm R_+ \to \bar B$ is the control of the agent. Note that \eqref{Control sys MFG} corresponds to \eqref{General control sys} with $k$ defined by $k(t, x) = K(m_t, x)$ for $t \geq 0$ and $x \in \bar{\Omega}$. Each agent is assumed to choose their control in order to solve the optimal control problem $\OCP(k)$ with $k$ defined from $K$ as before. This mean field game is denoted in the sequel by $\MFG(K)$.

The function $K$ models interactions between agents and states that the maximal speed at which an agent at position $x$ at time $t$ can move depends on the position $x$ itself and on the distribution of all agents at time $t$, $m_t$. This function can be used to model congestion phenomena in crowd motion by choosing $K(m_t, x)$ to be small when $m_t$ is ``large'' around $x$, which means that it is harder for agents to move on more crowded regions. For instance, $K$ can be chosen as
\begin{equation}
\label{eq:K-example}
K(\mu, x) = g\left(\int_{\bar{\Omega}} \chi(x-y)\eta(y)\diff \mu(y)\right),
\end{equation}
where $\chi: \mathbbm R^d \to \mathbbm R_+$ is a convolution kernel, $\eta: \bar\Omega \to \mathbbm R_+$ is a weight function in the space $\bar\Omega$, and $g: \mathbbm R_+ \to \mathbbm R_+$ is a decreasing function. The function $\chi$ may represent, for instance, the region around an agent at which they look in order to evaluate local congestion, while $\eta$ may be a function that is larger in regions difficult to move, such as regions with obstacles or other kinds of difficult terrain, and which can also be used to discount people who already reached the target set, as it was done in \cite{Mazanti2019Minimal, Dweik2020Sharp}. Note, however, that we do \emph{not} assume this specific form for $K$ in this paper.

Note that, since $K$ depends on $m_t$ for all $t \geq 0$, the optimal trajectories taken by the agents depend on $m_t$. On the other hand, $m_t$ itself describes the evolution of the agents, and hence is determined by their choices of trajectories. We are interested in this paper in \emph{equilibrium} situations, in which, roughly speaking, starting from time evolution of the distribution of agents $m:\mathbbm R_+ \to \mathcal{P}(\bar\Omega)$, the optimal trajectories chosen by agents induce an evolution of the initial distribution of agents $m_0$ that is precisely given by $t \mapsto m_t$.

We provide a more precise notion of equilibrium in what is known as the Lagrangian framework, in which, instead of describing the evolution of agents as a time-dependent measure $m: \mathbbm R_+ \to \mathcal P(\bar\Omega)$, we rely instead on a measure $Q$ on the set of all possible trajectories $\mathcal P(\mathbf C(\bar\Omega))$. Note that, given a measure $Q \in \mathcal P(\mathbf C(\bar\Omega))$, one can obtain the associated time-dependent measure $m$ by setting $m_t = e_{t\#} Q$ for $t \geq 0$. The Lagrangian approach is a classical approach in optimal transport problems (see, e.g., \cite{Ambrosio1, FilippoSantam}) which has been used to define equilibria of first-order mean field games in some recent works, such as \cite{BenamouCarlierSantam, CannarsaPiermarco, CardaliaguetPierre, CardaliaguetPierre2, Mazanti2019Minimal, Dweik2020Sharp, SadeghiMulti, Cannarsa2020Mild, Fischer2020Asymptotic}.

\begin{definition}
\label{defi:equilibrium}
Let $\Omega \subset \mathbbm R^d$ be a nonempty open bounded set, $\Gamma \subset \bar\Omega$ be a nonempty closed set, $K: \mathcal P(\bar\Omega) \times \bar{\Omega} \to \mathbbm R_+$, and $m_0 \in \mathcal{P}(\bar{\Omega})$. A measure $Q\in \mathcal{P}(\mathbf{C}(\bar{\Omega}))$ is called a \emph{Lagrangian equilibrium} of $\MFG(K)$ with initial condition $m_0$ if $e_{0\#}Q = m_0$ and $Q$-almost every $\gamma \in \mathbf{C}(\bar{\Omega})$ satisfies $\gamma \in \Opt(k_Q, 0, \gamma(0))$, where $k_Q:\mathbbm R_+ \times \bar{\Omega} \to \mathbbm R_+$ is defined for $t \geq 0$ and $x \in \bar\Omega$ by $k_Q(t,x)=K(e_{t\#}Q,x)$.
\end{definition}

In the sequel of the paper, we refer to Lagrangian equilibria simply as equilibria.

\begin{remark}
Another classical way to describe the evolution of agents in a mean field game, which dates back to \cite{HuangMalhame} and has been used and developed in several other references such as \cite{Gomes2016Extended, Lacker2015Mean, Carmona2015Probabilistic}, is to fix a probability space $(\mathsf Z, \mathcal Z, \mathbbm P)$ and describe the motion of agents through a time-dependent random variable in $\mathsf Z$, i.e., through a function $X$ defined on $\mathbbm R_+$ and such that, for every $t \geq 0$, $X(t): \mathsf Z \to \bar\Omega$ is measurable. One can then retrieve the time-dependent measure $m_t \in \mathcal P(\bar\Omega)$ as the law of $X(t)$, i.e., $m_t = X(t)_{\#} \mathbbm P$, but $X$ contains more information than $m$, since $X$ also carries information on the correlation of the distributions of agents at different times, for instance. This formulation using random variables has the additional advantage of being also adapted to study of other kinds of mean field games, such as second-order mean field games or mean field games of controls.

In addition, if\footnote{Here, $X(t, z)$ is a simplified notation for $X(t)(z)$.} $\mathbbm R_+ \ni t \mapsto X(t, z) \in \bar\Omega$ is continuous for almost every $z \in \mathsf Z$ and the function $\Xi: \mathsf Z \to \mathbf C(\bar\Omega)$ defined by setting $\Xi(z) = X(\cdot, z)$ for a.e.\ $z \in \mathsf Z$ is measurable, then the probability measure $Q \in \mathcal P(\mathbf C(\bar\Omega))$ describing the evolution of agents from a Lagrangian perspective can be retrieved as $Q = \Xi_{\#} \mathbbm P$.

Conversely, given a measure $Q \in \mathcal P(\mathbf C(\bar\Omega))$ describing the evolution of agents, one can consider the probability space $(\mathbf C(\bar\Omega), \mathcal B, Q)$, where $\mathcal B$ is the Borel $\sigma$-algebra of $\mathbf C(\bar\Omega)$, and in this case the evolution of agents can be described by the time-dependent random variable $X$ defined by $X(t) = e_t$ for $t \geq 0$.
\end{remark}

\subsection{Hypotheses}

Along the paper, we will need several assumptions on $\Omega$, $\Gamma$, $K$, and $k$, which we collect in this subsection. We start with the following assumptions on the sets $\Omega$ and $\Gamma$.
\begin{hypotheses}
\item\label{HypoOmega} The set $\Omega \subset \mathbbm R^d$ is nonempty, open, bounded, and connected.
\item\label{HypoGamma} The set $\Gamma \subset \bar\Omega$ is nonempty and closed.
\item\label{HypoOmegaC11} The boundary $\partial\Omega$ is a compact $\mathbf{C}^{1,1}$ submanifold of $\mathbbm R^d$ of dimension $d-1$.
\end{hypotheses}
\newcommand{\HypoOmega}{\ref{HypoOmega}--\ref{HypoOmegaC11}}

Whenever we assume that \ref{HypoOmegaC11} holds, we use $\mathbf n(x)$ to denote the outward unit normal vector to $\bar\Omega$ at the point $x \in \partial\Omega$.

When studying the optimal control problem $\OCP(k)$, we shall need the following assumptions on the function $k$.
\begin{hypotheses}[resume]
\item\label{HypoOCP-k-Bound} The function $k: \mathbbm R_+ \times \bar{\Omega} \to \mathbbm R_+$ is continuous and there exist positive constants $K_{\min}$ and $K_{\max}$ such that $k(t, x) \in [K_{\min}, K_{\max}]$ for every $(t, x) \in \mathbbm R_+ \times\bar{\Omega}$.
\item\label{HypoOCP-k-Lip} The function $k$ is Lipschitz continuous with respect to its second variable, uniformly with respect to the first variable, i.e., there exists $L > 0$ such that, for every $t \in \mathbbm R_+$ and $x_1, x_2 \in \bar{\Omega}$, we have
\[
\abs{k(t, x_1) - k(t, x_2)} \leq L \abs{x_1 - x_2}.
\]

\end{hypotheses}
The counterpart of Hypotheses 
\ref{HypoOCP-k-Bound} and \ref{HypoOCP-k-Lip}
concerning the function $K$ from the mean field game $\MFG(K)$ are the following.
\begin{hypotheses}[resume]
\item \label{HypoMFG-K-Bound} The function $K: \mathcal{P}(\bar{\Omega}) \times \bar{\Omega} \to \mathbbm R_+$ is continuous and there exist positive constants $K_{\min}$ and $K_{\max}$ such that $K(\mu, x) \in [K_{\min}, K_{\max}]$ for every $(\mu, x) \in \mathcal{P}(\bar{\Omega}) \times\bar{\Omega}$.
\item\label{HypoMFG-K-Lip} The functions $K$ is Lipschitz continuous with respect to its second variable, uniformly with respect to the first variable, i.e., there exists $L > 0$ such that, for every $\mu \in \mathcal P(\bar{\Omega})$ and $x_1, x_2 \in \bar{\Omega}$, we have
\[
\abs{K(\mu, x_1) - K(\mu, x_2)} \leq L \abs{x_1 - x_2}.
\] 
\end{hypotheses}
Note that 
\ref{HypoMFG-K-Bound} and \ref{HypoMFG-K-Lip}
are satisfied in the particular case where $K$ is chosen as in \eqref{eq:K-example} if $g$ is Lipschitz continuous, upper bounded, and lower bounded by a positive constant, $\chi$ is Lipschitz continuous, and $\eta$ is continuous.

\section{Preliminary results}
\label{sec:prelim}

This section presents results on $\OCP(k)$ and $\MFG(K)$ that will be useful in the sequel, most of which are either present in other references or easy to prove using classical techniques. It turns out that the results we will present in this section require fewer regularity assumptions on $\Omega$ than \ref{HypoOmega} and \ref{HypoOmegaC11}; namely, we replace here \ref{HypoOmegaC11} by the following assumption.
\begin{enumerate}[labelwidth=\widthof{\normalsize (H3$^\prime$)\ }, leftmargin=!]
\item[\customlabel{HypoOmegaGeodesic}{\normalfont{(H3$^\prime$)}}] There exists $D > 0$ such that, for every $x, y \in \bar\Omega$, there exists a curve $\gamma$ included in $\bar\Omega$ connecting $x$ to $y$ and of length at most $D \abs{x - y}$.
\end{enumerate}
Hypothesis \ref{HypoOmegaGeodesic} means that the geodesic distance in $\bar\Omega$ is equivalent to the usual Euclidean distance, and it holds in particular when \ref{HypoOmega} and \ref{HypoOmegaC11} are satisfied.

The first result we present concerns three elementary properties of $\OCP(k)$: existence of optimal trajectories, boundedness of the value function, and the dynamic programming principle. The proof is omitted, since all properties are either easy to prove or classical. Indeed, existence of optimal trajectories can be proved easily using compactness of a minimizing sequence (see \cite[Theorem~8.1.4]{CannarsaPiermarcoSinestrari} for a proof in the autonomous case, i.e., when $k: \mathbbm R_+ \times \bar\Omega \to \mathbbm R_+$ does not depend on its first variable). Boundedness of the value function follows from the fact that, using \ref{HypoOmegaGeodesic}, one can easily construct, for every point $x \in \bar\Omega$, an admissible trajectory connecting it to a point in $\Gamma$ with constant speed $K_{\min}$ and that arrives in $\Gamma$ in time at most $\frac{D d_\Gamma(x)}{K_{\min}}$, and the continuous function $d_\Gamma$ is bounded in the compact set $\bar\Omega$. Finally, the proof of the dynamic programming principle is classical: the autonomous case can be found, for instance, in \cite[Proposition~2.1]{BardiDolcetta} and \cite[(8.4)]{CannarsaPiermarcoSinestrari}, and the corresponding proofs can be easily adapted to our nonautonomous setting.

\begin{proposition}
\label{PropOCP-1}
Consider $\OCP(k)$ under hypotheses \ref{HypoOmega}, \ref{HypoGamma}, \ref{HypoOmegaGeodesic}, \ref{HypoOCP-k-Bound}, and \ref{HypoOCP-k-Lip}.
\begin{enumerate}
\item\label{PropOCP-ExistsOptimal} For every $(t_0, x_0) \in \mathbbm R_+ \times \bar\Omega$, there exists an optimal trajectory $\gamma$ for $\OCP(k)$ from $(t_0, x_0)$.
\item\label{PropOCP-VarphiBounded} There exists $T > 0$ such that, for every $(t_0, x_0) \in \mathbbm R_+ \times \bar{\Omega}$, the value function $\varphi$ satisfies $\varphi(t_0, x_0) \leq T$.
\item\label{PropOCP-DPP} For every $(t_0, x_0) \in \mathbbm R_+ \times \bar\Omega$ and $\gamma \in \Adm(k)$ such that $\gamma(t_0) = x_0$, we have, for every $h \geq 0$,
\begin{equation}
\label{eq:DPP}
\varphi(t_0 + h, \gamma(t_0 + h)) + h \geq \varphi(t_0, x_0),
\end{equation}
with equality if $\gamma \in \Opt(k, t_0, x_0)$ and $h \in [0, \tau(t_0, \gamma)]$. Conversely, if $\gamma \in \Adm(k)$ satisfies $\gamma(t_0) = x_0$, $\gamma$ is constant on $[0, t_0]$ and on $[t_0 + \tau(t_0, \gamma), +\infty)$, and equality holds in \eqref{eq:DPP} for every $h \in [0, \tau(t_0, \gamma)]$, then $\gamma \in \Opt(k, t_0, x_0)$.
\end{enumerate}
\end{proposition}

Our next result deals with Lipschitz continuity of the value function $\varphi$. This kind of result is classical for optimal control problems with free final time (see, e.g., \cite[Proposition~8.2.5]{CannarsaPiermarcoSinestrari} for a proof in the autonomous case), and a complete proof for the nonautonomous optimal control problem $\OCP(k)$ was given in \cite[Propositions~4.2 and 4.3]{Mazanti2019Minimal}. However, that reference uses the stronger assumption that $k \in \Lip(\mathbbm R_+ \times \bar\Omega; \mathbbm R_+)$. When the optimal control problem does not have state constraints, this assumption can be relaxed to \ref{HypoOCP-k-Lip} by first showing Lipschitz continuity of $\varphi$ with respect to $x$, which can be done by adapting the classical proof of \cite[Proposition~8.2.5]{CannarsaPiermarcoSinestrari}, and then using the dynamic programming principle to deduce Lipschitz continuity also with respect to $t$. This strategy was described in \cite[Proposition~3.8]{Dweik2020Sharp} and carried out in details in \cite[Lemma~4.7 and Proposition~4.8]{SadeghiMulti}, however those proofs rely on the absence of state constraints and cannot be easily generalized to optimal control problems with state constraints. For that reason, we present here a new proof, which is inspired by that of \cite[Propositions~4.2 and 4.3]{Mazanti2019Minimal} but uses a technique introduced in the proof of \cite[Proposition~3.9]{Dweik2020Sharp} in order to replace the assumption $k \in \Lip(\mathbbm R_+ \times \bar\Omega; \mathbbm R_+)$ by the weaker assumption \ref{HypoOCP-k-Lip}. As a first step, we prove Lipschitz continuity of $\varphi$ in space for fixed time.

\begin{lemma}\label{lemm-varphi-Lipschitz}
Consider the optimal control problem $\OCP(k)$ and its value function $\varphi$ and assume that \ref{HypoOmega}, \ref{HypoGamma}, \ref{HypoOmegaGeodesic}, \ref{HypoOCP-k-Bound}, and \ref{HypoOCP-k-Lip} are satisfied. Then there exists $C > 0$ such that, for every $t_0 \in \mathbbm R_+$ and $x_0,\, x_1 \in \bar{\Omega}$, we have
\begin{equation}
\label{eq:varphi-Lipschitz-in-x}
\abs{\varphi(t_0, x_0) - \varphi(t_0, x_1)} \leq C \abs{x_0 - x_1}.
\end{equation}
\end{lemma}

\begin{proof}
Let $D > 0$ be as in \ref{HypoOmegaGeodesic}, $L > 0$ be as in \ref{HypoOCP-k-Lip}, and $T > 0$ be as in the statement of Proposition~\ref{PropOCP-1}\ref{PropOCP-VarphiBounded}. It suffices to show that there exists $C > 0$ such that, for every $t_0 \in \mathbbm R_+$ and $x_0, x_1 \in \bar\Omega$, we have
\begin{equation}
\label{eq:varphi-one-side-Lipschitz}
\varphi(t_0, x_1) - \varphi(t_0, x_0) \leq C \abs{x_0 - x_1},
\end{equation}
since, in this case, \eqref{eq:varphi-Lipschitz-in-x} can be deduced by exchanging the role of $x_0$ and $x_1$.

Let $t_0 \in \mathbbm R_+$ and $x_0, x_1 \in \bar{\Omega}$. Let $\gamma_0 \in \Opt(k, t_0, x_0)$ and denote by $u_0$ the corresponding optimal control, i.e., $\dot\gamma_0(t) = k(t, \gamma_0(t)) u_0(t)$ for a.e.\ $t \in \mathbbm R_+$. Let $t_0^\ast = t_0 + \varphi(t_0, x_0)$ be the time at which $\gamma_0$ arrives at the target set $\Gamma$.

Let us first describe informally the idea of the proof. We will construct an admissible trajectory $\gamma_1 \in \Lip(\bar\Omega)$ which remains constant at $x_1$ in the time interval $[0, t_0]$, then moves from $x_1$ to $x_0$ in a time interval $[t_0, t_1]$, and then follows the same path of $\gamma_0$, but with a change in the time scale since it starts from $x_0$ at time $t_1 > t_0$. This trajectory will then arrive at $x_0^\ast \in \Gamma$, and we will prove that its arrival time at $\Gamma$ satisfies
\begin{equation}
\label{eq:Lip-estimate-to-be-shown}
\tau(t_0, \gamma_1) \leq \varphi(t_0, x_0) + C \abs{x_0 - x_1},
\end{equation}
which yields the conclusion since $\varphi(t_0, x_1) \leq \tau(t_0, \gamma_1)$. The difficult part of the proof is to perform a suitable change in time scale guaranteeing both that $\gamma_1$ is admissible and that its arrival time at $\Gamma$ satisfies the above inequality.

Applying \ref{HypoOmegaGeodesic} and renormalizing the speed of the curve whose existence is asserted in that hypothesis, we obtain the existence of $t_1 \geq t_0$ and a Lipschitz continuous curve $\sigma: [t_0, t_1] \to \bar\Omega$ such that $\sigma(t_0) = x_1$, $\sigma(t_1) = x_0$, $\abs{\dot\sigma(t)} = K_{\min}$ for almost every $t \in [t_0, t_1]$, and $t_1 - t_0 \leq \frac{D \abs{x_1 - x_0}}{K_{\min}}$.

Let us now define $\phi: [t_1, +\infty) \to [t_0, +\infty)$ as a solution of the problem
\[
\left\{
\begin{aligned}
\dot\phi(t) & = \frac{k(t, \gamma_0(\phi(t)))}{k(\phi(t), \gamma_0(\phi(t)))} & \quad & \text{ for } t \geq t_1, \\
\phi(t_1) & = t_0.
\end{aligned}
\right.
\]
Note that, since $(t, s) \mapsto \frac{k(t, \gamma_0(s))}{k(s, \gamma_0(s))}$ is continuous (but not necessarily Lipschitz continuous in its second argument), a solution $\phi$ to the above problem exists and is of class $\mathbf C^1$ (but it may not be unique). Moreover, $\dot\phi(t) \in \left[\frac{K_{\min}}{K_{\max}}, \frac{K_{\max}}{K_{\min}}\right]$ for every $t \in [t_1, +\infty)$, which implies that $\phi: [t_1, +\infty) \to [t_0, +\infty)$ is increasing and surjective, and hence invertible, and both $\phi$ and $\phi^{-1}$ are Lispchitz continuous, with Lipschitz constant $\frac{K_{\max}}{K_{\min}}$. We define $\sigma_1: [t_1, +\infty) \to \bar\Omega$ by $\sigma_1(t) = \gamma_0(\phi(t))$, which, by construction, satisfies $\sigma_1(t_1) = x_0$, $\sigma_1 \in \Lip(\bar\Omega)$, and $\dot\sigma_1(t) = k(t, \sigma_1(t)) u_1(t)$ for a.e.\ $t \in [t_1, +\infty)$, where $u_1$ is defined for $t \in [t_1, +\infty)$ by $u_1(t) = u_0(\phi(t))$. In particular, since $\phi^{-1}$ is Lipschitz continuous, $u_1$ is measurable. We define $t_1^\ast = \phi^{-1}(t_0^\ast)$ and remark that $\sigma_1(t_1^\ast) = \gamma_0(t_0^\ast) = x_0^\ast \in \Gamma$.

Finally, we define $\gamma_1 \in \Lip(\bar\Omega)$ by $\gamma_1(t) = x_1$ for $t \in [0, t_0]$, $\gamma_1(t) = \sigma(t)$ for $t \in [t_0, t_1]$, and $\gamma_1(t) = \sigma_1(t)$ for $t \in [t_1, +\infty)$. By construction, we have $\gamma_1 \in \Adm(k)$ and $\tau(t_0, \gamma_1) \leq t_1^\ast - t_0$. We are thus only left to show \eqref{eq:Lip-estimate-to-be-shown}.

Note that, for every $t \in [t_1, +\infty)$, we have $k(\phi(t), \gamma_0(\phi(t))) \dot\phi(t) = k(t, \gamma_0(\phi(t)))$ and thus, by integrating this identity and performing a change of variables, we deduce that, for every $t \geq t_1$,
\begin{equation}
\label{eq:interesting-identitiy}
\int_{t_0}^{\phi(t)} k(s, \gamma_0(s)) \diff s = \int_{t_1}^t k(s, \gamma_0(\phi(s))) \diff s.
\end{equation}
Let $G: [t_0, +\infty) \to [0, +\infty)$ be defined for $\theta \geq t_0$ by
\[
G(\theta) = \int_{t_0}^\theta k(s, \gamma_0(s)) \diff s.
\]
Then $G$ is differentiable, with $\dot G(\theta) = k(\theta, \gamma_0(\theta)) \in [K_{\min}, K_{\max}]$ for every $\theta \geq t_0$. In particular, $G$ is $K_{\max}$-Lipschitz continuous, invertible, and its inverse is $\frac{1}{K_{\min}}$-Lipschitz continuous. Moreover, using \eqref{eq:interesting-identitiy}, we have, for every $t \geq t_1$, that
\[
G(\phi(t)) = \int_{t_1}^t k(s, \gamma_0(\phi(s))) \diff s, \qquad G(t) = \int_{t_0}^t k(s, \gamma_0(s)) \diff s.
\]
Hence, for every $t \geq t_1$, we have
\begin{align*}
\abs{\phi(t) - t} & = \abs*{G^{-1}\left(\int_{t_1}^t k(s, \gamma_0(\phi(s))) \diff s\right) - G^{-1}\left(\int_{t_0}^t k(s, \gamma_0(s)) \diff s\right)} \displaybreak[0] \\
& \leq \frac{1}{K_{\min}} \abs*{\int_{t_1}^t k(s, \gamma_0(\phi(s))) \diff s - \int_{t_0}^t k(s, \gamma_0(s)) \diff s} \displaybreak[0] \\
& \leq \frac{1}{K_{\min}} \int_{t_1}^t \abs*{k(s, \gamma_0(\phi(s))) - k(s, \gamma_0(s))}\diff s + \frac{1}{K_{\min}}\int_{t_0}^{t_1} k(s, \gamma_0(s)) \diff s \displaybreak[0] \\
& \leq \frac{L K_{\max}}{K_{\min}} \int_{t_1}^t \abs*{\phi(s) - s} \diff s + \frac{K_{\max}}{K_{\min}} (t_1 - t_0).
\end{align*}
Thus, by Grönwall's inequality, we have, for every $t \geq t_1$,
\[
\abs{\phi(t) - t} \leq (t_1 - t_0) \frac{K_{\max}}{K_{\min}} e^{\frac{L K_{\max}}{K_{\min}} (t - t_1)},
\]
which yields, for $t = t_1^\ast$, that
\[
\abs*{t_1^\ast - t_0^\ast} \leq (t_1 - t_0) \frac{K_{\max}}{K_{\min}} e^{\frac{L K_{\max}}{K_{\min}} (t_1^\ast - t_1)}.
\]
Note that $0 \leq t_1^\ast - t_1 = \phi^{-1}(t_0^\ast) - \phi^{-1}(t_0) \leq \frac{K_{\max}}{K_{\min}} (t_0^\ast - t_0) = \frac{K_{\max}}{K_{\min}} \varphi(t_0, x_0) \leq \frac{T K_{\max}}{K_{\min}}$. Moreover, $t_1^\ast = t_0 + \tau(t_0, \gamma_1)$ and $t_0^\ast = t_0 + \varphi(t_0, x_0)$, showing that $\abs{t_1^\ast - t_0^\ast} = \abs{\tau(t_0, \gamma_1) - \varphi(t_0, x_0)}$. Hence, we deduce that
\[
\tau(t_0, \gamma_1) - \varphi(t_0, x_0) \leq (t_1 - t_0) \frac{K_{\max}}{K_{\min}} \exp\left(\frac{T L K_{\max}^2}{K_{\min}^2}\right).
\]
Recalling that $0 \leq t_1 - t_0 \leq \frac{D \abs{x_1 - x_0}}{K_{\min}}$, we finally obtain \eqref{eq:Lip-estimate-to-be-shown} with $C = \frac{D K_{\max}}{K_{\min}^2} \exp\left(\frac{T L K_{\max}^2}{K_{\min}^2}\right)$.
\end{proof}

Now, exactly as in \cite[Proposition~4.8]{SadeghiMulti}, we can deduce Lipschitz continuity of $\varphi$ by using Lemma~\ref{lemm-varphi-Lipschitz} and the dynamic programming principle from Proposition~\ref{PropOCP-1}\ref{PropOCP-DPP}. We provide a proof of this fact here for sake of completeness.

\begin{proposition}\label{varphi is Lipschitz}
Consider the optimal control problem $\OCP(k)$ and its value function $\varphi$ and assume that \ref{HypoOmega}, \ref{HypoGamma}, \ref{HypoOmegaGeodesic}, \ref{HypoOCP-k-Bound}, and \ref{HypoOCP-k-Lip} are satisfied. Then there exists $M > 0$ such that, for every $(t_0, x_0), (t_1, x_1) \in \mathbbm R_+ \times \bar{\Omega}$, we have
\[
\abs{\varphi(t_0, x_0) - \varphi(t_1, x_1)} \leq M \left(\abs{t_0 - t_1} + \abs{x_0 - x_1}\right).
\] 
\end{proposition}

\begin{proof}
We denote by $C > 0$ the Lipschitz constant from Lemma~\ref{lemm-varphi-Lipschitz}. Let $(t_0, x_0),\, (t_1, x_1) \in \mathbbm R_+ \times \bar{\Omega}$ and assume, with no loss of generality, that $t_0 < t_1$. Fix $\gamma_0 \in \Opt(k, t_0, x_0)$ and set $x_0^\ast = \gamma_0(t_1)$. By Lemma~\ref{lemm-varphi-Lipschitz}, we have
\begin{equation}
\label{eq:varphi-Lipschitz-intermediate-step-2}
\abs{\varphi(t_1, x_0^\ast) - \varphi(t_1, x_1)} \leq C\abs{x_0^\ast - x_1}.
\end{equation}

If $t_1 \leq t_0 + \varphi(t_0, x_0)$, then, by Proposition~\ref{PropOCP-1}\ref{PropOCP-DPP}, since $\gamma_0 \in \Opt(k, t_0, x_0)$, we have $\varphi(t_1, x_0^\ast) = \varphi(t_0, x_0) - (t_1 - t_0)$, and thus
\begin{equation}
\label{eq:varphi-Lipschitz-intermediate-step-3}
\abs{\varphi(t_0, x_0) - \varphi(t_1, x_1)} \leq \abs{t_1 - t_0} + C\abs{x_0^\ast - x_1}.
\end{equation}
Otherwise, we have $t_1 > t_0 + \varphi(t_0, x_0)$, which shows that $x_0^\ast = \gamma_0(t_1) = \gamma_0(t_0 + \varphi(t_0, x_0)) \in \Gamma$, and thus $\varphi(t_1, x_0^\ast) = 0$. Combining this with \eqref{eq:varphi-Lipschitz-intermediate-step-2} and the fact that $\varphi(t_0, x_0) < t_1 - t_0$, we deduce that \eqref{eq:varphi-Lipschitz-intermediate-step-3} also holds in this case.

Since $\gamma_0$ is $K_{\max}$-Lipschitz, we have $\abs{x_0^\ast - x_0} = \abs{\gamma_0(t_1) - \gamma_0(t_0)} \leq K_{\max} \abs{t_1 - t_0}$. Hence, combining with \eqref{eq:varphi-Lipschitz-intermediate-step-3}, we deduce that
\[
\abs{\varphi(t_0, x_0) - \varphi(t_1, x_1)} \leq (C K_{\max} + 1) \abs{t_1 - t_0} + C \abs{x_0 - x_1},
\]
yielding the conclusion.
\end{proof}

The next preliminary result we present is the fact that $\varphi$ satisfies a Hamilton--Jacobi equation. This kind of result is classical and can be obtained by adapting classical proofs used in the autonomous case, such as those of \cite[Chapter~IV, Proposition~2.3]{BardiDolcetta} and \cite[Theorem~8.1.8]{CannarsaPiermarcoSinestrari}. The statement provided here can also be found in \cite[Proposition~3.5]{Dweik2020Sharp} and \cite[Theorem~4.1]{Mazanti2019Minimal}.

\begin{proposition}
\label{PropOCP-HJ}
Consider $\OCP(k)$ under hypotheses \ref{HypoOmega}, \ref{HypoGamma}, \ref{HypoOmegaGeodesic}, \ref{HypoOCP-k-Bound}, and \ref{HypoOCP-k-Lip}. The value function $\varphi$ of $\OCP(k)$ satisfies the Hamilton--Jacobi equation
\begin{equation}
\label{H-J equation}
-\partial_t \varphi(t,x)+\abs{\nabla \varphi(t,x)} k(t, x) - 1 = 0
\end{equation}
in the following sense: $\varphi$ is a viscosity subsolution of \eqref{H-J equation} in $\mathbbm R_+ \times (\Omega \setminus \Gamma)$ and a viscosity supersolution of \eqref{H-J equation} in $\mathbbm R_+ \times (\bar\Omega \setminus \Gamma)$. Moreover, $\varphi$ satisfies $\varphi(t, x) = 0$ for every $(t, x) \in \mathbbm R_+ \times \Gamma$.
\end{proposition}

We next provide in Proposition~\ref{PropOCP-2} two properties of $\varphi$, the first one providing a lower bound on the rate of change of $\varphi$ over time at a fixed position, and the second one characterizing the optimal control at points at which $\varphi$ and the optimal trajectory are differentiable. The first result was shown in \cite[Proposition~4.4]{Mazanti2019Minimal}, once again under the stronger assumption that $k \in \Lip(\mathbbm R_+ \times \bar\Omega; \mathbbm R_+)$, but its proof was later refined in \cite[Proposition~3.9]{Dweik2020Sharp} in order to use only assumptions \ref{HypoOCP-k-Bound} and \ref{HypoOCP-k-Lip}. Even though \cite{Dweik2020Sharp} considers only minimal-time mean field games without state constraints, the proof of \cite[Proposition~3.9]{Dweik2020Sharp} remains unchanged when state constraints are present. The second result of Proposition~\ref{PropOCP-2} follows as a consequence of the first one and Proposition~\ref{PropOCP-HJ}, as detailed, for instance, in \cite[Corollary~4.1]{Mazanti2019Minimal}.

\begin{proposition}
\label{PropOCP-2}
Consider $\OCP(k)$ under hypotheses \ref{HypoOmega}, \ref{HypoGamma}, \ref{HypoOmegaGeodesic}, \ref{HypoOCP-k-Bound}, and \ref{HypoOCP-k-Lip}, and let $\varphi$ be the value function of $\OCP(k)$.
\begin{enumerate}
\item\label{PropOCP-LowerBound} There exists $c > 0$ such that, for every $x \in \bar\Omega$ and $t_0, t_1 \in \mathbbm R_+$ with $t_0 \neq t_1$, we have
\[
\frac{\varphi(t_1, x) - \varphi(t_0, x)}{t_1 - t_0} \geq c - 1.
\]
In particular, if $\varphi$ is differentiable at $(t_0, x)$, then $\partial_t \varphi(t_0, x) \geq c - 1$ and $\abs{\nabla\varphi(t_0, x)} \geq c$.
\item\label{PropOCP-OptimalControl} For every $(t_0, x_0) \in \mathbbm R_+ \times \bar\Omega$, if $\gamma \in \Opt(k, t_0, x_0)$, $t \in [t_0, t_0 + \varphi(t_0, x_0))$, and $\varphi$ is differentiable at $(t, \gamma(t))$, then $\abs{\nabla\varphi(t, \gamma(t))} \neq 0$ and
\[
\dot\gamma(t) = - k(t, \gamma(t)) \frac{\nabla\varphi(t, \gamma(t))}{\abs{\nabla\varphi(t, \gamma(t))}}.
\]
\end{enumerate}
\end{proposition}

We now turn to preliminary results concerning the mean field game $\MFG(K)$. The main result on $\MFG(K)$ we present here is the following, asserting existence of equilibria.

\begin{proposition}
\label{prop:exists-equilibrium}
Consider the mean field game $\MFG(K)$ under hypotheses \ref{HypoOmega}, \ref{HypoGamma}, \ref{HypoOmegaGeodesic}, \ref{HypoMFG-K-Bound}, and \ref{HypoMFG-K-Lip}. Then, for every $m_0\in \mathcal{P}(\bar{\Omega})$, there exists an equilibrium $Q\in \mathcal{P}(\mathbf{C}(\bar{\Omega}))$ for $\MFG(K)$ with initial condition $m_0$.
\end{proposition}

A result similar to Proposition~\ref{prop:exists-equilibrium} was shown in \cite[Theorem~5.1]{Mazanti2019Minimal}, but requiring the stronger assumption that $K \in \Lip(\mathcal P(\bar\Omega) \times \bar\Omega; \mathbbm R_+)$ instead of \ref{HypoMFG-K-Lip}. Proofs using only the weaker assumption \ref{HypoMFG-K-Lip} were provided in \cite[Theorem~4.4]{Dweik2020Sharp} and \cite[Theorem~5.1]{SadeghiMulti}. Both references consider only mean field games without state constraints, but the only point in those proofs where the absence of state constraints is used is to show Lipschitz continuity of value functions of optimal control problems with functions $k$ satisfying \ref{HypoOCP-k-Bound} and \ref{HypoOCP-k-Lip}. Since Proposition~\ref{varphi is Lipschitz} above proves this fact also in the presence of state constraints, the proof of Proposition~\ref{prop:exists-equilibrium} can be carried out exactly as in \cite[Theorem~4.4]{Dweik2020Sharp} and \cite[Theorem~5.1]{SadeghiMulti}, and it is thus omitted here. We also recall that equilibria of $\MFG(K)$ may fail to be unique, as illustrated in \cite[Remark~7.1]{Mazanti2019Minimal} in the case of a population without interaction and through a heuristic argument in \cite[Remark~5.7]{SadeghiMulti} in the case of a multi-population mean field game.

The definition of equilibrium used here, which follows that provided in \cite{Mazanti2019Minimal, SadeghiMulti}, requires $Q$-almost every $\gamma \in \mathbf C(\bar\Omega)$ to be an optimal trajectory for $\OCP(k_Q)$ from $(0, \gamma(0))$. An apparently stronger notion of equilibrium can be formulated by requiring \emph{every} trajectory in the support of $Q$ to be optimal for the same problem, a notion used, for instance, in \cite{Dweik2020Sharp, CannarsaPiermarco}. We prove in the next proposition that, for the current model, both notions are equivalent.

\begin{proposition}
\label{prop:equilibrium-support}
Consider the mean field game $\MFG(K)$ under hypotheses \ref{HypoOmega}, \ref{HypoGamma}, \ref{HypoOmegaGeodesic}, \ref{HypoMFG-K-Bound}, and \ref{HypoMFG-K-Lip}, and let $m_0 \in \mathcal P(\bar\Omega)$. A measure $Q\in \mathcal{P}(\mathbf{C}(\bar{\Omega}))$ is an equilibrium of $\MFG(K)$ with initial condition $m_0$ if and only if $e_{0\#}Q = m_0$ and every $\gamma \in \supp(Q)$ satisfies $\gamma \in \Opt(k_Q, 0, \gamma(0))$, where $k_Q:\mathbbm R_+ \times \bar{\Omega} \to \mathbbm R_+$ is defined for $t \geq 0$ and $x \in \bar\Omega$ by $k_Q(t,x)=K(e_{t\#}Q,x)$.
\end{proposition}

\begin{proof}
Notice first that $t \mapsto e_{t\#} Q$ is Lipschitz continuous with Lipschitz constant $K_{\max}$, since, for $t_0, t_1 \in \mathbbm R_+$, we have
\begin{align*}
\mathbf W_1(e_{t_0\#} Q, e_{t_1\#} Q) & = \sup_{\Phi \in \Lip_1(\bar\Omega; \mathbbm R)} \int_{\bar\Omega} \Phi(x) \diff (e_{t_0\#} Q - e_{t_1\#} Q)(x) \displaybreak[0] \\
& = \sup_{\Phi \in \Lip_1(\bar\Omega; \mathbbm R)} \int_{\mathbf C(\bar\Omega)} \left[\Phi(\gamma(t_0)) - \Phi(\gamma(t_1))\right] \diff Q(\gamma) \leq K_{\max} \abs{t_0 - t_1}.
\end{align*}
In particular, $k_Q$ satisfies \ref{HypoOCP-k-Bound} and \ref{HypoOCP-k-Lip}.

Let us assume that $Q \in \mathcal P(\mathbf C(\bar\Omega))$ is such that $e_{0\#}Q = m_0$ and every $\gamma \in \supp(Q)$ satisfies $\gamma \in \Opt(k_Q, 0, \gamma(0))$. Since $\supp(Q)$ is a set of full $Q$ measure, we deduce immediately that $Q$-almost every $\gamma \in \mathbf C(\bar\Omega)$ satisfies $\gamma \in \Opt(k_Q, 0, \gamma(0))$, and hence $Q$ is an equilibrium of $\MFG(K)$ with initial condition $m_0$.

To prove the converse implication, let us now assume that $Q$ is an equilibrium of $\MFG(K)$ with initial condition $m_0$. Fix $\gamma \in \supp(Q)$. By definition of support, for every open neighborhood $V$ of $\gamma$ in $\mathbf C(\bar\Omega)$, we have $Q(V) > 0$, and thus, by assumption, there exists $\gamma_V \in V$ such that $\gamma_V \in \Opt(k_Q, 0, \gamma_V(0))$. In particular, there exists a sequence $(\gamma_n)_{n \in \mathbbm N}$ such that $\gamma_n \in \Opt(k_Q, 0, \gamma_n(0))$ for every $n \in \mathbbm N$ and $\gamma_n \to \gamma$ in the topology of $\mathbf C(\bar\Omega)$ (i.e., uniformly on compact subsets of $\mathbbm R_+$) as $n \to +\infty$.

Since $\gamma_n \in \Opt(k_Q, 0, \gamma_n(0))$, we have in particular $\gamma_n \in \Adm(k_Q)$, which implies that, for every $t_0, t_1 \in \mathbbm R_+$, we have
\[
\abs{\gamma_n(t_1) - \gamma_n(t_0)} \leq \int_{t_0}^{t_1} k_Q(s, \gamma_n(s)) \diff s.
\]
Letting $n \to +\infty$, we deduce that
\[
\abs{\gamma(t_1) - \gamma(t_0)} \leq \int_{t_0}^{t_1} k_Q(s, \gamma(s)) \diff s,
\]
which shows in particular that $\gamma$ is $K_{\max}$-Lipschitz continuous. Moreover, dividing by $\abs{t_1 - t_0}$ and letting $t_1 \to t_0$, we deduce that $\abs{\dot\gamma(t)} \leq k_Q(t, \gamma(t))$ for almost every $t \in \mathbbm R_+$, showing that $\gamma \in \Adm(k_Q)$.

To prove optimality of $\gamma$, notice that $\gamma_n(\varphi_Q(0, \gamma_n(0))) \in \Gamma$, where $\varphi_Q$ denotes the value function of $\OCP(k_Q)$. Since $\Gamma$ is closed and, by Proposition~\ref{varphi is Lipschitz}, $\varphi_Q$ is continuous, we deduce by letting $n \to +\infty$ that $\gamma(\varphi_Q(0, \gamma(0))) \in \Gamma$, showing that $\tau(0, \gamma) \leq \varphi_Q(0, \gamma(0))$. On the other hand, since $\gamma \in \Adm(k_Q)$, we have $\varphi_Q(0, \gamma(0)) \leq \tau(0, \gamma)$ by the definition \eqref{value function} of $\varphi_Q$, yielding that $\tau(0, \gamma) = \varphi_Q(0, \gamma(0))$. Thus $\gamma \in \Opt(k_Q, 0, \gamma(0))$, as required.
\end{proof}

\section{Further properties of the optimal control problem}
\label{More on OPT}

We now provide some further results on the optimal control problem $\OCP(k)$. In the references \cite{Mazanti2019Minimal, Dweik2020Sharp, SadeghiMulti}, additional properties of the value function and of optimal trajectories, such as $\mathbf{C}^{1, 1}$ regularity of optimal trajectories or differentiability of the value function along optimal trajectories, were obtained only for optimal control problems without state constraints. The technique used in those references is to apply the (unconstrained) Pontryagin Maximum Principle to obtain further properties of optimal trajectories that can then be used to deduce additional results on $\OCP(k)$. The main difficulty in generalizing this technique is that, even though versions of Pontryagin Maximum Principle taking into account control systems with state constraints such as \eqref{General control sys} exist in the literature (see, e.g., \cite[Chapter~5]{Clarke}), it is a difficult problem to get nice additional properties of $\OCP(k)$ from their conclusions.

The strategy we follow in this section relies instead on a penalization technique adapted from \cite{CannarsaCastelpietraCardaliaguet}, which consists in considering an optimal control problem with no state constraints but such that the maximal speed of each agent decays fast to $0$ as the agent moves away from $\bar\Omega$. It is then possible to show that, if the penalization parameter is small enough, optimal trajectories of the penalized problem never leave $\bar\Omega$ and thus they coincide with optimal trajectories of the original problem with state constraints. As consequences, we provide a boundary condition for the Hamilton--Jacobi equation \eqref{H-J equation} and a characterization of optimal controls, the latter being a key ingredient for showing, in Section~\ref{sec MFG system}, that equilibria of $\MFG(K)$ satisfy an MFG system.

\subsection{Penalized optimal control problem}
\label{sec:penalized}

We present in this section the penalized optimal control problem we consider in order to study optimal trajectories of $\OCP(k)$. Before turning to the core of this section, let us first recall, in the next proposition, some classical consequences of \ref{HypoOmegaC11} on the signed distance to $\partial\Omega$, whose proofs can be found, for instance, in \cite[Theorems~5.1 and 5.7]{DelfourZolesio}.

\begin{proposition}
\label{prop:OmegaC11}
Let $\Omega \subset \mathbbm R^d$ be a set satisfying \ref{HypoOmegaC11}.
\begin{enumerate}
\item The signed distance $d^{\pm}_{\partial \Omega}$ is Lipschitz continuous on $\mathbbm R^d$, with Lipschitz constant equal to $1$.
\item\label{item:W-dpm} There exists a neighborhood $W$ of $\partial\Omega$ such that the signed distance to $\partial\Omega$ satisfies $d^{\pm}_{\partial \Omega}\in \mathbf{C}^{1,1}(W; \mathbbm R)$. Moreover, $\abs[\big]{\nabla d^{\pm}_{\partial \Omega}(x)}=1$ for every $x\in W$.
\item For every $x \in \partial\Omega$, $\nabla d^{\pm}_{\partial \Omega}(x)$ is the outward unit normal to $\Omega$ at $x$.
\end{enumerate}
\end{proposition}

Let us now introduced the penalized optimal control problem. We fix $\Omega$, $\Gamma$, and $k$ satisfying assumptions \HypoOmega{}, \ref{HypoOCP-k-Bound}, and \ref{HypoOCP-k-Lip}, and we extend $k$ to a continuous function defined on $\mathbbm R_+ \times \mathbbm R^d$ and taking values in $[K_{\min}, K_{\max}]$ such that, for every $t \in \mathbbm R_+$, $x \mapsto k(t, x)$ is $L$-Lipschitz continuous on $\mathbbm R^d$, with $L$ independent of $t$. Such an extension can be constructed, for instance, by proceeding as in \cite{Banach-Lip-Extend}.

For $\epsilon > 0$, we define $k_{\epsilon}: \mathbbm R_+ \times \mathbbm R^d\to \mathbbm R_+$ by
\begin{equation}
\label{eq:defi-k-epsilon}
k_{\epsilon}(t,x)=k(t,x)\left(1-\frac{1}{\epsilon}d_{\Omega}(x)\right)_+,
\end{equation}
where the notation $a_+$ is defined by $a_+=\max(0,a)$ for $a\in \mathbbm R$. In particular, $k_\epsilon$ is continuous, nonnegative, upper bounded by $K_{\max}$, $x \mapsto k_\epsilon(t, x)$ is Lipschitz continuous for every $t \in \mathbbm R_+$ with a Lipschitz constant $L_\epsilon = L + \frac{K_{\max}}{\epsilon}$ independent of $t$, $k_\epsilon(t, x) = k(t, x)$ for every $(t, x) \in \mathbbm R_+ \times \bar\Omega$, and $k_\epsilon(t, x) = 0$ if $d_\Omega(x) \geq \epsilon$.

We consider the penalized control system
\begin{equation}
\label{eq:penalized-control-system}
\Dot{\gamma}(t)=k_{\epsilon}(t,\gamma(t))u(t), \qquad t \geq 0,
\end{equation}
where $\gamma(t) \in \mathbbm R^d$ is the state (which is no longer constrained to remain in $\bar\Omega$) and $u(t) \in \bar B$ is the control. We denote by $\OCP_\epsilon(k_\epsilon)$ the minimal-time optimal control problem of finding, for any $(t_0, x_0) \in \mathbbm R_+ \times \mathbbm R^d$, a trajectory $\gamma \in \Lip(\mathbbm R^d)$ with $\gamma(t_0) = x_0$ and a measurable control $u: \mathbbm R_+ \to \bar B$ satisfying \eqref{eq:penalized-control-system} and minimizing the time at which $\gamma$ reaches the target set $\Gamma$ for the first time after $t_0$. Similarly to Definition~\ref{def:OCP}, we denote by $\Adm_\epsilon(k_\epsilon)$ the set of admissible curves for \eqref{eq:penalized-control-system}, i.e., the set of $\gamma \in \Lip(\mathbbm R^d)$ such that \eqref{eq:penalized-control-system} is satisfied for some measurable $u: \mathbbm R_+ \to \bar B$. The definition of the first exit time $\tau(t_0, \gamma)$ from Definition~\ref{def:OCP}\ref{def:OCP-tau} is extended to curves $\gamma \in \Lip(\mathbbm R^d)$, and we define optimal trajectories for $\OCP_\epsilon(k_\epsilon)$ as in Definition~\ref{def:OCP}\ref{item:def:optimal-trajectory}, but requiring in addition that $\tau(t_0, \gamma) < +\infty$ for a trajectory $\gamma$ to be considered optimal from $(t_0, x_0) \in \mathbbm R_+ \times \mathbbm R^d$. The set of optimal trajectories for $\OCP_\epsilon(k_\epsilon)$ from $(t_0, x_0) \in \mathbbm R_+ \times \mathbbm R^d$ is denoted by $\Opt_\epsilon(k_\epsilon, t_0, x_0)$. The value function $\varphi_\epsilon: \mathbbm R_+ \times \mathbbm R^d \to \mathbbm R_+ \cup \{+\infty\}$ of $\OCP_\epsilon(k_\epsilon)$ is defined similarly to Definition~\ref{def:OCP}\ref{def:value-function}.

Standard arguments allow to show that $\OCP_\epsilon(k_\epsilon)$ also satisfies a dynamic programming principle: if $(t_0, x_0) \in \mathbbm R_+ \times \mathbbm R^d$ and $\gamma \in \Adm_\epsilon(k_\epsilon)$ is such that $\gamma(t_0) = x_0$, then, for every $h \geq 0$, we have
\begin{equation}
\label{eq:DPP-epsilon}
\varphi_\epsilon(t_0 + h, \gamma(t_0 + h)) + h \geq \varphi_\epsilon(t_0, x_0),
\end{equation}
with equality if $\gamma \in \Opt_\epsilon(k_\epsilon, t_0, x_0)$ and $h \in [0, \varphi_\epsilon(t_0, x_0)]$. However, contrarily to $\OCP(k)$, we may have nonexistence of optimal trajectories for $\OCP_\epsilon(k_\epsilon)$. More precisely, $\Opt_\epsilon(k_\epsilon,\allowbreak t_0,\allowbreak x_0) = \varnothing$ if $d_\Omega(x_0) \geq \epsilon$, but one can show by standard arguments that this set is nonempty as soon as $d_\Omega(x_0) < \epsilon$. In addition, $\varphi_\epsilon(t, x) < +\infty$ if and only if $d_\Omega(x) < \epsilon$.

The first result we show for $\OCP_\epsilon(k_\epsilon)$ is the following property of optimal trajectories.

\begin{proposition}
\label{prop:opt-epsilon-remains-close}
Consider the optimal control problem $\OCP_\epsilon(k_{\epsilon})$ under assumptions \HypoOmega{}, \ref{HypoOCP-k-Bound}, and \ref{HypoOCP-k-Lip}. Let $(t_0,x_0)\in \mathbbm R_+\times \mathbbm R^d$ and $\gamma \in \Opt_\epsilon(k_{\epsilon},t_0,x_0)$. Then $d_{\Omega}(\gamma(t)) < \epsilon$ for every $t \in \mathbbm R_+$.
\end{proposition}

\begin{proof}
Denote by $u: \mathbbm R_+ \to \bar B$ an optimal control associated with $\gamma$. Assume, to obtain a contradiction, that there exists $t_1 \in \mathbbm R_+$ such that $d_{\Omega}(\gamma(t_1)) \ge \epsilon$. In particular, $k_{\epsilon}(t, \gamma(t_1)) = 0$ for every $t\in \mathbbm R_+$, and thus the function $\Tilde{\gamma}: \mathbbm R_+ \to \mathbbm R^d$, defined by $\Tilde{\gamma}(t) = \gamma(t_1)$ satisfies $\dot{\tilde\gamma}(t) = k_\epsilon(t, \tilde\gamma(t)) u(t)$ for every $t\in \mathbbm R_+$. Since the initial value problem consisting of this equation and the initial condition $\tilde\gamma(t_1) = \gamma(t_1)$ admits a unique solution, then necessarily $\Tilde{\gamma} = \gamma$. However, since $\gamma \in \Opt_\epsilon(k_{\epsilon},t_0,x_0)$, we have $\tau(t_0, \gamma) < +\infty$ and $\gamma(t_0 + \tau(t_0, \gamma)) \in \Gamma \subset \bar\Omega$, which contradicts the fact that $d_\Omega(\gamma(t_0 + \tau(t_0, \gamma))) = d_\Omega(\tilde\gamma(t_0 + \tau(t_0, \gamma)))= d_\Omega(\gamma(t_1)) \geq \epsilon$.
\end{proof}

We now apply Pontryagin Maximum Principle to optimal trajectories of the penalized problem $\OCP_\epsilon(k_\epsilon)$ starting from a point in $\mathbbm R^d \setminus \Gamma$.

\begin{proposition}\label{Conse Pontryagin}
Consider the optimal control problem $\OCP_\epsilon(k_{\epsilon})$ under assumptions \HypoOmega{}, \ref{HypoOCP-k-Bound}, and \ref{HypoOCP-k-Lip}, and with $k_\epsilon$ defined from $k$ as in \eqref{eq:defi-k-epsilon}. Let $(t_0,x_0)\in \mathbbm R_+\times (\mathbbm R^d \setminus \Gamma)$, $\gamma \in \Opt_\epsilon(k_{\epsilon}, t_0, x_0)$, $T = \varphi_{\epsilon}(t_0,x_0)$, and $u: \mathbbm R_+ \to \bar{B}$ be an optimal control associated with $\gamma$. Then there exist $\lambda \geq 0$ and an absolutely continuous function $p: [t_0, t_0+T] \to \mathbbm R^d$ such that the following assertions hold.
\begin{enumerate}
\item\label{PMP:nontrivial} $\lambda + \max_{t\in [t_0,t_0+T]} \abs*{p(t)} > 0$.   

\item\label{PMP:costate-equation} For almost every $t \in [t_0,t_0+T]$, we have
\begin{equation}
\label{eq:PMP:costate}
\dot p(t) \in \conv\left\{\zeta \in \mathbbm R^d \suchthat (\zeta, p(t)) \in N_{G_\epsilon(t)}(\gamma(t), \dot\gamma(t))\right\},
\end{equation}
where, for $t \geq 0$, the set $G_\epsilon(t)$ is defined by
\begin{equation}
\label{eq:PMP:defi-G}
G_\epsilon(t) = \left\{(x, v) \in \mathbbm R^d \times \mathbbm R^d \suchthat \exists u \in \bar B \text{ such that } k_\epsilon(t, x) u = v\right\}.
\end{equation}

\item\label{PMP:p-normal} $-p(t_0 + T) \in N_{\Gamma}(\gamma(t_0 + T))$.

\item\label{PMP:optimal-control} One has
\begin{equation*}
u(t) = \frac{p(t)}{\abs*{p(t)}},
\end{equation*}
almost everywhere on $\{t\in [t_0,t_0 + T] \suchthat p(t) \neq 0\}$.  

\item\label{PMP:final-Hamiltonian} $\lambda = k_\epsilon(t_0 + T, \gamma(t_0 + T)) \abs*{p(t_0 + T)}$.
\end{enumerate}
\end{proposition}

Before turning to the proof of Proposition~\ref{Conse Pontryagin}, we remark that, by taking $\gamma \in \Opt_\epsilon(k_\epsilon,\allowbreak t_0,\allowbreak x_0)$, we implicitly assume that $\Opt_\epsilon(k_\epsilon,\allowbreak t_0,\allowbreak x_0)$ is not empty, and thus $d_{\Omega}(x_0) < \epsilon$.

\begin{proof}
We apply \cite[Theorem~8.4.1]{Vinter2010Optimal} to $\OCP_\epsilon(k_\epsilon)$ and the optimal trajectory $\gamma$ over the interval $[t_0, t_0 + T]$. Notice, first, that, since $\gamma \in \Opt_\epsilon(k_\epsilon,\allowbreak t_0,\allowbreak x_0)$, we have in particular $T < +\infty$ and, since $x_0 \in \mathbbm R^d \setminus \Gamma$, we have $T > 0$. We denote by $L_\epsilon$ the Lipschitz constant of $x \mapsto k_\epsilon(t, x)$, and we recall that this constant is independent of $t \in \mathbbm R_+$.

We first introduce some notation in accordance to the statement of \cite[Theorem~8.4.1]{Vinter2010Optimal}. Let $g: \mathbbm R \times \mathbbm R^d \times \mathbbm R \times \mathbbm R^d \to \mathbbm R$ be defined for $(t_1, x_1, t_2, x_2) \in \mathbbm R \times \mathbbm R^d \times \mathbbm R \times \mathbbm R^d$ by $g(t_1, x_1, t_2, x_2) = t_2 - t_1$, define $C = \{(t_1, x_1, t_2, x_2) \in \mathbbm R \times \mathbbm R^d \times \mathbbm R \times \mathbbm R^d \suchthat t_1 = t_0,\, x_1 = x_0,\, t_2 \geq t_1,\, x_2 \in \Gamma\}$, and let $F_\epsilon: \mathbbm R \times \mathbbm R^d \rightrightarrows \mathbbm R^d$ be the set-valued map defined for $(t, x) \in \mathbbm R \times \mathbbm R^d$ by $F_\epsilon(t, x) = \{k_{\epsilon}(\max(t, 0), x) u \suchthat u \in \bar B\}$. Notice that the optimal control problem $\OCP_\epsilon(k_\epsilon)$ starting from the given $(t_0, x_0) \in \mathbbm R_+ \times \mathbbm R^d$ can be rephrased as: minimize $g(t_0, \tilde\gamma(t_0), t_0 + T, \tilde\gamma(t_0 + T))$ over intervals $[t_0, t_0 + T]$ and absolutely continuous functions $\tilde\gamma: [t_0, t_0 + T] \to \mathbbm R^d$ such that $\dot{\tilde\gamma}(t) \in F_\epsilon(t, \tilde\gamma(t))$ for a.e.\ $t \in [t_0, t_0 + T]$ and $(t_0, \tilde\gamma(t_0), t_0 + T, \tilde\gamma(t_0 + T)) \in C$. Remark also that the set $G_\epsilon(t)$ from \ref{PMP:costate-equation} is the graph of $\mathbbm R^d \ni x \mapsto F_\epsilon(t, x) \subset \mathbbm R^d$.

Let us verify the assumptions of \cite[Theorem~8.4.1]{Vinter2010Optimal}. First, the constant $\delta > 0$ from \cite[Theorem~8.4.1]{Vinter2010Optimal} can be chosen arbitrarily. Since $g$ is of class $\mathbf C^{\infty}$ and $\Gamma$ is closed, assumption (H1) from \cite[Theorem~8.4.1]{Vinter2010Optimal} is verified. The fact that $k_\epsilon$ is continuous also allows one to easily verify assumption (H2) from \cite[Theorem~8.4.1]{Vinter2010Optimal}. Using the fact that $k_\epsilon$ is $L_\epsilon$-Lipschitz continuous in its second variable uniformly with respect to the first variable, assumption (H3) from \cite[Theorem~8.4.1]{Vinter2010Optimal} is satisfied with $\beta = 0$ and $k_{F_\epsilon}(t) = L_\epsilon$ for every $t \in [t_0, t_0 + T]$. Since $t_0$ is fixed, assumption (H4) from \cite[Theorem~8.4.1]{Vinter2010Optimal} is not needed. Finally, using once again that $k_\epsilon$ is $L_\epsilon$-Lipschitz continuous in its second variable uniformly with respect to the first variable, we deduce that (H5) from \cite[Theorem~8.4.1]{Vinter2010Optimal} holds with $c_1 = K_{\max}$, $k_1 = L_\epsilon$, and with an arbitrary $\delta_1 > 0$. Thus, all assumptions of \cite[Theorem~8.4.1]{Vinter2010Optimal} are verified.

The statement of \cite[Theorem~8.4.1]{Vinter2010Optimal} now asserts the existence of an absolutely continuous $p: [t_0, t_0+T] \to \mathbbm R^d$ and real numbers $\lambda \geq 0$, $\xi$, and $\eta$, satisfying (i)--(iv) and (vi) from the statement of \cite[Theorem~8.4.1]{Vinter2010Optimal}. Items (i) and (ii) are exactly \ref{PMP:nontrivial} and \ref{PMP:costate-equation} above.

Assertion \cite[Theorem~8.4.1(iii)]{Vinter2010Optimal} states that
\begin{equation}
\label{eq:Vinter-iii}
(-\xi, p(t_0), \eta, -p(t_0 + T)) \in \lambda \partial g(t_0, x_0, t_0 + T, \gamma(t_0 + T)) + N_C(t_0, x_0, t_0 + T, \gamma(t_0 + T)),
\end{equation}
where $\partial g$ denotes the limiting subdifferential of $g$ (see \cite[Definition~4.3.1]{Vinter2010Optimal} for its definition). Since $g$ is smooth, we have $\partial g(t_0, x_0, t_0 + T, \gamma(t_0 + T)) = \{\nabla g(t_0, x_0, t_0 + T, \gamma(t_0 + T))\} = \{(-1, 0, 1, 0)\}$.

Let us now compute the limiting normal cone of $C$ at $(t_0, x_0, t_0 + T, \gamma(t_0 + T))$. For that purpose, we first compute its proximal normal cone at every $(t_1, x_1, t_2, x_2) \in C$ with $t_2 > t_1$. Let $(s_1, p_1, s_2, p_2) \in N_C^{\mathrm P}(t_1, x_1, t_2, x_2)$. By Definition~\ref{defi:normal-cones}\ref{defi:proximal-normal-cone}, there exists $M > 0$ such that
\begin{multline*}
(s_1, p_1, s_2, p_2) \cdot (t_1^\prime - t_1, x_1^\prime - x_1, t_2^\prime - t_2, x_2^\prime - x_2) \\ \leq M \left[\abs{t_1^\prime - t_1}^2 + \abs{x_1^\prime - x_1}^2 + \abs{t_2^\prime - t_2}^2 + \abs{x_2^\prime - x_2}^2\right]
\end{multline*}
for every $(t_1^\prime, x_1^\prime, t_2^\prime, x_2^\prime) \in C$. By definition of $C$, we have $t_1 = t_1^\prime = t_0$, $x_1 = x_1^\prime = x_0$, and thus
\begin{equation*}
s_2 (t_2^\prime - t_2) + p_2 \cdot (x_2^\prime - x_2) \leq M \left[\abs{t_2^\prime - t_2}^2 + \abs{x_2^\prime - x_2}^2\right].
\end{equation*}
Taking $t_2^\prime = t_2$, we deduce that $p_2 \cdot (x_2^\prime - x_2) \leq M \abs{x_2^\prime - x_2}^2$ for every $x_2^\prime \in \Gamma$, and thus $p_2 \in N_\Gamma^{\mathrm P}(x_2)$. Taking $x_2^\prime = x_2$, we deduce that $s_2 (t_2^\prime - t_2) \leq M \abs{t_2^\prime - t_2}^2$ for every $t_2^\prime \geq t_1^\prime = t_0$. In particular, for every $\rho > 0$, taking $t_2^\prime = t_2 + \rho$, we have $s_2 \rho \leq M \rho^2$, which yields, since $\rho > 0$ is arbitrary, that $s_2 \leq 0$. On the other hand, taking $t_2^\prime = t_2 - \rho$ for $\rho \in (0, t_2 - t_1]$, we have $- s_2 \rho \leq M \rho^2$, which yields, since $\rho \in (0, t_2 - t_1]$ is arbitrary and $t_2 > t_1$, that $- s_2 \leq 0$, and thus $s_2 = 0$. We have thus shown that\footnote{The converse inclusion can be shown by straightforward arguments, but it is not necessary for our proof. Notice also that this inclusion is false in the case $t_2 = t_1$, but this case is not needed in our proof.}
\[
N_C^{\mathrm P}(t_1, x_1, t_2, x_2) \subset \mathbbm R \times \mathbbm R^d \times \{0\} \times N_\Gamma^{\mathrm P}(x_2),
\]
It now follows, using Definition~\ref{defi:normal-cones}\ref{defi:limiting-normal-cone}, that\footnote{Once again, we also have the converse inclusion, but it is not necessary for our proof.}
\[
N_C(t_0, x_0, t_0 + T, \gamma(t_0 + T)) \subset \mathbbm R \times \mathbbm R^d \times \{0\} \times N_\Gamma(\gamma(t_0 + T)).
\]
In particular, \eqref{eq:Vinter-iii} imposes no constraints on $\xi$ and $p(t_0)$, and asserts that $\eta = \lambda$ and $-p(t_0 + T) \in N_\Gamma(\gamma(t_0 + T))$, showing \ref{PMP:p-normal}.

To show \ref{PMP:optimal-control}, notice that \cite[Theorem~8.4.1(iv)]{Vinter2010Optimal} states that, for almost every $t \in [t_0, t_0 + T]$, we have $p(t) \cdot \dot\gamma(t) \geq p(t) \cdot v$ for every $v \in F_\epsilon(t, \gamma(t))$, which means that, for almost every $t \in [t_0, t_0 + T]$, we have  $k_\epsilon(t, \gamma(t)) p(t)\cdot u(t) \geq k_\epsilon(t, \gamma(t)) p(t) \cdot w$ for every $w \in \bar B$. By Proposition~\ref{prop:opt-epsilon-remains-close}, we have that $d_\Omega(\gamma(t)) < \epsilon$ for every $t \in [t_0, t_0 + T]$, yielding from \eqref{eq:defi-k-epsilon} that $k_\epsilon(t, \gamma(t)) > 0$. Hence, for almost every $t \in [t_0, t_0 + T]$, we have $p(t) \cdot u(t) \geq p(t) \cdot w$ for every $w \in \bar B$, and this is equivalent to \ref{PMP:optimal-control}.

Finally, define $H_\epsilon: \mathbbm R \times \mathbbm R^d \times \mathbbm R^d \to \mathbbm R$ for $(t, x, p) \in \mathbbm R \times \mathbbm R^d \times \mathbbm R^d$ by $H_\epsilon(t, x, p) = \sup_{v \in F_\epsilon(t, x)} p \cdot v$, and notice that $H_\epsilon(t, x, p) = k_\epsilon(t, x) \abs{p}$. In particular, remarking that $H_\epsilon$ is continuous, \cite[Theorem~8.4.1(vi)]{Vinter2010Optimal} then asserts that $\eta = k_\epsilon(t_0 + T, \gamma(t_0 + T)) \abs{p(t_0 + T)}$, and thus \ref{PMP:final-Hamiltonian} holds since $\eta = \lambda$.
\end{proof}

The absolutely continuous function $p$ from Proposition~\ref{Conse Pontryagin}\ref{PMP:costate-equation} is known as the \emph{costate} associated with the optimal trajectory $\gamma$. Even though Proposition~\ref{Conse Pontryagin}\ref{PMP:costate-equation} provides the differential inclusion \eqref{eq:PMP:costate} for the costate $p$, this inclusion is in general hard to manipulate. We provide, in the next lemma, the main consequence of this differential inclusion that we will use in the sequel. Recall that, for $r > 0$, $\bar B_r$ denotes the closed ball centered at the origin and with radius $r$.

\begin{lemma}
\label{lemm:pre-Gronwall-for-p}
Consider the optimal control problem $\OCP_\epsilon(k_{\epsilon})$ under assumptions \HypoOmega{}, \ref{HypoOCP-k-Bound}, and \ref{HypoOCP-k-Lip}, and with $k_\epsilon$ defined from $k$ as in \eqref{eq:defi-k-epsilon}. Let $(t_0,x_0)$, $\gamma$, $T$, and $p$ be as in the statement of Proposition~\ref{Conse Pontryagin} and $L_\epsilon > 0$ be the Lipschitz constant of $x \mapsto k_\epsilon(t, x)$, which is independent of $t$. Then, for almost every $t \in [t_0, t_0 + T]$, we have
\[\conv\left\{\zeta \in \mathbbm R^d \suchthat (\zeta, p(t)) \in N_{G_\epsilon(t)}(\gamma(t), \dot\gamma(t))\right\} \subset \bar B_{L_\epsilon \abs{p(t)}}.\]
In particular, for almost every $t \in [t_0, t_0 + T]$, we have $\abs{\dot p(t)} \leq L_\epsilon \abs{p(t)}$.
\end{lemma}

\begin{proof}
Let $G_\epsilon$ be defined as in the statement of Proposition~\ref{Conse Pontryagin}\ref{PMP:costate-equation}. Since $\bar B_{L_\epsilon \abs{p(t)}}$ is convex, the result is proved if we show that, for a.e.\ $t \in [t_0, t_0 + T]$ and every $\zeta \in \mathbbm R^d$, if $(\zeta, p(t)) \in N_{G_\epsilon(t)}(\gamma(t), \dot\gamma(t))$, then $\abs{\zeta} \leq L_\epsilon \abs{p(t)}$. Note that this inequality is trivial if $\zeta = 0$, and thus we assume from now on that $\zeta \neq 0$.

Fix $t \in [t_0, t_0 + T]$ at which $p$ and $\gamma$ are differentiable and \eqref{eq:PMP:costate} holds. Let $\zeta \in \mathbbm R^d$ be such that $(\zeta, p(t)) \in N_{G_\epsilon(t)}(\gamma(t), \dot\gamma(t))$. By definition of the limiting normal cone of $G_\epsilon(t)$, there exist sequences $(x_n, v_n)_{n \in \mathbbm N}$ in $G_\epsilon(t)$ and $(\zeta_n, p_n)_{n \in \mathbbm N}$ in $\mathbbm R^d \times \mathbbm R^d$ such that $(x_n, v_n) \to (\gamma(t), \dot\gamma(t))$ and $(\zeta_n, p_n) \to (\zeta, p(t))$ as $n \to +\infty$ and $(\zeta_n, p_n) \in N_{G_\epsilon(t)}^{\mathrm P}(x_n, v_n)$ for every $n \in \mathbbm N$. Since $\zeta \neq 0$, we assume, with no loss of generality, that $\zeta_n \neq 0$ for every $n \in \mathbbm N$.

Take $n \in \mathbbm N$. Since $(x_n, v_n) \in G_\epsilon(t)$, it follows from \eqref{eq:PMP:defi-G} that there exists $u_n \in \bar B$ such that $v_n = k_\epsilon(t, x_n) u_n$. Since $(\zeta_n, p_n) \in N_{G_\epsilon(t)}^{\mathrm P}(x_n, v_n)$, there exists $M_n > 0$ such that
\begin{equation}
\label{eq:ineg-p-proximal-normal}
(\zeta_n, p_n) \cdot (y - x_n, k_\epsilon(t, y) u - k_\epsilon(t, x_n) u_n) \leq M_n \left[\abs*{y - x_n}^2 + \abs*{k_\epsilon(t, y) u - k_\epsilon(t, x_n) u_n}^2\right]
\end{equation}
for every $y \in \mathbbm R^d$ and $u \in \bar B$. Let $\alpha_n = ((n+1) (1 + L_\epsilon^2) M_n)^{-1} > 0$ and take $y = x_n + \alpha_n \zeta_n$ and $u = u_n$ in \eqref{eq:ineg-p-proximal-normal}. Then
\begin{multline*}
\alpha_n \abs*{\zeta_n}^2 + \left[k_\epsilon(t, x_n + \alpha_n \zeta_n) - k_\epsilon(t, x_n)\right] p_n \cdot u_n \\
\leq M_n \left[\alpha_n^2 \abs*{\zeta_n}^2 + \abs*{k_\epsilon(t, x_n + \alpha_n \zeta_n) - k_\epsilon(t, x_n)}^2\right].
\end{multline*}
Hence
\[
\alpha_n \abs*{\zeta_n}^2 \leq M_n \alpha_n^2 \abs{\zeta_n}^2 + M_n L_\epsilon^2 \alpha_n^2 \abs{\zeta_n}^2 + \alpha_n L_\epsilon \abs{\zeta_n} \abs{p_n}.
\]
Since $\alpha_n > 0$ and $\abs*{\zeta_n} > 0$, we deduce that
\[
\left(1 - M_n \alpha_n - M_n L_\epsilon^2 \alpha_n\right)\abs*{\zeta_n} \leq L_\epsilon \abs{p_n}.
\]
By definition of $\alpha_n$, we have $M_n \alpha_n + M_n L_\epsilon^2 \alpha_n = \frac{1}{n + 1} < 1$, and thus
\[
\abs*{\zeta_n} \leq \frac{L_\epsilon}{1 - \frac{1}{n+1}} \abs{p_n}.
\]
The conclusion follows by letting $n \to +\infty$.
\end{proof}

Using Lemma~\ref{lemm:pre-Gronwall-for-p}, we can prove additional regularity properties of optimal trajectories and optimal controls for $\OCP_\epsilon(k_\epsilon)$.

\begin{proposition}
\label{prop:consequences-of-Pontryagin-for-optimal-trajectories-of-ocp-epsilon}
Consider the optimal control problem $\OCP_\epsilon(k_{\epsilon})$ under assumptions \HypoOmega{}, \ref{HypoOCP-k-Bound}, and \ref{HypoOCP-k-Lip}, and with $k_\epsilon$ defined from $k$ as in \eqref{eq:defi-k-epsilon}. Let $(t_0,x_0)$, $\gamma$, $T$, $u$, and $p$ be as in the statement of Proposition~\ref{Conse Pontryagin} and $L_\epsilon > 0$ be the Lipschitz constant of $k_\epsilon$ with respect to its second variable, which is independent of its first variable. For $t \in [t_0, t_0 + T]$, define $\Xi_\epsilon(t)$ by
\begin{equation}
\label{eq:defi-Xi-epsilon}
\Xi_\epsilon(t) = \conv\left\{\xi \in \mathbbm R^d \suchthat (\xi \abs{p(t)}, p(t)) \in N_{G_\epsilon(t)}(\gamma(t), \dot\gamma(t))\right\},
\end{equation}
where $G_\epsilon(t)$ is given by \eqref{eq:PMP:defi-G}. Then, up to redefining $u$ in a set of Lebesgue measure zero, the following assertions hold.
\begin{enumerate}
\item\label{item:p-is-nonzero} For every $t \in [t_0, t_0 + T]$, we have $\abs{p(t)} \neq 0$ and $u(t) = \frac{p(t)}{\abs{p(t)}}$.
\item\label{item:Xi} For almost every $t \in [t_0, t_0 + T]$ and every $\xi \in \Xi_\epsilon(t)$, we have $\abs{\xi} \leq L_\epsilon$.
\item\label{item:system-on-gamma-u} We have $\gamma \in \mathbf{C}^{1}([t_0, t_0 + T]; \mathbbm R^d)$, $u \in \Lip_{L_\epsilon}([t_0, t_0 + T]; \mathbbm S^{d-1})$, and $(\gamma, u)$ satisfies, for almost every $t \in [t_0, t_0 + T]$, the system
\begin{equation}\label{ODE dot u}
    \left \{
    \begin{aligned}
      \Dot{\gamma}(t) & = k_{\epsilon}(t, \gamma(t)) u(t)  \\
      \Dot{u}(t) & \in \Proj_{u(t)}^{\perp} \Xi_\epsilon(t)
    \end{aligned}
		\right.
\end{equation}
where, for $x\in \mathbbm S^{d-1}$ and $A \subset \mathbbm R^d$, $\Proj_x^{\perp} A$ denotes the projection of the vectors of $A$ onto the tangent space of $\mathbbm S^{d-1}$ at $x$, defined by $\Proj_x^\perp A = \{a - (x \cdot a) x \suchthat a \in A\}$.
\end{enumerate}
\end{proposition}

\begin{proof}
Let us first prove \ref{item:p-is-nonzero}. By Lemma~\ref{lemm:pre-Gronwall-for-p}, we have that $\abs{\dot p(t)} \leq \abs{p(t)} L_\epsilon$, implying that, for every $t_1, t_2\in [t_0, t_0 + T]$, we have
\[
\abs*{p(t_2)}\le \abs*{p(t_1)} + L_\epsilon \int_{\min\{t_1, t_2\}}^{\max\{t_1, t_2\}} \abs*{p(s)} \diff s.
\]
Hence, by Gr\"onwall's inequality, for every $t_1, t_2\in [t_0, t_0 + T]$, we have
\[
\abs*{p(t_2)} \le \abs*{p_{\epsilon}(t_1)} e^{L_\epsilon \abs*{t_2 - t_1}}.
\]
If there exists $t_1$ such that $p(t_1)=0$, then $p(t)=0$ for every $t \in [t_0,t_0+T_{\epsilon}]$. Letting $\lambda \geq 0$ be as in the statement of Proposition~\ref{Conse Pontryagin}, we have, by Proposition~\ref{Conse Pontryagin}\ref{PMP:final-Hamiltonian}, that $\lambda = 0$, which is a contradiction according to Proposition~\ref{Conse Pontryagin}\ref{PMP:nontrivial}. Therefore $\abs*{p(t)}\neq 0$ for every $t\in [t_0,t_0+T]$. We then deduce from Proposition~\ref{Conse Pontryagin}\ref{PMP:optimal-control} that $u(t) = \frac{p(t)}{\abs{p(t)}}$ for almost every $t \in [t_0, t_0 + T]$, and the conclusion for every $t \in [t_0, t_0 + T]$ holds by redefining $u$ in a set of measure zero. Hence \ref{item:p-is-nonzero} is proved.

Note that, for every $t \in [t_0, t_0 + T]$, we have $\xi \in \Xi_\epsilon(t)$ if and only if
\[
\frac{\xi}{\abs{p(t)}} \in \conv\left\{\zeta \in \mathbbm R^d \suchthat (\zeta, p(t)) \in N_{G_\epsilon(t)}(\gamma(t), \dot\gamma(t))\right\}.
\]
In particular, \ref{item:Xi} follows as a consequence of \ref{item:p-is-nonzero} and Lemma~\ref{lemm:pre-Gronwall-for-p}. Notice also that, by Proposition~\ref{Conse Pontryagin}\ref{PMP:costate-equation}, we have $\dot p(t) \in \abs{p(t)} \Xi_\epsilon(t)$ for a.e.\ $t \in [t_0, t_0 + T]$.

Let us now prove \ref{item:system-on-gamma-u}. Note that, by \ref{item:p-is-nonzero}, $u(t) \in \mathbbm S^{d-1}$ for every $t \in [t_0, t_0 + T]$ and, since $p$ is absolutely continuous, $u$ is also absolutely continuous. Let $\xi: [t_0, t_0 + T] \to \mathbbm R^d$ be a measurable function with $\xi(t) \in \Xi_\epsilon(t)$ and $\dot p(t) = \abs{p(t)}\xi(t)$ for a.e.\ $t \in [t_0, t_0 + T]$. We compute, for a.e.\ $t \in [t_0, t_0 + T]$,
\[
\Dot{u}(t) = \frac{\Dot{p}(t) \abs*{p(t)} - \frac{p(t) \cdot \Dot{p}(t)}{\abs*{p(t)}} p(t)}{\abs*{p(t)}^2} = \xi(t) - (u(t) \cdot \xi(t)) u(t),
\]
which yields the differential inclusion for $u$ in system \eqref{ODE dot u}. By \ref{item:Xi}, we deduce that $\abs{\dot u(t)} \leq L_\epsilon$ for a.e.\ $t \in [t_0, t_0 + T]$, yielding that $u$ is $L_\epsilon$-Lipschitz continuous on $[t_0, t_0 + T]$. The first equation in \eqref{ODE dot u} is simply \eqref{eq:penalized-control-system}, and its right-hand side is continuous on $[t_0, t_0 + T]$, showing that $\gamma \in \mathbf{C}^{1}([t_0, t_0 + T]; \mathbbm R^d)$.
\end{proof}

We will need in the sequel the following technical lemma.

\begin{lemma}
\label{lemm:technical}
Consider the optimal control problem $\OCP_\epsilon(k_{\epsilon})$ under assumptions \HypoOmega{}, \ref{HypoOCP-k-Bound}, and \ref{HypoOCP-k-Lip}, and with $k_\epsilon$ defined from $k$ as in \eqref{eq:defi-k-epsilon}. Let $(t_0,x_0)$, $\gamma$, $T$, and $p$ be as in the statement of Proposition~\ref{Conse Pontryagin} and $L > 0$ be the Lipschitz constant of $\mathbbm R^d \ni x \mapsto k(t, x) \in [K_{\min}, K_{\max}]$, which is independent of $t$. For $t \in [t_0, t_0 + T]$, define $\Xi_\epsilon(t)$ as in the statement of Proposition~\ref{prop:consequences-of-Pontryagin-for-optimal-trajectories-of-ocp-epsilon}, i.e.,
\[
\Xi_\epsilon(t) = \conv\left\{\xi \in \mathbbm R^d \suchthat (\xi \abs{p(t)}, p(t)) \in N_{G_\epsilon(t)}(\gamma(t), \dot\gamma(t))\right\},
\]
where $G_\epsilon(t)$ is given by \eqref{eq:PMP:defi-G}. Finally, let $W$ be a neighborhood of $\partial\Omega$ in which the signed distance $d^{\pm}_{\partial\Omega}$ belongs to $\mathbf C^{1, 1}(W; \mathbb R)$, which exists thanks to Proposition~\ref{prop:OmegaC11}\ref{item:W-dpm}.

For almost every $t \in [t_0, t_0 + T]$, if $\gamma(t) \in W$ and $\gamma(t) \notin \bar\Omega$, then, for every $\xi \in \Xi_\epsilon(t)$, we have
\begin{equation}
\label{eq:tech-lemma}
\abs*{\xi \cdot \nabla d_{\partial\Omega}^\pm(\gamma(t)) - \frac{1}{\epsilon} k(t, \gamma(t))} \leq \left(1 - \frac{1}{\epsilon} d_\Omega(\gamma(t))\right) L.
\end{equation}
\end{lemma}

\begin{proof}
Fix $t \in [t_0, t_0 + T]$ at which $\gamma$ is differentiable and such that $\gamma(t) \in W$ and $\gamma(t) \notin \bar\Omega$. Since the set of $\xi \in \mathbbm R^d$ such that \eqref{eq:tech-lemma} holds is convex, it suffices to show that, for every $\xi \in \mathbbm R^d$, if $(\xi \abs{p(t)}, p(t)) \in N_{G_\epsilon(t)}(\gamma(t), \dot\gamma(t))$, then \eqref{eq:tech-lemma} holds, where $G_\epsilon(t)$ is given by \eqref{eq:PMP:defi-G}. Fix $\xi \in \mathbbm R^d$ such that $(\xi \abs{p(t)}, p(t)) \in N_{G_\epsilon(t)}(\gamma(t), \dot\gamma(t))$. Let $u$ be an optimal control associated with $\gamma$ and recall that, by Proposition~\ref{prop:consequences-of-Pontryagin-for-optimal-trajectories-of-ocp-epsilon}, $u$ is Lipschitz continuous.

By definition of $N_{G_\epsilon(t)}(\gamma(t), \dot\gamma(t))$, there exist sequences $(x_n)_{n \in \mathbbm N}$ in $\mathbbm R^d$, $(u_n)_{n \in \mathbbm N}$ in $\bar B$, $(\zeta_n)_{n \in \mathbbm N}$ in $\mathbbm R^d$, and $(p_n)_{n \in \mathbbm N}$ in $\mathbbm R^d$ such that, as $n \to +\infty$, we have $x_n \to \gamma(t)$, $k_\epsilon(t, x_n) u_n \to k_\epsilon(t, \gamma(t)) u(t)$, $\zeta_n \to \xi \abs{p(t)}$, and $p_n \to p(t)$, and in addition $(\zeta_n, p_n) \in N_{G_\epsilon(t)}^{\mathrm P}(x_n, k_\epsilon(t, x_n) u_n)$ for every $n \in \mathbbm N$. By Proposition~\ref{prop:consequences-of-Pontryagin-for-optimal-trajectories-of-ocp-epsilon}\ref{item:p-is-nonzero}, we have $\abs{p(t)} \neq 0$, and we then define $\xi_n = \frac{\zeta_n}{\abs{p(t)}}$ for $n \in \mathbbm N$. Then, as $n \to +\infty$, we have $\xi_n \to \xi$ and, since $k_\epsilon$ is continuous and\footnote{Recall that, by Proposition~\ref{prop:opt-epsilon-remains-close}, we have $d_\Omega(\gamma(t)) < \epsilon$.} $k_\epsilon(t, \gamma(t)) > 0$, we also have $u_n \to u(t)$.

For every $n \in \mathbbm N$, since $(\zeta_n, p_n) \in N_{G_\epsilon(t)}^{\mathrm P}(x_n, k_\epsilon(t, x_n) u_n)$, there exists $M_n > 0$ such that
\begin{multline}
\label{eq:estim-xin-proximal}
(\xi_n \abs{p(t)}, p_n) \cdot (y - x_n, k_\epsilon(t, y) w - k_\epsilon(t, x_n) u_n) \\
 \leq M_n \left[\abs*{y - x_n}^2 + \abs*{k_\epsilon(t, y) w - k_\epsilon(t, x_n) u_n}^2\right]
\end{multline}
for every $y \in \mathbbm R^d$ and $w \in \bar B$. Let $(\alpha_n)_{n \in \mathbbm N}$ be a sequence of positive real numbers such that $\alpha_n \to 0$ and $M_n \alpha_n \to 0$ as $n \to +\infty$. Since $\gamma(t) \in W$, $\gamma(t) \notin \bar\Omega$, and $d_\Omega(\gamma(t)) < \epsilon$, for $n$ large enough, we have $x_n \in W$, $x_n \notin \bar\Omega$, and $d_\Omega(x_n) < \epsilon$. In particular, by Proposition~\ref{prop:OmegaC11}\ref{item:W-dpm}, $\nabla d_{\partial\Omega}^\pm (x_n)$ is well-defined and $\abs{\nabla d_{\partial\Omega}^\pm (x_n)} = 1$. For $n \in \mathbbm N$, apply \eqref{eq:estim-xin-proximal} with $y = x_n + \alpha_n \nabla d_{\partial\Omega}^\pm (x_n)$ and $w = u_n$. Then
\[
\alpha_n \abs{p(t)} \xi_n \cdot \nabla d_{\partial\Omega}^\pm (x_n) + p_n \cdot u_n \left[k_\epsilon(t, x_n + \alpha_n \nabla d_{\partial\Omega}^\pm (x_n)) - k_\epsilon(t, x_n)\right] \leq M_n \alpha_n^2 (1 + L_\epsilon^2),
\]
where $L_\epsilon$ is the Lipschitz constant of $\mathbbm R^d \ni x \mapsto k_\epsilon(t, x) \in \mathbbm R_+$, which is independent of $t$. Dividing by $\alpha_n$, we get
\begin{equation}
\label{eq:monster-1}
\abs{p(t)} \xi_n \cdot \nabla d_{\partial\Omega}^\pm (x_n) + p_n \cdot u_n \frac{k_\epsilon(t, x_n + \alpha_n \nabla d_{\partial\Omega}^\pm (x_n)) - k_\epsilon(t, x_n)}{\alpha_n} \leq M_n \alpha_n (1 + L_\epsilon^2).
\end{equation}
Since $\nabla d_{\partial\Omega}^\pm$ is continuous in $W$, we have, as $n \to +\infty$, that $\xi_n \cdot \nabla d_{\partial\Omega}^\pm (x_n) \to \xi \cdot \nabla d_{\partial\Omega}^\pm (\gamma(t))$, $p_n \cdot u_n \to p(t) \cdot u(t) = \abs{p(t)}$, and, by construction, $M_n \alpha_n \to 0$.

Let us denote for simplicity $z_n = x_n + \alpha_n \nabla d_{\partial\Omega}^\pm (x_n)$. Since $\alpha_n \to 0$ as $n \to +\infty$ and $\gamma(t) \in W$, $\gamma(t) \notin \bar\Omega$, and $d_\Omega(\gamma(t)) < \epsilon$, we deduce that, for $n$ large enough, $x_n \in W$, $z_n \in W$, $x_n \notin \bar\Omega$, $z_n \notin \bar\Omega$, $d_\Omega(x_n) < \epsilon$, and $d_\Omega(z_n) < \epsilon$ for $n$ large enough. Using those facts and \eqref{eq:defi-k-epsilon}, we compute
\begin{align*}
\frac{1}{\alpha_n}\left[k_\epsilon(t, z_n) - k_\epsilon(t, x_n)\right] & = \frac{1}{\alpha_n}\left[k(t, z_n) \left(1 - \frac{1}{\epsilon} d_{\partial\Omega}^\pm(z_n)\right) - k(t, x_n) \left(1 - \frac{1}{\epsilon} d_{\partial\Omega}^\pm(x_n)\right)\right] \displaybreak[0] \\
& = \frac{1}{\epsilon} k(t, z_n) \frac{d_{\partial\Omega}^\pm(x_n) - d_{\partial\Omega}^{\pm}(z_n)}{\alpha_n} + \beta_n \left(1 - \frac{1}{\epsilon} d_{\partial\Omega}^\pm(x_n)\right),
\end{align*}
where
\[
\beta_n = \frac{k(t, z_n) - k(t, x_n)}{\alpha_n}.
\]
Note that, since $k$ is $L$-Lipschitz continuous in its second argument, we have $\abs{\beta_n} \leq L$ for $n$ large enough. In particular, up to extracting a subsequence, there exists $\beta_\ast \in [-L, L]$ such that $\beta_n \to \beta_\ast$ as $n \to +\infty$. Moreover, since $d_{\partial\Omega}^\pm \in \mathbf{C}^{1, 1}(W; \mathbbm R)$, we have $\frac{d_{\partial\Omega}^\pm(x_n) - d_{\partial\Omega}^{\pm}(z_n)}{\alpha_n} \to -\abs*{\nabla d_{\partial\Omega}^\pm(\gamma(t))}^2 = -1$ as $n \to +\infty$. Hence, letting $n \to +\infty$ in \eqref{eq:monster-1} and dividing by $\abs{p(t)}$, we deduce that
\begin{equation}
\label{eq:monster-2}
\xi \cdot \nabla d_{\partial\Omega}^\pm (\gamma(t)) - \frac{1}{\epsilon} k(t, \gamma(t)) + \beta_\ast \left(1 - \frac{1}{\epsilon} d_{\partial\Omega}^\pm(\gamma(t))\right) \leq 0.
\end{equation}

Performing the same computations as above from \eqref{eq:estim-xin-proximal} but now taking $y = x_n - \alpha_n \nabla d_{\partial\Omega}^\pm (x_n)$ and $w = u_n$, we get that
\[
-\xi \cdot \nabla d_{\partial\Omega}^\pm (\gamma(t)) + \frac{1}{\epsilon} k(t, \gamma(t)) + \tilde\beta_\ast \left(1 - \frac{1}{\epsilon} d_{\partial\Omega}^\pm(\gamma(t))\right) \leq 0
\]
for some $\tilde\beta_\ast \in [-L, L]$. Together with \eqref{eq:monster-2}, this yields \eqref{eq:tech-lemma}, as required.
\end{proof}

Proposition~\ref{prop:opt-epsilon-remains-close} stated that, for every $\epsilon > 0$, optimal trajectories for $\OCP_\epsilon(k_\epsilon)$ always remain in an $\epsilon$ neighborhood of $\bar\Omega$. Using the previous results, we can now be more precise and show that, if $\epsilon > 0$ is small enough, optimal trajectories for $\OCP_\epsilon(k_\epsilon)$ starting at $(t_0, x_0) \in \mathbbm R_+ \times \bar\Omega$ never leave $\bar\Omega$.

\begin{proposition}
\label{gamma-eps in Omeg}
Consider the optimal control problem $\OCP_\epsilon(k_\epsilon)$ under assumptions \HypoOmega{}, \ref{HypoOCP-k-Bound}, and \ref{HypoOCP-k-Lip} and $k_\epsilon$ defined from $k$ as in \eqref{eq:defi-k-epsilon}. There exists $\epsilon_0 > 0$ such that, for every $\epsilon \in (0, \epsilon_0)$, $(t_0, x_0) \in \mathbbm R_+ \times \mathbbm R^d$, and $\gamma \in \Opt_\epsilon(k_\epsilon, t_0, x_0)$, we have $d_\Omega(\gamma(t)) \leq d_\Omega(x_0)$ for every $t \in \mathbbm R_+$. In particular, if $x_0 \in \bar\Omega$, then $\gamma(t) \in \bar\Omega$ for every $t \in \mathbbm R_+$.
\end{proposition}

\begin{proof}
Take $(t_0, x_0) \in \mathbbm R_+ \times \mathbbm R^d$. Notice first that, if $x_0 \in \Gamma$, then $\Opt_\epsilon(k_\epsilon, t_0, x_0)$ is equal to the singleton containing only the trajectory that remains in $x_0$ for all times. We thus assume, for the rest of the proof, that $x_0 \in \bar\Omega \setminus \Gamma$.

Let $W$ be as in the statement of Proposition~\ref{prop:OmegaC11}\ref{item:W-dpm} and take $\epsilon_\ast > 0$ such that $\{x \in \mathbbm R^d \suchthat d_{\partial\Omega}(x) < \epsilon_\ast\} \subset W$. Let $C$ be a Lipschitz constant of $\nabla d_{\partial \Omega}^{\pm}$ on $W$ and $L$, $K_{\max}$, and $K_{\min}$ be as in \ref{HypoOCP-k-Bound} and \ref{HypoOCP-k-Lip}. Let $\epsilon_0 = \min\left\{\epsilon_\ast, \frac{K_{\min}}{L + C K_{\max}}\right\}$ and take $\epsilon \in (0,\epsilon_0)$.

Fix $\gamma \in \Opt_\epsilon(k_\epsilon, t_0, x_0)$ and let $T = \varphi_\epsilon(t_0, x_0)$. Recall that, by Proposition~\ref{prop:opt-epsilon-remains-close}, $d_\Omega(\gamma(t)) < \epsilon$ for every $t \in \mathbbm R_+$. Let $p: [t_0, t_0 + T] \to \mathbbm R^d$ be the costate associated with $\gamma$, whose existence is asserted in Proposition~\ref{Conse Pontryagin}. By Propositions~\ref{Conse Pontryagin}\ref{PMP:costate-equation} and \ref{prop:consequences-of-Pontryagin-for-optimal-trajectories-of-ocp-epsilon}\ref{item:p-is-nonzero}, there exists a measurable function $\xi: [t_0, t_0 + T] \to \mathbbm R^d$ such that $\xi(t) \in \Xi_\epsilon(t)$ and $\dot p(t) = \abs{p(t)} \xi(t)$ for a.e.\ $t \in [t_0, t_0 + T]$, where $\Xi_\epsilon(t)$ is defined in \eqref{eq:defi-Xi-epsilon}.

By definition of optimal trajectories, we have that $\gamma(t) = x_0$ for every $t \in [0, t_0]$ and $\gamma(t) = \gamma(t_0 + T) \in \Gamma \subset \bar\Omega$ for every $t \in [t_0 + T, +\infty)$, and hence $d_\Omega(\gamma(t)) \leq d_\Omega(x_0)$ for every $t \in [0, t_0] \cup [t_0 + T, +\infty)$. Assume, to obtain a contradiction, that there exist $a, b \in [t_0, t_0 + T]$ such that $a < b$, $d_\Omega(\gamma(t)) > d_\Omega(x_0)$ for $t \in (a,b)$ and $d_\Omega(\gamma(t)) = d_\Omega(x_0)$ for $t\in \{a,b\}$. In particular, $d_\Omega(\gamma(t)) > 0$ for $t \in (a, b)$, implying that $\gamma(t) \notin \bar\Omega$ for $t \in (a, b)$ and $\gamma(t) \notin \Omega$ for $t \in [a, b]$.

Recall that, by Proposition~\ref{prop:opt-epsilon-remains-close}, $d_{\Omega}(\gamma(t)) < \epsilon$ for every $t \in \mathbbm R_+$, and thus, in particular, $\gamma(t) \in W$ for every $t \in [a, b]$. Hence, by Propositions~\ref{prop:OmegaC11} and \ref{prop:consequences-of-Pontryagin-for-optimal-trajectories-of-ocp-epsilon}\ref{item:system-on-gamma-u}, the map $t \mapsto d_{\partial \Omega}^{\pm}(\gamma(t))$ is differentiable on $(t_0, t_0 + T)$, $d_{\partial\Omega}^\pm(\gamma(t)) > d_\Omega(x_0)$ for $t \in (a,b)$, and $d_{\partial\Omega}^\pm(\gamma(t)) = d_\Omega(x_0)$ for $t \in \{a,b\}$. Thus, its derivative is nonnegative at $a$ and nonpositive at $b$, i.e.,
\[
\Dot{\gamma}(a) \cdot \nabla d_{\partial \Omega}^{\pm}(\gamma(a)) \ge 0 \qquad \text{ and } \qquad \Dot{\gamma}(b) \cdot \nabla d_{\partial \Omega}^{\pm}(\gamma(b)) \le 0.
\]
Since, by Propositions~\ref{prop:opt-epsilon-remains-close} and \ref{prop:consequences-of-Pontryagin-for-optimal-trajectories-of-ocp-epsilon}, we have $\Dot{\gamma}(t) = k_{\epsilon}(t, \gamma(t))\frac{p(t)}{\abs*{p(t)}}$ and $\frac{k_{\epsilon}(t, \gamma(t))}{\abs*{p(t)}} > 0$ for every $t \in [t_0, t_0 + T]$, we deduce that
\begin{equation}\label{contradicts p-epsilon}
    p(a)\cdot \nabla d_{\partial \Omega}^{\pm}(\gamma(a)) \ge 0 \qquad \text{ and } \qquad
    p(b)\cdot \nabla d_{\partial \Omega}^{\pm}(\gamma(b)) \le 0.
\end{equation}
Consider now the map $\alpha:t\mapsto p(t)\cdot \nabla d_{\partial \Omega}^{\pm}(\gamma(t))$. Since $\gamma(t) \in W$ for every $t \in [a, b]$ and $d_{\partial\Omega}^\pm \in \mathbf C^{1, 1}(W; \mathbbm R)$, $\alpha$ is absolutely continuous on $[a, b]$ and, for a.e.\ $t \in (a, b)$, we have, using Lemma~\ref{lemm:technical} and the fact that $\nabla d_{\partial\Omega}^\pm \circ \gamma$ is $C K_{\max}$-Lipschitz continuous in $[a, b]$, that
\begin{align*}
    \dot\alpha(t) & = \Dot{p}(t)\cdot \nabla d_{\partial \Omega}^{\pm}(\gamma(t))+p(t)\cdot \frac{\diff\left[\nabla d_{\partial \Omega}^{\pm} \circ \gamma\right]}{\diff t}(t) \displaybreak[0] \\
		& = \abs{p(t)} \xi(t) \cdot \nabla d_{\partial \Omega}^{\pm}(\gamma(t))+p(t)\cdot \frac{\diff\left[\nabla d_{\partial \Omega}^{\pm} \circ \gamma\right]}{\diff t}(t) \displaybreak[0] \\
		& \geq \frac{1}{\epsilon} \abs{p(t)} k(t, \gamma(t)) - \left(1 - \frac{1}{\epsilon} d_\Omega(\gamma(t))\right) L \abs{p(t)} - C K_{\max} \abs{p(t)} \displaybreak[0] \\
    & \geq \abs{p(t)} \left(\frac{K_{\min}}{\epsilon} - L - C K_{\max}\right) > 0,
\end{align*}
which contradicts \eqref{contradicts p-epsilon}.
\end{proof}

We are now in position to show the main result of this section.

\begin{theorem}
\label{thm:Opt-eps-equals-Opt}
Consider the optimal control problem $\OCP(k)$ under assumptions \HypoOmega{}, \ref{HypoOCP-k-Bound}, and \ref{HypoOCP-k-Lip}, as well as the problem $\OCP_\epsilon(k_\epsilon)$ with $k_\epsilon$ defined from $k$ as in \eqref{eq:defi-k-epsilon}. There exists $\epsilon_0 > 0$ such that, for every $\epsilon \in (0, \epsilon_0)$ and $(t_0, x_0) \in \mathbbm R_+ \times \bar\Omega$, we have $\Opt_\epsilon(k_\epsilon, t_0, x_0) = \Opt(k, t_0, x_0)$ and $\varphi_\epsilon(t_0, x_0) = \varphi(t_0, x_0)$.
\end{theorem}

\begin{proof}
Let $\epsilon_0 > 0$ be as in the statement of Proposition~\ref{gamma-eps in Omeg}, $\epsilon \in (0, \epsilon_0)$, and $(t_0,x_0)\in \mathbbm R\times \bar{\Omega}$. Let $\gamma_{\epsilon}\in \Opt_\epsilon(k_\epsilon, t_0, x_0)$ and $\gamma \in \Opt(k, t_0, x_0)$, and denote $T_\epsilon = \varphi_\epsilon(t_0, x_0)$ and $T = \varphi(t_0, x_0)$. Clearly, $\gamma \in \Adm_\epsilon(k_\epsilon)$, and thus $T_{\epsilon} \le T$. On the other hand, thanks to Proposition~\ref{gamma-eps in Omeg}, $\gamma_{\epsilon} \in \Adm(k)$, which leads us to $T \le T_{\epsilon}$. Hence $T=T_{\epsilon}$, which shows that we have both $\gamma_\epsilon \in \Opt(k, t_0, x_0)$ and $\gamma \in \Opt_\epsilon(k_\epsilon, t_0, x_0)$, as required.
\end{proof}

Thanks to Theorem~\ref{thm:Opt-eps-equals-Opt}, all properties established in this section for optimal trajectories of $\OCP_\epsilon(k_\epsilon)$ also hold for optimal trajectories of $\OCP(k)$. The next statement, whose proof is an immediate consequence of Proposition~\ref{prop:consequences-of-Pontryagin-for-optimal-trajectories-of-ocp-epsilon}\ref{item:system-on-gamma-u} and Theorem~\ref{thm:Opt-eps-equals-Opt}, gathers the properties of optimal trajectories of $\OCP(k)$ that will be used in the sequel.

\begin{corollary}
\label{coro:smooth}
Consider the optimal control problem $\OCP(k)$ under assumptions \HypoOmega{}, \ref{HypoOCP-k-Bound}, and \ref{HypoOCP-k-Lip}. Let $(t_0, x_0) \in \mathbbm R_+ \times \bar\Omega$, $\gamma \in \Opt(k, t_0, x_0)$, $T = \varphi(t_0, x_0)$, and $u: \mathbbm R_+ \to \bar B$ be an optimal control associated with $\gamma$. Then, up to redefining $u$ in a set of Lebesgue measure zero, we have $\gamma \in \mathbf C^{1}([t_0, t_0 + T]; \bar\Omega)$ and $u \in \Lip([t_0, t_0 + T]; \mathbbm S^{d-1})$.
\end{corollary}

Before concluding this section, we also provide the following result on the local Lipschitz continuity of $\varphi_\epsilon$, which can be obtained as a consequence of Proposition~\ref{gamma-eps in Omeg}. For its statement, we introduce, for $\epsilon > 0$, the set
\[
\Omega_\epsilon = \{x \in \mathbbm R^d \suchthat d_\Omega(x) < \epsilon\}.
\]

\begin{proposition}
\label{prop:varphi-eps-loc-Lip}
Consider the optimal control problem $\OCP_\epsilon(k_\epsilon)$ and its value function $\varphi_\epsilon$ under assumptions \HypoOmega{}, \ref{HypoOCP-k-Bound}, and \ref{HypoOCP-k-Lip} and $k_\epsilon$ defined from $k$ as in \eqref{eq:defi-k-epsilon}. Then there exists $\epsilon_0 > 0$ such that, for every $\epsilon \in (0, \epsilon_0)$, the value function $\varphi_\epsilon: \mathbbm R_+ \times \Omega_\epsilon \to \mathbbm R_+$ of $\OCP_\epsilon(k_\epsilon)$ is locally Lipschitz continuous. More precisely, for every $\eta \in [0, \epsilon)$, $\varphi_\epsilon$ is Lipschitz continuous on $\mathbbm R_+ \times \bar\Omega_\eta$.
\end{proposition}

\begin{proof}
Let $\epsilon_0 > 0$ be as in the statement of Proposition~\ref{gamma-eps in Omeg} and fix $\epsilon \in (0, \epsilon_0)$ and $\eta \in [0, \epsilon)$. We consider the optimal control problem $\widehat\OCP$ defined as follows: given $(t_0, x_0) \in \mathbbm R_+ \times \bar\Omega_\eta$, minimize $\tau(t_0, \gamma)$ over all Lipschitz continuous curves $\gamma: \mathbbm R_+ \to \bar\Omega_\eta$ satisfying \eqref{eq:penalized-control-system} for some measurable function $u: \mathbbm R_+ \to \bar B$ and $\gamma(t_0) = x_0$. We denote by $\widehat\varphi: \mathbbm R_+ \times \bar\Omega_\eta \to \mathbbm R_+$ the value function of $\widehat\OCP$ and by $\widehat\Opt(k_\epsilon, t_0, x_0)$ the set of optimal trajectories for $(t_0, x_0)$, with the convention that $\gamma \in \widehat\Opt(k_\epsilon, t_0, x_0)$ must satisfy $\gamma(t) = x_0$ for $t \in [0, t_0]$ and $\gamma(t) = \gamma(t_0 + \widehat\varphi(t_0, x_0))$ for $t \in [t_0 + \widehat\varphi(t_0, x_0), +\infty)$.

Note that $\widehat\OCP$ is a minimal-time optimal control problem with dynamics defined by $k_\epsilon$ and with the state constraint $\gamma(t) \in \bar\Omega_\eta$ for every $t \geq 0$. In particular, $\widehat\OCP$ coincides with the optimal control problem $\OCP(k)$ if one replaces $\Omega$ by $\Omega_\eta$ and $k$ by the restriction of $k_\epsilon$ to $\mathbbm R_+ \times \bar\Omega_\eta$, and assumptions \HypoOmega{}, \ref{HypoOCP-k-Bound}, and \ref{HypoOCP-k-Lip} are satisfied for $\widehat\OCP$, with $K_{\min}$ in \ref{HypoOCP-k-Bound} replaced by $\widehat K_{\min} = \left(1 - \frac{\eta}{\epsilon}\right)K_{\min}$ and $L$ in \ref{HypoOCP-k-Lip} replaced by $\widehat L = L_\epsilon = L + \frac{K_{\max}}{\epsilon}$. Hence, all properties of Section~\ref{sec:prelim} apply to $\widehat\OCP$ and thus, by Proposition~\ref{varphi is Lipschitz}, $\widehat\varphi$ is Lipschitz continuous on $\mathbbm R_+ \times \bar\Omega_\eta$.

In order to conclude, it suffices to show that $\varphi_\epsilon$ and $\widehat\varphi$ coincide on this set. For this purpose, we proceed as in the proof of Theorem~\ref{thm:Opt-eps-equals-Opt}. Let $(t_0,x_0) \in \mathbbm R\times \bar{\Omega}_\eta$. Let $\gamma_{\epsilon}\in \Opt_\epsilon(k_\epsilon, t_0, x_0)$ and $\widehat\gamma \in \widehat\Opt(k_\epsilon, t_0, x_0)$, and denote $T_\epsilon = \varphi_\epsilon(t_0, x_0)$ and $\widehat T = \widehat\varphi(t_0, x_0)$. Clearly, $\widehat\gamma \in \Adm_\epsilon(k_\epsilon)$, and thus $T_{\epsilon} \le \widehat T$. On the other hand, thanks to Proposition~\ref{gamma-eps in Omeg}, we have $d_\Omega(\gamma_\epsilon(t)) \leq d_\Omega(x_0) \leq \eta$ for every $t \geq 0$, showing that $\gamma_\epsilon$ takes values in $\bar\Omega_\eta$. Hence $\gamma_\epsilon$ is an admissible trajectory for the optimal control problem $\widehat\OCP$, showing that $\widehat T \le T_{\epsilon}$. Hence $\widehat T=T_{\epsilon}$, as required.
\end{proof}

\subsection{Boundary condition of the value function}
\label{sec:boundary-varphi}

Proposition~\ref{PropOCP-HJ} in Section~\ref{sec:prelim} asserts that $\varphi$ is a viscosity solution of the Hamilton--Jacobi equation \eqref{H-J equation} and also provides its boundary condition on $\mathbbm R_+ \times \Gamma$. The only information provided in Proposition~\ref{PropOCP-HJ} on $\varphi$ on the part of the boundary $\mathbbm R_+ \times (\partial\Omega \setminus \Gamma)$ is that $\varphi$ is a viscosity supersolution there. The main result of this section, Theorem~\ref{Thm viscosity boundary cond}, provides additional information on $\varphi$ on $\mathbbm R_+ \times (\partial\Omega \setminus \Gamma)$. We recall that, given $x \in \partial\Omega$, $\mathbf n(x)$ denotes the outward unit normal vector to $\bar\Omega$ at $x$.

\begin{theorem}\label{Thm viscosity boundary cond}
Consider the optimal control problem $\OCP(k)$ and its value function $\varphi$ under assumptions \HypoOmega{}, \ref{HypoOCP-k-Bound}, and \ref{HypoOCP-k-Lip}. Then $\varphi$ satisfies $\nabla \varphi(t, x) \cdot \mathbf n(x) \ge 0$ in the viscosity supersolution sense for every $(t, x) \in \mathbbm R_+ \times (\partial\Omega \setminus \Gamma)$.
\end{theorem}

\begin{proof}
We consider in this proof that $k$ is extended to a continuous function, still denoted by $k$, defined in $[-1, +\infty) \times \mathbbm R^d$ which is Lipschitz continuous with respect to its second variable, uniformly with respect to the first one. Note that, up to considering the optimal control problem $\OCP(k(\cdot - 1, \cdot))$, all of the previous results on minimal-time optimal control problems apply to $k$ with the time domain $\mathbbm R_+$ replaced by $[-1, +\infty)$

Fix $(t_0, x_0) \in \mathbbm R_+ \times (\partial\Omega \setminus \Gamma)$ and let $\xi \in \mathbf C^1(V; \mathbbm R)$ be such that $\xi(t_0, x_0) = \varphi(t_0, x_0)$ and $\xi(t, x) \leq \varphi(t, x)$ for every $(t, x) \in V$, where $V$ is a neighborhood of $(t_0, x_0)$ in $[-1, +\infty) \times \bar\Omega$. Assume, to obtain a contradiction, that $\nabla \xi(t_0, x_0) \cdot \mathbf n(x_0) < 0$. In particular, we have $\nabla\xi(t_0, x_0) \neq 0$.

Let $\gamma \in \Opt(k, t_0, x_0)$, $u$ be an optimal control for $\gamma$, and $T = \varphi(t_0, x_0)$. Since $x_0 \notin \Gamma$, we have $T > 0$ and, by Corollary~\ref{coro:smooth}, we have $\gamma \in \mathbf C^1([t_0, t_0 + T]; \bar\Omega)$ and $u \in \Lip([t_0, t_0 + T]; \mathbbm S^{d-1})$.

Let $\sigma: [-1, +\infty) \to \mathbbm R^d$ denote the solution of the initial value problem
\[
\left\{
\begin{aligned}
\dot\sigma(t) & = - k(t, \sigma(t)) \frac{\nabla\xi(t_0, x_0)}{\abs{\nabla\xi(t_0, x_0)}}, \\
\sigma(t_0) & = x_0.
\end{aligned}
\right.
\]
Since $k$ is continuous, by construction, $\sigma \in \mathbf C^1([-1, +\infty), \mathbbm R^d)$. We claim that there exists $\epsilon \in (0, 1]$ such that $\sigma(t) \in \bar\Omega$ and $(t, \sigma(t)) \in V$ for every $t \in [t_0 - \epsilon, t_0]$. Indeed, let $W$ be as in Proposition~\ref{prop:OmegaC11}\ref{item:W-dpm} and notice that, since $x_0 \in \partial\Omega$, there exists $\epsilon_0 \in (0, 1]$ such that $\sigma(t) \in W$ for every $t \in [t_0 - \epsilon_0, t_0 + \epsilon_0]$. Moreover, up to reducing $\epsilon_0$, we also have $(t, \sigma(t)) \in V$ for every $t \in [t_0 - \epsilon_0, t_0 + \epsilon_0]$. Consider the map $\beta: t \mapsto d_{\partial\Omega}^{\pm}(\sigma(t))$, which is Lipschitz continuous on $[t_0 - \epsilon_0, t_0 + \epsilon_0]$. By differentiating, we have, for a.e.\ $t \in [t_0 - \epsilon_0, t_0 + \epsilon_0]$,
\begin{equation}
\label{eq:dot-beta}
    \dot\beta(t) = \Dot{\sigma}(t) \cdot \nabla d_{\partial\Omega}^{\pm}(\sigma(t))=-k(t, \sigma(t)) \nabla d_{\partial\Omega}^{\pm}(\sigma(t)) \cdot \frac{\nabla \xi(t_0, x_0)}{\abs*{\nabla \xi(t_0, x_0)}}.
\end{equation}
Recalling that $\mathbf n(x_0) = \nabla d_{\partial\Omega}^{\pm}(x_0)$ and using the continuity of $\nabla d_{\partial\Omega}$, we deduce from the fact that $\nabla \xi(t_0, x_0) \cdot \mathbf n(x_0) < 0$ that there exists $\epsilon \in (0, \epsilon_0]$ such that $\nabla d_{\partial\Omega}^{\pm}(\sigma(t)) \cdot \nabla \xi(t_0, x_0) < 0$ for every $t \in [t_0 - \epsilon, t_0 + \epsilon]$. Hence, by \eqref{eq:dot-beta}, we have $\dot\beta(t) > 0$ for a.e.\ $t \in [t_0 - \epsilon, t_0 + \epsilon]$. Since $\beta(t_0) = 0$, we thus have $\beta(t) \leq 0$ for $t \in [t_0 - \epsilon, t_0]$, which implies that $\sigma(t) \in \bar\Omega$ for every $t \in [t_0 - \epsilon, t_0]$. Since $\epsilon \in (0, \epsilon_0]$, we also have in particular that $(t, \sigma(t)) \in V$ for every $t \in [t_0 - \epsilon, t_0]$.

Define $\tilde{\gamma}:[-1, +\infty) \to \bar{\Omega}$ by
\begin{equation*}
\tilde\gamma(t) = \begin{dcases*}
\sigma(t_0 - \epsilon), & if $t \in [-1, t_0 - \epsilon)$, \\
\sigma(t), & if $t \in [t_0 - \epsilon, t_0)$, \\
\gamma(t), & if $t \geq t_0$.
\end{dcases*}
\end{equation*}
Hence $\tilde\gamma$ is an admissible trajectory for $k$ and thus, by Proposition~\ref{PropOCP-1}\ref{PropOCP-DPP}, for every $h \in [0, \epsilon]$, we have $\varphi(t_0, x_0) \ge \varphi(t_0 - h, \tilde{\gamma}(t_0 - h)) - h$, and thus $\xi(t_0, x_0) \ge \xi(t_0 - h, \tilde{\gamma}(t_0 - h)) - h$. Notice also that, for $h \in [0, \epsilon]$, we have $\tilde\gamma(t_0 - h) = \sigma(t_0 - h)$, yielding that $\xi(t_0, x_0) \ge \xi(t_0 - h, \sigma(t_0 - h)) - h$. Since $\xi$ and $\sigma$ are $\mathbf C^1$ functions, for $h \in [0, \epsilon]$, we have
\[
\xi(t_0 - h, \sigma(t_0 - h)) = \xi(t_0, x_0) - h \partial_t \xi(t_0, x_0) - h \nabla \xi(t_0, x_0) \cdot \dot\sigma(t_0) + o(h).
\]
Inserting this fact in the inequality $\xi(t_0, x_0) \ge \xi(t_0 - h, \sigma(t_0 - h)) - h$ for $h \in (0, \epsilon]$, using the definition of $\sigma$, dividing by $h$ and letting $h \to 0^+$, we deduce that
\begin{equation}\label{tilde gamma diff}
-\partial_t\xi(t_0, x_0) + k(t_0, x_0) \abs{\nabla \xi(t_0, x_0)} - 1 \le 0.
\end{equation}

Since $\gamma \in \Opt(k, t_0, x_0)$, we have, by Proposition~\ref{PropOCP-1}\ref{PropOCP-DPP}, that $\varphi(t_0, x_0) = \varphi(t_0 + h, \gamma(t_0 + h)) + h$ for every $h \in [0, T]$, and thus $\xi(t_0, x_0) \geq \xi(t_0 + h, \gamma(t_0 + h)) + h$. We also have, for $h \in [0, T]$,
\[
\xi(t_0 + h, \gamma(t_0 + h)) = \xi(t_0, x_0) + h \partial_t \xi(t_0, x_0) + h \nabla \xi(t_0, x_0) \cdot \dot\gamma(t_0^+) + o(h),
\]
where we use the fact that $\gamma \in \mathbf C^1([t_0, t_0 + T]; \bar\Omega)$ and $\dot\gamma(t_0^+)$ denotes the limit of $\dot\gamma(t)$ as $t \to t_0^+$. Using the fact that $\gamma$ satisfies the first equation of \eqref{General control sys}, we deduce as before, dividing by $h$ and letting $h \to 0^+$, that
\[
\partial_t \xi(t_0, x_0) + k(t_0, x_0) \nabla \xi(t_0, x_0) \cdot u(t_0) + 1 \leq 0.
\]
Adding with \eqref{tilde gamma diff}, we deduce that $\nabla \xi(t_0, x_0) \cdot u(t_0) + \abs{\nabla \xi(t_0, x_0)} \leq 0$ and, since $u(t_0) \in \bar B$, this implies that $u(t_0) = -\frac{\nabla \xi(t_0, x_0)}{\abs{\nabla \xi(t_0, x_0)}}$.

Consider now the function $\alpha: t \mapsto d^{\pm}_{\partial\Omega}(\gamma(t))$ defined on $[t_0, t_0 + T]$. Since $\gamma(t_0) = x_0 \in \partial\Omega$, there exists $\delta_0 \in (0, T]$ such that $\gamma(t) \in W$ for every $t \in [t_0, t_0 + \delta_0]$, and thus $\alpha$ is $\mathbf C^1$ on $[t_0, t_0 + \delta_0]$, with
\[
\dot\alpha(t) = \nabla d^{\pm}_{\partial\Omega}(\gamma(t)) \cdot \dot\gamma(t) = k(t, \gamma(t)) \nabla d^{\pm}_{\partial\Omega}(\gamma(t)) \cdot u(t).
\]
In particular, $\dot\alpha(t_0) = -k(t_0, x_0) \mathbf n(x_0) \cdot \frac{\nabla \xi(t_0, x_0)}{\abs{\nabla \xi(t_0, x_0)}} > 0$, showing that there exists $\delta \in (0, \delta_0]$ such that $\alpha(t) > \alpha(t_0) = 0$ for every $t \in (t_0, t_0 + \delta]$. From the definition of $\alpha$, this means that $\gamma(t) \notin \bar\Omega$ for $t \in (t_0, t_0 + \delta]$, which is a contradiction since $\gamma \in \Opt(k, t_0, x_0)$. This contradiction establishes that $\nabla\xi(t_0, x_0) \cdot \mathbf n(x_0) \geq 0$, as required.
\end{proof}

\begin{remark}
\label{remk:boundary-differentiable}
Under the assumptions of Theorem~\ref{Thm viscosity boundary cond}, if $(t_0, x_0) \in \mathbbm R_+ \times (\partial\Omega \setminus \Gamma)$ and $\varphi$ is differentiable at $(t_0, x_0)$, then the inequality $\nabla\varphi(t_0, x_0) \cdot \mathbf n(x_0) \geq 0$ holds in the classical sense.
\end{remark}

Proposition~\ref{PropOCP-HJ} asserts that the value function $\varphi$ is a viscosity solution of the Hamilton--Jacobi equation \eqref{H-J equation} in $\mathbbm R_+ \times (\Omega \setminus \Gamma)$ and thus, as a consequence of standard properties of viscosity solutions, \eqref{H-J equation} is satisfied in the classical sense at all points $(t, x) \in \mathbbm R_+ \times (\Omega \setminus \Gamma)$ at which $\varphi$ is differentiable. Thanks to Theorem~\ref{Thm viscosity boundary cond}, we can prove that \eqref{H-J equation} is actually satisfied at all points $(t, x) \in \mathbbm R_+ \times (\bar\Omega \setminus \Gamma)$ at which $\varphi$ is differentiable.

\begin{proposition}
\label{prop:HJ-satisfied-classical-sense}
Consider the optimal control problem $\OCP(k)$ and its value function $\varphi$ under assumptions \HypoOmega{}, \ref{HypoOCP-k-Bound}, and \ref{HypoOCP-k-Lip}. Let $(t_0, x_0) \in \mathbbm R_+ \times (\bar\Omega \setminus \Gamma)$ be such that $\varphi$ is differentiable at $(t_0, x_0)$. Then
\[
-\partial_t \varphi(t_0, x_0) + \abs{\nabla \varphi(t_0, x_0)} k(t_0, x_0) - 1 = 0.
\]
\end{proposition}

\begin{proof}
If $x_0 \in \Omega \setminus \Gamma$, this fact is a consequence of classical properties of viscosity solutions, and so we assume in the sequel that $x_0 \in \partial\Omega \setminus \Gamma$. By Remark~\ref{remk:boundary-differentiable}, we have that $\nabla\varphi(t_0, x_0) \cdot \mathbf n(x_0) \geq 0$. Moreover, by Proposition~\ref{PropOCP-2}\ref{PropOCP-LowerBound}, we have $\nabla\varphi(t_0, x_0) \neq 0$. Since, by Proposition~\ref{PropOCP-HJ}, $\varphi$ is a viscosity supersolution of \eqref{H-J equation} in $\mathbbm R_+ \times (\bar\Omega \setminus \Gamma)$, we have in particular that
\[
-\partial_t \varphi(t_0, x_0) + \abs{\nabla \varphi(t_0, x_0)} k(t_0, x_0) - 1 \geq 0.
\]
We are thus left to show the converse inequality.

Let $D$ be the set of unit vectors pointing to the inside of $\bar\Gamma$ at $x_0$, i.e., $D$ is the set of $u_0 \in \mathbbm S^{d-1}$ for which there exists $h_0 > 0$ such that $x_0 + h u_0 \in \bar\Omega$ for every $h \in [0, h_0]$. Notice that $\{u_0 \in \mathbbm S^{d-1} \suchthat u_0 \cdot \mathbf n(x_0) < 0\} \subset D \subset \{u_0 \in \mathbbm S^{d-1} \suchthat u_0 \cdot \mathbf n(x_0) \leq 0\}$ and thus, in particular, $\bar D = \{u_0 \in \mathbbm S^{d-1} \suchthat u_0 \cdot \mathbf n(x_0) \leq 0\}$. Since $\nabla\varphi(t_0, x_0) \cdot \mathbf n(x_0) \geq 0$, we deduce that $-\frac{\nabla\varphi(t_0, x_0)}{\abs{\nabla\varphi(t_0, x_0)}} \in \bar D$.

Let $u \in D$ and $\sigma \in \mathbf C^1(\mathbbm R_+; \mathbbm R^d)$ be the unique solution of
\begin{equation}\label{ODE traj 2}
    \left\{
    \begin{aligned}
      \Dot{\sigma}(t)& = k(t, \sigma(t)) u\\
      \sigma(t_0)&=x_0.
    \end{aligned} \right.
\end{equation}
Since $u \in D$, there exists $h_0 > 0$ such that $\sigma(t) \in \bar\Omega$ for every $t \in [t_0, t_0 + h_0]$. Let $\gamma \in \Adm(k)$ be defined by $\gamma(t) = x_0$ for $t \in [0, t_0]$, $\gamma(t) = \sigma(t)$ for $t \in [t_0, t_0 + h_0]$, and $\gamma(t) = \sigma(t_0 + h_0)$ for $t \geq t_0 + h_0$. By Proposition~\ref{PropOCP-1}\ref{PropOCP-DPP}, we have, for every $h \in [0, h_0]$,
\begin{equation}\label{admiss dynamic prog}
\varphi(t_0, x_0) \le h + \varphi(t_0 + h, \sigma(t_0+h)).
\end{equation}
One the other hand, notice that $\sigma(t_0 + h) = x_0 + h k(t_0, x_0) u + o(h)$ as $h \to 0$. Since $\varphi$ is differentiable at $(t_0, x_0)$ and Lipschitz continuous, we deduce that, as $h \to 0^+$,
\[
    \varphi(t_0+h,\sigma(t_0+h)) -\varphi(t_0,x_0) =
    h \partial_t \varphi(t_0, x_0) + h k(t_0 , x_0) \nabla\varphi(t_0, x_0) \cdot u + o(h).
\]
Dividing the above expression by $h$, using \eqref{admiss dynamic prog} and letting $h \to 0^+$, we get
\[
-1 \le \partial_t \varphi(t_0, x_0) + k(t_0, x_0) \nabla\varphi(t_0, x_0) \cdot u.
\]
Since this holds for every $u \in D$ and the right-hand side of the above inequality is continuous in $u$, we get that
\[
\partial_t \varphi(t_0, x_0) + k(t_0, x_0) \inf_{u \in \bar D} \nabla\varphi(t_0, x_0) \cdot u + 1 \geq 0.
\]
Since $-\frac{\nabla\varphi(t_0, x_0)}{\abs{\nabla\varphi(t_0, x_0)}} \in \bar D$, the above infimum is attained at $u = -\frac{\nabla\varphi(t_0, x_0)}{\abs{\nabla\varphi(t_0, x_0)}}$, yielding that
\begin{equation*}
-\partial_t \varphi(t_0, x_0) + \abs{\nabla\varphi(t_0, x_0)} k(t_0, x_0) - 1 \leq 0,
\end{equation*}
as required.
\end{proof}

\begin{remark}
At the light of Proposition~\ref{prop:HJ-satisfied-classical-sense}, a natural question is whether $\varphi$ is also a viscosity subsolution of \eqref{H-J equation} in $\mathbbm R_+ \times (\bar\Omega \setminus \Gamma)$, since this would imply the result of Proposition~\ref{prop:HJ-satisfied-classical-sense}. However, this may fail to be the case, as illustrated by the following example. Consider the case $d = 1$, $\Omega = (0, 1)$, $\Gamma = \{0\}$, and $k: \mathbbm R_+ \times [0, 1] \to \mathbbm R_+$ given by $k(t, x) = 1$ for every $(t, x) \in \mathbbm R_+ \times [0, 1]$. One easily computes that $\varphi(t, x) = x$ for every $(t, x) \in \mathbbm R_+ \times [0, 1]$. For every $\alpha \in (-\infty, 1]$, the function $\xi: \mathbbm R_+ \times [0, 1]$ defined for $(t, x) \in \mathbbm R_+ \times [0, 1]$ by $\xi(t, x) = 1 + \alpha(x - 1)$ satisfies $\xi(t, 1) = \varphi(t, 1) = 1$ and $\xi(t, x) \geq \varphi(t, x)$ for every $(t, x) \in \mathbbm R_+ \times [0, 1]$. However, for $\alpha < -1$ and $t \in \mathbbm R_+$, we have
\[
-\partial_t \xi(t, 1) + \abs{\partial_x \xi(t, 1)} k(t, 1) - 1 = \abs{\alpha} - 1 > 0,
\]
showing that $\varphi$ cannot be a viscosity subsolution of \eqref{H-J equation} at $(t, 1)$.
\end{remark}

\subsection{Characterization of optimal controls}
\label{sec:charact-opt-controls}

We now turn to the problem of characterizing the optimal control $u:\mathbbm R_+\to \bar{B}$ associated with an optimal trajectory $\gamma\in \Opt(k, t_0, x_0)$. By Proposition~\ref{PropOCP-2}\ref{PropOCP-OptimalControl}, the optimal control $u$ can be written as $u(t) = -\frac{\nabla\varphi(t,\gamma(t))}{\abs*{\nabla\varphi(t,\gamma(t)}}$ for every $t \in [t_0, t_0 + \varphi(t_0, x_0))$ such that $\varphi$ is differentiable at $(t, \gamma(t))$. For optimal control problems without state constraints, one can typically prove that $\varphi$ is a semiconcave function (see, e.g., \cite[Theorem~8.2.7]{CannarsaPiermarcoSinestrari}) and use this fact and properties of superdifferentials of semiconcave functions to deduce that $\varphi$ is indeed differentiable along optimal trajectories, except possibly at their starting and ending times (see, e.g., \cite[Theorem~8.4.6]{CannarsaPiermarcoSinestrari}). 

Since value functions of optimal control problems with state constraints may fail to be semiconcave, we propose in this section an alternative way to characterize the optimal control. In particular, we are also able to provide a characterization of the optimal control under weaker assumptions: semiconcavity results usually require the dynamics of the system to be $\mathbf C^{1, 1}$ with respect to the state, as in \cite[Theorem~8.2.7]{CannarsaPiermarcoSinestrari}, but the Lipschitz continuity assumption \ref{HypoOCP-k-Lip} is actually sufficient for our strategy. The approach we follow here is an adaptation to the case with state constraints of the characterization of optimal controls from \cite{SadeghiMulti}. We start by introducing the set of optimal directions at a point.

\begin{definition}
\label{def:U}
Consider the optimal control problem $\OCP(k)$ and assume that \HypoOmega, \ref{HypoOCP-k-Bound}, and \ref{HypoOCP-k-Lip} hold. Let $(t_0, x_0) \in \mathbbm R_+ \times \bar{\Omega}$. We define the set $\mathcal U(t_0, x_0)$ of \emph{optimal directions} at $(t_0, x_0)$ as the set of $u_0 \in \mathbbm S^{d-1}$ for which there exists $\gamma \in \Opt(k, t_0, x_0)$ such that the corresponding optimal control $u \in \Lip([t_0, t_0 + \varphi(t_0, x_0)], \mathbbm S^{d-1})$ associated with $\gamma$ satisfies $u(t_0) = u_0$.
\end{definition}

Notice that, by Proposition~\ref{PropOCP-1}\ref{PropOCP-ExistsOptimal} and Corollary~\ref{coro:smooth}, the set $\mathcal{U}(t_0,x_0)$ is non-empty whenever $x_0 \in \bar\Omega \setminus \Gamma$. We now show that
$\mathcal U(t_0, x_0)$ is singleton along optimal trajectories, except possibly at their initial and final points. This result was already presented in \cite[Proposition~4.11]{Mazanti2019Minimal} for optimal control problems without state constraints and with the stronger assumption that $k \in \Lip(\mathbbm R_+ \times \bar\Omega; \mathbbm R_+)$, but the proof presented in that reference also applies to our setting. We provide the proof here for completeness and also since it is an easy and interesting consequence of Corollary~\ref{coro:smooth}.

\begin{proposition}
\label{PropUSingleElement}
Consider the optimal control problem $\OCP(k)$ and its value function $\varphi$ and assume that \HypoOmega{}, \ref{HypoOCP-k-Bound}, and \ref{HypoOCP-k-Lip} hold. Let $(t_0, x_0) \in \mathbbm R_+ \times \bar{\Omega}$ and $\gamma \in \Opt(k, t_0, x_0)$. Then, for every $t \in (t_0, t_0 + \varphi(t_0, x_0))$, the set $\mathcal{U}(t, \gamma(t))$ contains exactly one element.
\end{proposition}

\begin{proof}
Let $u: \mathbbm R_+ \to \bar B$ be the optimal control associated with $\gamma$, $t_1 \in (t_0, t_0 + \varphi(t_0, x_0))$, and $x_1 = \gamma(t_1)$. Since $x_1 \in \bar\Omega \setminus \Gamma$, we have $\mathcal{U}(t_1, x_1) \neq \varnothing$. Moreover, by Proposition~\ref{PropOCP-1}\ref{PropOCP-DPP}, we have $\varphi(t_1, x_1) + t_1 - t_0 = \varphi(t_0, x_0)$ and, noticing that $\tau(t_1, \gamma) = t_0 + \varphi(t_0, x_0) - t_1 = \varphi(t_1, x_1)$, we deduce that the restriction of $\gamma$ to $[t_1, t_0 + \varphi(t_0, x_0)]$ (extended by constant values to $\mathbbm R_+$) is an optimal trajectory for $\OCP(k)$ from $(t_1, x_1)$, and thus $u(t_1) \in \mathcal U(t_1, x_1)$.

We conclude the proof by showing that, for every $u_1 \in \mathcal U(t_1, x_1)$, we have $u_1 = u(t_1)$. For that purpose, take $u_1 \in \mathcal U(t_1, x_1)$. By definition of $\mathcal U(t_1, x_1)$, there exists $\tilde\gamma \in \Opt(k, t_1, x_1)$ such that its associated control $\tilde u$ satisfies $\tilde u(t_1) = u_1$. Consider the trajectory $\hat\gamma: \mathbbm R_+ \to \bar\Omega$ defined by $\hat\gamma(t) = \gamma(t)$ for $t \leq t_1$ and $\hat\gamma(t) = \tilde\gamma(t)$ for $t \geq t_1$, and notice that its associated control $\hat u$ satisfies $\hat u(t) = u(t)$ for $t < t_1$ and $\hat u(t) = \tilde u(t)$ for $t > t_1$. Then $\hat\gamma \in \Adm(k)$ and, by construction, $\tau(t_0, \hat\gamma) = t_1 - t_0 + \tau(t_1, \tilde\gamma) = t_1 - t_0 + \varphi(t_1, x_1) = \varphi(t_0, x_0)$, which yields that $\hat\gamma \in \Opt(k, t_0, x_0)$. In particular, by Corollary~\ref{coro:smooth}, we have $\hat u \in \Lip([t_0, t_0 + \varphi(t_0, x_0)]; \mathbbm S^{d-1})$, and thus $\hat u(t_1) = \lim_{t \to t_1^{-}} \hat u(t) = u(t_1)$ and $\hat u(t_1) = \lim_{t \to t_1^{+}} \hat u(t) = \tilde u(t_1) = u_1$, yielding that $u_1 = u(t_1)$, as required.
\end{proof}

Similarly to \cite{SadeghiMulti}, the goal of this section is to characterize $\mathcal U(t_0, x_0)$ as the set of directions of maximal descent of the value function $\varphi$ of $\OCP(k)$. For that purpose, we first introduce the notion of descent rate of $\varphi$ along a given direction.

\begin{definition}
Consider the optimal control problem $\OCP(k)$ and its value function $\varphi$ and assume that \HypoOmega{}, \ref{HypoOCP-k-Bound}, and \ref{HypoOCP-k-Lip} hold. Let $(t_0, x_0) \in \mathbbm R_+ \times \bar\Omega$.
\begin{enumerate}
\item The set of \emph{inward pointing directions} at $x_0$ is the set $\In(x_0)$ of vectors $u_0 \in \mathbbm S^{d-1}$ for which there exists $h_0 > 0$ such that $x_0 + h u_0 \in \bar\Omega$ for every $h \in [0, h_0]$.
\item The \emph{descent rate} of $\varphi$ at $(t_0, x_0)$ is the function $\Delta_{(t_0, x_0)}: \In(x_0) \to \mathbbm R$ defined for $u_0 \in \In(x_0)$ by
\[
\Delta_{(t_0, x_0)}(u_0) = \limsup_{h \to 0^+} \frac{\varphi(t_0 + h, x_{0} + h k(t_0, x_{0}) u_0) - \varphi(t_0, x_{0})}{h}.
\]
\end{enumerate}
\end{definition}

Note that $\In(x_0) = \mathbbm S^{d-1}$ if $x_0 \in \Omega$, while, thanks to \ref{HypoOmegaC11}, for $x_0 \in \partial\Omega$, we have
\[
\{u_0 \in \mathbbm S^{d-1} \suchthat u_0 \cdot \mathbf n(x_0) < 0\} \subset \In(x_0) \subset \{u_0 \in \mathbbm S^{d-1} \suchthat u_0 \cdot \mathbf n(x_0) \leq 0\}.
\]
In general, both inclusions may be strict. As a consequence of those inclusions, we have
\[
\overline{\In}(x_0) = \{u_0 \in \mathbbm S^{d-1} \suchthat u_0 \cdot \mathbf n(x_0) \leq 0\}.
\]

The next proposition provides important properties of the descent rate function.

\begin{proposition}
\label{prop:Delta}
Consider the optimal control problem $\OCP(k)$ and assume that \HypoOmega{}, \ref{HypoOCP-k-Bound}, and \ref{HypoOCP-k-Lip} hold. Let $(t_0, x_0) \in \mathbbm R_+ \times \bar\Omega$.
\begin{enumerate}
\item\label{item:Delta-geq--1} For every $u_0 \in \In(x_0)$, we have $\Delta_{(t_0, x_0)}(u_0) \geq -1$.
\item\label{item:Delta-Lipschitz} The function $\Delta_{(t_0, x_0)}$ is Lipschitz continuous on $\In(x_0)$, with a Lipschitz constant depending only on $K_{\max}$ and the Lipschitz constant of $\varphi$, and thus independent of $(t_0, x_0)$.
\end{enumerate}
\end{proposition}

\begin{proof}
To show \ref{item:Delta-geq--1}, take $u_0 \in \In(x_0)$ and let $\sigma \in \mathbf C^1(\mathbbm R_+; \mathbbm R^d)$ be the unique solution of
\[
\left\{
\begin{aligned}
\dot\sigma(t) & = k(t, \sigma(t)) u_0, \\
\sigma(t_0) & = x_0.
\end{aligned}
\right.
\]
Since $u_0 \in \In(x_0)$, there exists $t_\ast > t_0$ such that $\sigma(t) \in \bar\Omega$ for every $t \in [t_0, t_\ast]$. Define $\gamma \in \Adm(k)$ by
\[
\gamma(t) = \begin{dcases*}
x_0, & if $t \in [0, t_0]$, \\
\sigma(t), & if $t \in [t_0, t_\ast]$, \\
\sigma(t_\ast), & if $t \geq t_\ast$.
\end{dcases*}
\]
Applying Proposition~\ref{PropOCP-1}\ref{PropOCP-DPP} to $\gamma$, we deduce that $\varphi(t_0 + h, \gamma(t_0 + h)) + h \geq \varphi(t_0, x_0)$ for every $h \geq 0$, and thus, for $h \in (0, t_\ast - t_0]$, we have
\[
\frac{\varphi(t_0 + h, \sigma(t_0 + h)) - \varphi(t_0, x_0)}{h} \geq -1.
\]
Since $\sigma(t_0 + h) = x_0 + h k(t_0, x_0) u_0 + o(h)$ as $h \to 0^+$ and $\varphi$ is Lipschitz continuous, we deduce that $\varphi(t_0 + h, \sigma(t_0 + h)) = \varphi(t_0 + h, x_0 + h k(t_0, x_0) u_0) + o(h)$, yielding that, as $h \to 0^+$,
\begin{equation}
\label{eq:decay-varphi-geq--1}
\frac{\varphi(t_0 + h, x_0 + h k(t_0, x_0) u_0) - \varphi(t_0, x_0)}{h} \geq -1 + o(1),
\end{equation}
and the conclusion of \ref{item:Delta-geq--1} follows from the definition of $\Delta_{(t_0, x_0)}$.

In order to prove \ref{item:Delta-Lipschitz}, take $u_1, u_2 \in \In(x_0)$ and denote by $C > 0$ the Lipschitz constant of $\varphi$. For $h > 0$ small enough, we have
\begin{multline*}
\frac{\varphi(t_0 + h, x_{0} + h k(t_0, x_{0}) u_1) - \varphi(t_0, x_{0})}{h} = \frac{\varphi(t_0 + h, x_{0} + h k(t_0, x_{0}) u_2) - \varphi(t_0, x_{0})}{h} \\
 + \frac{\varphi(t_0 + h, x_{0} + h k(t_0, x_{0}) u_1) - \varphi(t_0 + h, x_{0} + h k(t_0, x_{0}) u_2)}{h} \\
\leq \frac{\varphi(t_0 + h, x_{0} + h k(t_0, x_{0}) u_2) - \varphi(t_0, x_{0})}{h} + C K_{\max} \abs{u_1 - u_2}.
\end{multline*}
Taking the $\limsup$ as $h \to 0^+$, we deduce that $\Delta_{(t_0, x_0)}(u_1) \leq \Delta_{(t_0, x_0)}(u_2) + C K_{\max} \abs{u_1 - u_2}$. Since this holds for every $u_1, u_2 \in \In(x_0)$, we obtain that
\[
\abs{\Delta_{(t_0, x_0)}(u_1) - \Delta_{(t_0, x_0)}(u_2)} \leq C K_{\max} \abs{u_1 - u_2} \qquad \text{ for every } u_1, u_2 \in \In(x_0),
\]
as required.
\end{proof}

As a consequence of Proposition~\ref{prop:Delta}\ref{item:Delta-Lipschitz}, $\Delta_{(t_0, x_0)}: \In(x_0) \to \mathbbm R$ can be extended in a unique way to a Lipschitz continuous function defined on $\overline{\In}(x_0)$. In the sequel, by a slight abuse of notation, we use $\Delta_{(t_0, x_0)}$ to denote this Lipschitz continuous extension. Note that this extension still satisfies the assertions of Proposition~\ref{prop:Delta}.

We are now in position to provide the definition of the set of directions of maximal descent of the value function.

\begin{definition}
\label{DefW}
Consider the optimal control problem $\OCP(k)$ and its value function $\varphi$ and assume that \HypoOmega, \ref{HypoOCP-k-Bound}, and \ref{HypoOCP-k-Lip} hold. Let $(t_0, x_0) \in \mathbbm R_+ \times \bar{\Omega}$. We define the set $\mathcal W(t_0, x_0)$ of \emph{directions of maximal descent of $\varphi$} at $(t_0, x_0)$ by
\[
\mathcal W(t_0, x_0) = \Delta^{-1}_{(t_0, x_0)}(\{-1\}) = \{u_0 \in \overline\In(x_0) \suchthat \Delta_{(t_0, x_0)}(u_0) = -1\}.
\]
\end{definition}

Notice that the term \emph{maximal descent} is motivated by Proposition~\ref{prop:Delta}\ref{item:Delta-geq--1}, since we are considering the elements $u_0 \in \overline\In(x_0)$ reaching the lower bound $-1$ on $\Delta_{(t_0, x_0)}$.

\begin{remark}
\label{remk:W-interior}
If $x_0 \in \Omega$ (or, more generally, if the set $\In(x_0)$ is closed), then $\mathcal W(t_0, x_0)$ is simply the set of $u_0 \in \mathbbm S^{d-1}$ such that
\begin{equation}
\label{eq:decay-lim}
\lim_{h \to 0^+} \frac{\varphi(t_0 + h, x_{0} + h k(t_0, x_{0}) u_0) - \varphi(t_0, x_{0})}{h} = -1,
\end{equation}
which is the definition of $\mathcal W(t_0, x_0)$ provided previously in \cite[Definition~4.11]{SadeghiMulti} for optimal control problems without state constraints. Indeed, notice that, if $\In(x_0)$ is closed, then $u_0 \in \mathcal W(t_0, x_0)$ if and only if 
\begin{equation}
\label{eq:decay-limsup}
\limsup_{h \to 0^+} \frac{\varphi(t_0 + h, x_{0} + h k(t_0, x_{0}) u_0) - \varphi(t_0, x_{0})}{h} = -1.
\end{equation}
On the other hand, \eqref{eq:decay-varphi-geq--1} also yields that
\[
\liminf_{h \to 0^+} \frac{\varphi(t_0 + h, x_{0} + h k(t_0, x_{0}) u) - \varphi(t_0, x_{0})}{h} \geq -1
\]
for every $u \in \In(x_0)$, and hence \eqref{eq:decay-lim} is equivalent to \eqref{eq:decay-limsup}.
\end{remark}

Our next result shows that, at the points $(t_0, x_0)$ where $\varphi$ is differentiable, $\mathcal W(t_0, x_0)$ contains a unique direction of maximal descent which, as one might expect, is equal to $-\frac{\nabla \varphi(t_0, x_0)}{\abs*{\nabla \varphi(t_0, x_0)}}$. This was already shown in \cite[Proposition~4.13]{SadeghiMulti} for optimal control problems without state constraints, and that proof also carries over to the present case thanks to Proposition~\ref{prop:HJ-satisfied-classical-sense}. For sake of completeness, and since this proof is quite elementary, we provide it here.

\begin{proposition}
\label{PropWNormalizedGrad}
Consider the optimal control problem $\OCP(k)$ and its value function $\varphi$ and assume that \HypoOmega, \ref{HypoOCP-k-Bound}, and \ref{HypoOCP-k-Lip} hold. Let $(t_0, x_0) \in \mathbbm R_+ \times (\bar{\Omega} \setminus \Gamma)$ be such that $\varphi$ is differentiable at $(t_0, x_0)$. Then
\begin{equation}
\label{eq:W-case-differentiable}
\mathcal{W}(t_0, x_0) = \left\{-\frac{\nabla \varphi(t_0, x_0)}{\abs*{\nabla \varphi(t_0, x_0)}} \right\}.
\end{equation}
\end{proposition}

\begin{proof}
Since $\varphi$ is differentiable at $(t_0, x_0)$, we have, for every $u \in \In(x_0)$,
\begin{equation}
\label{eq:Delta-case-differentiable}
\Delta_{(t_0, x_0)}(u) = \partial_t \varphi(t_0, x_0) + k(t_0, x_0) \nabla\varphi(t_0, x_0) \cdot u.
\end{equation}
By continuity, the above equality also holds for every $u \in \overline\In(x_0)$. We also have, by Proposition~\ref{prop:HJ-satisfied-classical-sense}, that
\begin{equation}
\label{eq:HJ-case-differentiable}
-\partial_t \varphi(t_0, x_0) + k(t_0, x_0) \abs{\nabla\varphi(t_0, x_0)} - 1 = 0.
\end{equation}
Moreover, recall that $k(t_0, x_0) > 0$ and,  by Proposition~\ref{PropOCP-2}\ref{PropOCP-LowerBound}, $\nabla\varphi(t_0, x_0) \neq 0$. If $u_0 \in \mathcal W(t_0, x_0)$, then $\Delta_{(t_0, x_0)}(u_0) = -1$ and, combining with \eqref{eq:Delta-case-differentiable} and \eqref{eq:HJ-case-differentiable}, we deduce that
\[
k(t_0, x_0) \left[\abs{\nabla\varphi(t_0, x_0)} + \nabla\varphi(t_0, x_0) \cdot u_0\right] = 0,
\]
which yields $u_0 = -\frac{\nabla\varphi(t_0, x_0)}{\abs{\nabla\varphi(t_0, x_0)}}$ since $u_0 \in \mathbbm S^{d-1}$. Conversely, defining $u_0 = -\frac{\nabla\varphi(t_0, x_0)}{\abs{\nabla\varphi(t_0, x_0)}}$, Remark~\ref{remk:boundary-differentiable} ensures that $u_0 \in \overline{\In}(x_0)$ and it is immediate to compute, using \eqref{eq:Delta-case-differentiable} and \eqref{eq:HJ-case-differentiable}, that $\Delta_{(t_0, x_0)}(u_0) = -1$, showing that $u_0 \in \mathcal W(t_0, x_0)$.
\end{proof}

The reason why we go through the definition of descent rate $\Delta_{(t_0, x_0)}$ in order to define $\mathcal W(t_0, x_0)$ instead of the more direct definition provided in \cite[Definition~4.11]{SadeghiMulti} is that, if $x_0 \in \partial\Omega$ and $u_0 \in \mathbbm S^{d-1}$, one might have $x_0 + h k(t_0, x_0) u_0 \notin \bar\Omega$ for every $h > 0$ small enough. Since $\varphi$ is defined only in the set $\mathbbm R_+ \times \bar\Omega$, this means that the term $\varphi(t_0 + h, x_{0} + h k(t_0, x_{0}) u_0)$ is not well-defined for any $h > 0$ small enough, and thus the limit in the left-hand side of \eqref{eq:decay-lim} does not make sense. An alternative approach, however, is to replace $\varphi$ in \eqref{eq:decay-lim} by the value function $\varphi_\epsilon$ of the penalized optimal control problem $\OCP_\epsilon(k_\epsilon)$ defined in Section~\ref{sec:penalized}. This is the subject of our next definition, which coincides with \cite[Definition~4.11]{SadeghiMulti} but with $\varphi$ is replaced by $\varphi_\epsilon$ whenever its argument might fail to belong to its domain $\mathbbm R_+ \times \bar\Omega$.

\begin{definition}
Consider the optimal control problems $\OCP(k)$ and $\OCP_\epsilon(k_\epsilon)$ and their respective value functions $\varphi$ and $\varphi_\epsilon$ and assume that \HypoOmega, \ref{HypoOCP-k-Bound}, and \ref{HypoOCP-k-Lip} hold. Let $(t_0, x_0) \in \mathbbm R_+ \times \bar\Omega$. For $\epsilon > 0$, we define the set $\mathcal W_\epsilon(t_0, x_0)$ of \emph{directions of maximal descent of $\varphi_\epsilon$} at $(t_0, x_0)$ by
\[
\mathcal W_\epsilon(t_0, x_0) = \left\{u_0 \in \mathbbm S^{d-1} \suchthat \lim_{h \to 0^+} \frac{\varphi_\epsilon(t_0 + h, x_{0} + h k(t_0, x_{0}) u_0) - \varphi(t_0, x_{0})}{h} = -1\right\}.
\]
\end{definition}

Our next result concerns the relation between optimal directions (i.e., elements of $\mathcal U(t_0, x_0)$) and directions of maximal descent of $\varphi$ (i.e., elements of $\mathcal W(t_0, x_0)$), and asserts that both notions actually coincide, and that they also coincide with $\mathcal W_\epsilon(t_0, x_0)$ for $\epsilon > 0$ small enough. The fact that $\mathcal U(t_0, x_0) = \mathcal W(t_0, x_0)$ was already established in \cite[Theorem~4.14]{SadeghiMulti} for minimal-time optimal control problems without state constraints, and the proof we provide here follows the same lines, but passes through the optimal control problem $\OCP_\epsilon(k_\epsilon)$ in order to handle the state constraints appropriately.

\begin{theorem}\label{thm Ut_0,x_0}
Consider the optimal control problems $\OCP(k)$ and $\OCP_\epsilon(k_\epsilon)$ and assume that \HypoOmega, \ref{HypoOCP-k-Bound}, and \ref{HypoOCP-k-Lip} hold. There exists $\epsilon_0 > 0$ such that, for every $\epsilon \in (0, \epsilon_0)$ and $(t_0, x_0) \in \mathbbm R_+ \times \bar{\Omega}$, we have $\mathcal{U}(t_0, x_0) = \mathcal{W}(t_0, x_{0}) = \mathcal W_\epsilon(t_0, x_0)$.
\end{theorem}

\begin{proof}
We first remark that, if $x_0 \in \Gamma$, then $\mathcal{U}(t_0, x_0) = \mathcal{W}(t_0, x_{0}) = \mathcal W_\epsilon(t_0, x_0) = \varnothing$, and so we are only left to consider the case $x_0 \in \bar{\Omega} \setminus \Gamma$. We let $\epsilon_0 > 0$ be as in the statement of Propositions~\ref{gamma-eps in Omeg} and \ref{prop:varphi-eps-loc-Lip} and Theorem~\ref{thm:Opt-eps-equals-Opt}, and we fix $\epsilon \in (0, \epsilon_0)$. We proceed by proving that $\mathcal{U}(t_0, x_0) \subset \mathcal{W}(t_0, x_{0}) \subset \mathcal W_\epsilon(t_0, x_0) \subset \mathcal{U}(t_0, x_0)$.

\medskip

\noindent\emph{Part~I: Proof of the inclusion $\mathcal{U}(t_0, x_{0}) \subset \mathcal{W}(t_0, x_{0})$.} Let $u_0 \in \mathcal U(t_0, x_0)$ and take $\gamma \in \Opt(k, t_0, x_0)$ and $u: \mathbbm R_+ \to \bar B$ an optimal control associated with $\gamma$ such that $u \in \Lip([t_0, t_0 + \varphi(t_0, x_0)]; \mathbbm S^{d-1})$ and $u(t_0) = u_0$. By Proposition~\ref{PropOCP-1}\ref{PropOCP-DPP}, we have, for every $h \in (0, \varphi(t_0, x_0)]$, that
\begin{equation}
\label{eq:frac-DPP}
\frac{\varphi(t_0 + h, \gamma(t_0 + h)) - \varphi(t_0, x_0)}{h} = -1.
\end{equation}

We claim that $u_0 \in \overline\In(x_0)$. Indeed, this is trivial if $x_0 \in \Omega$, and, if $x_0 \in \partial\Omega$, since $\gamma$ takes values in $\bar\Omega$ and $\dot\gamma(t_0) = k(t_0, x_0) u_0$, one can easily check that $u_0 \cdot \mathbf n(x_0) \leq 0$, implying that $u_0 \in \overline\In(x_0)$.

Let $(u_n)_{n \in \mathbbm N}$ be a sequence in $\In(x_0)$ such that $u_n \to u_0$ as $n \to +\infty$. Fix $n \in \mathbbm N$ and notice that, as $h \to 0^+$, we have
\begin{align*}
\gamma(t_0 + h) & = x_0 + h k(t_0, x_0) u_0 + o(h) \\
& = x_0 + h k(t_0, x_0) u_n + h k(t_0, x_0) (u_0 - u_n) + o(h).
\end{align*}
Recall that $\varphi$ is Lipschitz continuous (Proposition~\ref{varphi is Lipschitz}) and denote by $M > 0$ its Lipschitz constant. Using the fact that $x_0 + h k(t_0, x_0) u_n \in \bar\Omega$ for $h > 0$ small enough, we deduce that, as $h \to 0^+$,
\[
\abs{\varphi(t_0 + h, \gamma(t_0 + h)) - \varphi(t_0 + h, x_0 + h k(t_0, x_0) u_n)} \leq h M K_{\max} \abs{u_n - u_0} + o(h)
\]
Combining with \eqref{eq:frac-DPP} and letting $h \to 0$, we deduce that, for every $n \in \mathbbm N$,
\[
\Delta_{(t_0, x_0)}(u_n) = \limsup_{h \to 0^+} \frac{\varphi(t_0 + h, x_0 + h k(t_0, x_0) u_n) - \varphi(t_0, x_0)}{h} \leq -1 + M K_{\max} \abs{u_n - u_0}.
\]
Hence, letting $n \to +\infty$ and using Proposition~\ref{prop:Delta}, we deduce that $\Delta_{(t_0, x_0)}(u_0) = -1$, showing that $u_0 \in \mathcal W(t_0, x_0)$, as required.

\medskip

\noindent\emph{Part~II: Proof of the inclusion $\mathcal W(t_0, x_0) \subset \mathcal W_\epsilon(t_0, x_0)$.} Let $u_0 \in \mathcal W(t_0, x_0)$ and consider a sequence $(u_n)_{n \in \mathbbm N}$ in $\In(x_0)$ such that $u_n \to u_0$ as $n \to +\infty$. By Proposition~\ref{prop:varphi-eps-loc-Lip}, $\varphi_\epsilon$ is Lipschitz continuous on $\mathbbm R_+ \times \bar\Omega_{\epsilon/2}$, and we denote by $C_\epsilon > 0$ a Lipschitz constant of this map. Notice that, for $h \in \left[0, \frac{\epsilon}{2 K_{\max}}\right]$, we have $d_\Omega(x_0 + h k(t_0, x_0) u_n) \leq \frac{\epsilon}{2}$ for every $n \in \mathbbm N \cup \{0\}$. Hence, for $h \in \left(0, \frac{\epsilon}{2 K_{\max}}\right]$ and $n \in \mathbbm N$, we have
\begin{multline*}
\frac{\varphi_\epsilon(t_0 + h, x_0 + h k(t_0, x_0) u_0) - \varphi(t_0, x_0)}{h} \\ \leq \frac{\varphi_\epsilon(t_0 + h, x_0 + h k(t_0, x_0) u_n) - \varphi(t_0, x_0)}{h} + C_\epsilon K_{\max} \abs{u_n - u_0}.
\end{multline*}
For $n \in \mathbbm N$, since $u_n \in \In(x_0)$, we have $x_0 + h k(t_0, x_0) u_n \in \bar\Omega$ for $h > 0$ small enough, and thus, by Theorem~\ref{thm:Opt-eps-equals-Opt}, we have $\varphi_\epsilon(t_0 + h, x_0 + h k(t_0, x_0) u_n) = \varphi(t_0 + h, x_0 + h k(t_0, x_0) u_n)$ for $h$ small enough. Letting $h \to 0^+$, we thus deduce that
\[
\limsup_{h \to 0^+} \frac{\varphi_\epsilon(t_0 + h, x_0 + h k(t_0, x_0) u_0) - \varphi(t_0, x_0)}{h} \leq \Delta_{(t_0, x_0)}(u_n) + C_\epsilon K_{\max} \abs{u_n - u_0}.
\]
Taking now the limit as $n \to +\infty$, we obtain that
\begin{equation}
\label{eq:W-subset-W-eps-limsup}
\limsup_{h \to 0^+} \frac{\varphi_\epsilon(t_0 + h, x_0 + h k(t_0, x_0) u_0) - \varphi(t_0, x_0)}{h} \leq -1.
\end{equation}
Let $\gamma: \mathbbm R_+ \to \mathbbm R^d$ be defined by $\gamma(t) = x_0$ for $t \in [0, t_0]$ and as the solution of the differential equation $\dot\gamma(t) = k_\epsilon(t, \gamma(t)) u_0$ with initial condition $\gamma(t_0) = x_0$ for $t \geq t_0$. Then $\gamma \in \Adm_\epsilon(k_\epsilon)$ and, by the dynamic programming principle \eqref{eq:DPP-epsilon} for $\OCP_\epsilon(k_\epsilon)$, and using also Theorem~\ref{thm:Opt-eps-equals-Opt}, we have, for $h \geq 0$,
\[
\frac{\varphi_\epsilon(t_0 + h, \gamma(t_0 + h)) - \varphi(t_0, x_0)}{h} \geq -1.
\]
Using that $\gamma(t_0 + h) = x_0 + h k(t_0, x_0) u_0 + o(h)$ as $h \to 0^+$, $\gamma(t_0 + h) \in \bar\Omega_{\epsilon/2}$ for $h \geq 0$ small enough, and that $\varphi_\epsilon$ is Lipschitz continuous in $\mathbbm R_+ \times \bar\Omega_{\epsilon/2}$, we deduce that, as $h \to 0^+$,
\[
\frac{\varphi_\epsilon(t_0 + h, x_0 + h k(t_0, x_0) u_0) - \varphi(t_0, x_0)}{h} \geq -1 + o(1).
\]
Together with \eqref{eq:W-subset-W-eps-limsup}, this shows that
\[
\lim_{h \to 0^+} \frac{\varphi_\epsilon(t_0 + h, x_0 + h k(t_0, x_0) u_0) - \varphi(t_0, x_0)}{h} = -1,
\]
and thus $u_0 \in \mathcal W_\epsilon(t_0, x_0)$, as required.

\medskip

\noindent\emph{Part~III: Proof of the inclusion $\mathcal W_\epsilon(t_0, x_0) \subset \mathcal U(t_0, x_0)$.} Let $u_0 \in \mathcal W_\epsilon(t_0, x_0)$ and $h > 0$, which is implicitly always assumed to be small enough. Then, by definition of $\mathcal W_\epsilon(t_0, x_0)$, we have, as $h \to 0^+$,
\begin{equation}
\label{eq:u0-in-W}
\varphi_\epsilon(t_0 + h, x_{0} + h k(t_0, x_{0}) u_{0}) = \varphi(t_0, x_{0}) - h + o(h).
\end{equation}
Define $\gamma_{0}: [t_0, t_0 + h] \to \mathbbm R^d$ by 
\begin{equation}\label{ODE cont1}
    \left \{
    \begin{aligned}
    \Dot{\gamma}_{0}(t) & = k_\epsilon(t,\gamma_{0}(t)) u_{0},
    \\  
    \gamma_{0}(t_0) & = x_{0}.       
    \end{aligned} \right.
\end{equation}
Let $x_{1}^h = \gamma_{0}(t_0 + h)$ and $t_1^h = t_0 + h$. Since $\Gamma$ is closed and $x_0 \notin \Gamma$, we have $x_1^h \notin \Gamma$ for $h > 0$ small enough. Let $\gamma_{1}^h \in \Opt_\epsilon(k_\epsilon, t_1^h, x_{1}^h)$ and $u_{1}^h$ be the optimal control associated with $\gamma_{1}^h$, which satisfies $u_1^h \in \Lip([t_1^h, t_1^h + \varphi_\epsilon(t_1^h, x_1^h)]; \mathbbm S^{d-1})$ by Proposition~\ref{prop:consequences-of-Pontryagin-for-optimal-trajectories-of-ocp-epsilon}. Set $\Bar{u}_{1}^h = u_{1}^h(t_1^h)\in \mathbbm S^{d-1}$ and define $\Bar\gamma_1^h: [t_1^h, t_1^h + h] \to \mathbbm R^d$ by
\begin{equation}\label{ODE cont2}
    \left \{
    \begin{aligned}
    \Dot{\Bar{\gamma}}_{1}^h(t) & = k_\epsilon(t, \Bar{\gamma}_{1}^h(t)) \Bar{u}_{1}^h \\  
    \Bar{\gamma}_{1}^h(t_1^h) & = x_{1}^h.
    \end{aligned} \right.
\end{equation}
Let us also set $t_2^h=t_1^h+h$, $x_{2}^h = \gamma_{1}^h(t_2^h)$ and $\Bar{x}_{2}^h = \Bar{\gamma}_{1}^h(t_2^h)$. We split the sequel of the proof in two cases.

\medskip

\noindent\emph{Case 1.} We assume in this case that $\lim_{h \to 0^+} \Bar{u}_{1}^h = u_{0}$. Let $\hat u_1^h \in \Lip(\mathbbm S^{d-1})$ be defined by $\hat u_1^h(t) = \bar u_1^h$ for $t \in [0, t_1^h]$, $\hat u_1^h(t) = u_1^h(t)$ for $t \in [t_1^h, t_1^h + \varphi_\epsilon(t_1^h, x_1^h)]$, and $\hat u_1^h(t) = u_1^h(t_1^h + \varphi(t_1^h, x_1^h))$ for $t \geq t_1^h + \varphi(t_1^h, x_1^h)$. Since $\gamma_{1}^h$ and $\hat u_{1}^h$ are Lipschitz continuous and their Lipschitz constants do not depend on $h$ (see Proposition~\ref{prop:consequences-of-Pontryagin-for-optimal-trajectories-of-ocp-epsilon}), one deduces from Arzelà--Ascoli Theorem that there exist a positive sequence $(h_n)_{n \in \mathbbm N}$ converging to $0$ as $n \to +\infty$ and elements $\gamma^* \in \Lip_{K_{\max}}(\mathbbm R^d)$ and $u^* \in \Lip(\mathbbm S^{d-1})$ such that $\gamma_{1}^{h_n} \to \gamma^*$ and $\hat u_{1}^{h_n} \to u^*$ as $n \to +\infty$, uniformly on compact time intervals. Since $\gamma_1^h \in \Opt_\epsilon(k_\epsilon, t_1^h, x_1^h)$ for $h > 0$ and $t_1^h \to t_0$ and $x_1^h \to x_0$ as $h \to 0^+$, one can easily show, using the continuity of $\varphi_\epsilon$, that $\gamma^* \in \Opt_\epsilon(k_\epsilon, t_0, x_0)$ and its corresponding optimal control coincides with $u^*$ on $[t_0, t_0 + \varphi(t_0, x_0)]$. Moreover, since $x_0 \in \bar\Omega$, we have from Theorem~\ref{thm:Opt-eps-equals-Opt} that $\Opt_\epsilon(k_\epsilon, t_0, x_0) = \Opt(k, t_0, x_0)$. In addition,
\[
u^*(t_0) = \lim_{n \to +\infty} \hat u_{1}^{h_n} (t_1^{h_n}) = \lim_{n \to +\infty} \Bar{u}_{1}^{h_n} = u_{0},
\]
which implies that $u_0 \in \mathcal U(t_0, x_0)$, as required.

\medskip

\noindent\emph{Case 2.} We now consider the case where $(\bar u_1^h)_{h > 0}$ does not converge to $u_0$ as $h \to 0^+$, and we prove that this case is not possible. Let $\epsilon > 0$ and $(h_n)_{n\in \mathbbm N}$ be a positive sequence such that $h_n \to 0$ as $n\to +\infty$ and $\abs{\Bar{u}_{1}^{h_n}-u_{0}} \ge \epsilon$ for every $n\in \mathbbm N$. For simplicity, we set $t_1^{h_n}=t_1^n$, $x_{1}^{h_n}=x_{1}^n$, and similarly for all other variables whose upper index is $h_n$. In order to clarify the constructions used in this case, we illustrate them in Figure~\ref{FigCase2}.

\tikzset{
 mid arrow/.style={postaction={decorate,decoration={
        markings,
        mark=at position .5 with {\arrow{Stealth}}
      }}}
}

\begin{figure}[ht]
\centering
\begin{tikzpicture}
\node (x0) at (0, 0) {};
\node (x1) at (0, -2) {};
\node (x2bar) at (3, -2) {};
\node (x2) at (3.25, -3.25) {};

\draw[red, mid arrow] (x0.center) -- node[midway, left] {$\gamma_0$} (x1.center);
\draw[blue, mid arrow] (x0.center) -- node[midway, above right] {$\gamma_2^n$} (x2bar.center);
\draw[violet, mid arrow] (x1.center) to[out = 0, in = 180] node[midway, below left] {$\gamma_1^n$} (x2.center);
\draw[violet, mid arrow] (x2.center) -- ++(1, 0);
\draw[green!50!black, mid arrow] (x1.center) -- node[midway, above] {$\bar \gamma_1^n$} (x2bar.center);
\draw[orange!75!black, mid arrow] (x2bar.center) -- node[midway, right] {$\gamma_3^n$} (x2.center);

\fill (x0) circle[radius=0.05] node[left] {$x_0$};
\fill (x1) circle[radius=0.05] node[left] {$x_1^n$};
\fill (x2bar) circle[radius=0.05] node[right] {$\bar x_2^n$};
\fill (x2) circle[radius=0.05] node[below] {$x_2^n$};
\fill (2, -4/3) circle[radius=0.05] node[above right] {$x_3^n$};
\end{tikzpicture}

\caption{Illustration of the constructions used in the proof of Theorem~\ref{thm Ut_0,x_0}.}
\label{FigCase2}
\end{figure}
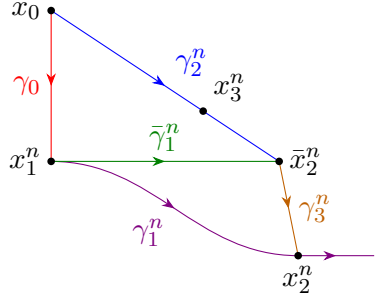

Notice that $\abs{x_1^n - x_0} \leq K_{\max} h_n$, $\abs{\gamma_1^n(t) - x_1^n} \leq K_{\max} h_n$ for every $t \in [t_1^n, t_1^n + h_n]$, and $\abs{\bar x_2^n - x_1^n} \leq K_{\max} h_n$. Hence, for all $t \in [t_1^n, t_1^n + h_n]$, we have $x_1^n, \gamma_1^n(t), x_2^n, \bar x_2^n \in \bar B(x_0, 2 K_{\max} h_n)$. Since $x_0 \in \bar\Omega$, for $n$ large enough, this ball is included in $\bar\Omega_{\epsilon/2} = \{x \in \mathbbm R^d \suchthat d_\Omega(x) \leq \frac{\epsilon}{2}\}$, and in particular $k_\epsilon(t, x) \geq \frac{1}{2} K_{\min}$ for every $(t, x) \in \mathbbm R_+ \times \bar B(x_0, 2 K_{\max} h_n)$.

Integrating \eqref{ODE cont1} on $[t_0, t_1^n]$, we get
\[
x_{1}^n-x_{0} = \int_{t_0}^{t_1^n} k_\epsilon(s, \gamma_{0}(s)) \diff s\, u_{0},
\]
and, proceeding similarly for \eqref{ODE cont2}, we get
\[
\Bar{x}_{2}^n- x_{1}^n = \int_{t_1^n}^{t_2^n} k_\epsilon(s, \Bar{\gamma}_{1}^n(s)) \diff s\, \Bar{u}_{1}^n.
\]
Denote the integrals in the right-hand side of the above equalities by $I_0^n$ and $I_1^n$, respectively. We have
\begin{equation*}
    \begin{aligned}
    \abs{\Bar{x}_{2}^n-x_{0}}^2 & = (I_0^n u_{0} + I_1^n \Bar{u}_{1}^n)\cdot (I_0^n u_{0} + I_1^n \Bar{u}_{1}^n)\\
    & = (I_0^n)^2 + (I_1^n)^2 + 2I_0^n I_1^n u_{0} \cdot \Bar{u}_{1}^n \\
    &=\abs{x_{1}^n-x_{0}}^2 + \abs*{\Bar{x}_{2}^n - x_{1}^n}^2 + 2 I_0^n I_1^n u_{0} \cdot \Bar{u}_{1}^n.
    \end{aligned}
\end{equation*}
We know that $\abs{\Bar{u}_{1}^{n}-u_{0}} \ge \epsilon$, which leads us to observe that there exists $\alpha \in (0, 1)$ such that $u_{0} \cdot \Bar{u}_{1}^n < \alpha$ for every $n\in \mathbbm N$. Thus
\[
\abs{\Bar{x}_{2}^n - x_{0}}^2 < \abs{x_{1}^n-x_{0}}^2 + \abs*{\Bar{x}_{2}^n - x_{1}^n}^2 + 2 \alpha I_0^n I_1^n.
\]
Define
\[
\rho := \sqrt{1 - (1-\alpha) \frac{K_{\min}^2}{8 K_{\max}^2}},
\]
then obviously $\rho<1$ and
\begin{align}
\abs{\Bar{x}_{2}^n - x_{0}}^2 & < \left(\abs{x_{1}^n-x_{0}} + \abs*{\Bar{x}_{2}^n - x_{1}^n}\right)^2 - 2 (1 - \alpha) I_0^n I_1^n \notag \\
 & = \left(1 - (1 - \alpha) \frac{2 I_0^n I_1^n}{\left(I_0^n + I_1^n\right)^2}\right)\left(\abs{x_{1}^n-x_{0}} + \abs*{\Bar{x}_{2}^n - x_{1}^n}\right)^2 \notag \\
 & \leq \rho^2 \left(\abs*{ x_{1}^n - x_{0}} + \abs*{\Bar{x}_{2}^n - x_{1}^n}\right)^2, \label{eq:dist-x2-x0}
\end{align}
where we use that $I_i^n \in \left[\frac{1}{2} h K_{\min}, h K_{\max}\right]$ for $i \in \{1, 2\}$. Let $u_{2}^n = \frac{\Bar{x}_{2}^n - x_{0}}{\abs{\Bar{x}_{2}^n - x_{0}}}$ (with the convention $u_2^n = 0$ if $\bar x_2^n = x_0$) and define $\gamma_{2}^n: [t_0, t_0 + \tau^n] \to \mathbbm R^d$ by
\begin{equation}\label{ODE cont3}
    \left \{
    \begin{aligned}
    \Dot{\gamma}_{2}^n(t) & = k_\epsilon(t, \gamma_{2}^n(t)) u_{2}^n \\  
    \gamma_{2}^n(t_0) & = x_{0},       
    \end{aligned}
		\right.
\end{equation}
where $\tau^n \geq 0$ is chosen so that\footnote{This is possible since, for $n$ large enough, all points in the segment from $x_0$ to $\bar x_2^n$ lie within $\Omega_{\epsilon/2}$, on which $k_\epsilon$ is lower bounded by $\frac{1}{2}K_{\min}$.} $\gamma_2^n(t_0 + \tau^n) = \bar x_2^n$.

\begin{claim}
As $n \to +\infty$, we have $\tau^n \leq 2 \rho h_n + o(h_n)$.
\end{claim}

\begin{proof}
Note that we have nothing to prove in the case $\bar x_2^n = x_0$, and hence we assume $\bar x_2^n \neq x_0$ in the sequel. If $\abs{\bar x_2^n - x_0} \leq \rho \abs{x_1^n - x_0}$, we let $x_3^n = \bar x_2^n$, otherwise we choose $x_{3}^n$ as the unique point in the segment $(x_{0}, \bar x_{2}^n)$ such that $\abs{x_{3}^n - x_{0}} = \rho \abs{x_{1}^n - x_{0}}$. In both cases, we have $\abs{x_{3}^n - x_{0}} = \bar\rho \abs{x_{1}^n - x_{0}}$ for some $\bar \rho \leq \rho$. Let $\tau_1^n$ be the time that $\gamma_{2}^n$ takes to reach the point $x_{3}^n$, i.e., $\gamma_{2}^n (t_0 + \tau_1^n) = x_{3}^n$. (Note that $\tau_1^n = \tau^n$ in the case $\abs{\bar x_2^n - x_0} \leq \rho \abs{x_1^n - x_0}$.)

We first show that $\tau_1^n \le \rho h_n + o(h_n)$. To obtain that, we observe, by integrating \eqref{ODE cont1} and \eqref{ODE cont3} and doing a change of variables, that
\begin{equation}
\label{eq:estimate-tau-1-n}
    \begin{aligned}
    \int_{t_0}^{t_0+\tau_1^n}k_\epsilon(s, \gamma_{2}^n(s)) \diff s & = \abs{x_3^n - x_0} = \bar\rho \abs{x_1^n - x_0} = \bar\rho \int_{t_0}^{t_0 + h_n} k_\epsilon(s, \gamma_0(s)) \diff s \\
		& = \int_{t_0}^{t_0 + \bar\rho h_n} k_\epsilon\left(t_0 + \frac{s - t_0}{\bar\rho}, \gamma_0\left(t_0 + \frac{s - t_0}{\bar\rho}\right)\right) \diff s \\
		& = \int_{t_0}^{t_0+\bar\rho h_n} k_\epsilon(s, \gamma_{2}^n (s))\diff s\\
    & \hphantom{{} = {} } + \int_{t_0}^{t_0+\bar\rho h_n} \left[k_\epsilon\left(t_0 + \frac{s - t_0}{\bar\rho}, \gamma_{0}\left(t_0+\frac{s-t_0}{\bar\rho}\right)\right) - k_\epsilon(s,\gamma_{2}^n(s))\right] \diff s.
    \end{aligned}
\end{equation}
Let us show that
\begin{equation}
\label{eq:limit-diff-k-eps}
\lim_{n \to +\infty} \frac{1}{h_n} \int_{t_0}^{t_0+\bar\rho h_n} \left[k_\epsilon\left(t_0 + \frac{s - t_0}{\bar\rho}, \gamma_{0}\left(t_0+\frac{s-t_0}{\bar\rho}\right)\right) - k_\epsilon(s,\gamma_{2}^n(s))\right] \diff s = 0
\end{equation}
Let $\delta > 0$. Since $k_\epsilon$ is continuous in $(t_0, x_0)$, there exists $\eta > 0$ such that $\abs{k_\epsilon(t, x) - k_\epsilon(t_0, x_0)} < \delta$ for every $(t, x) \in \mathbbm R_+ \times \mathbbm R^d$ satisfying $\abs{t - t_0} < \eta$ and $\abs{x - x_0} < \eta$. Since $h_n \to 0$ as $n \to +\infty$, there exists $N \in \mathbbm N$ such that, for every $n \geq N$, we have $h_n < \eta$ and $K_{\max} h_n < \eta$. Noticing that, for every $s \in [t_0, t_0 + \bar\rho h_n]$, we have $\abs*{\frac{s - t_0}{\bar\rho}} \leq h_n < \eta$, $\abs*{\gamma_{0}\left(t_0+\frac{s-t_0}{\bar\rho}\right) - x_0} \leq K_{\max} h_n < \eta$, $\abs{s - t_0} \leq \bar\rho h_n < \eta$, and $\abs{\gamma_{2}^n(s) - x_0} \leq \bar\rho K_{\max} h_n < \eta$, we deduce that, for $n \geq N$, we have, for every $s \in [t_0, t_0 + \bar\rho h_n]$,
\begin{gather*}
k_\epsilon(t_0, x_0) - \delta < k_\epsilon\left(t_0 + \frac{s - t_0}{\bar\rho}, \gamma_{0}\left(t_0+\frac{s-t_0}{\bar\rho}\right)\right) < k_\epsilon(t_0, x_0) + \delta, \\
k_\epsilon(t_0, x_0) - \delta < k_\epsilon(s,\gamma_{2}^n(s)) < k_\epsilon(t_0, x_0) + \delta.
\end{gather*}
Subtracting those inequalities, integrating on $s$ in $[t_0, t_0 + \bar\rho h_n]$, and dividing by $h_n$, we deduce that
\[
\frac{1}{h_n} \abs*{\int_{t_0}^{t_0+\bar\rho h_n} \left[k_\epsilon\left(t_0 + \frac{s - t_0}{\bar\rho}, \gamma_{0}\left(t_0+\frac{s-t_0}{\bar\rho}\right)\right) - k_\epsilon(s,\gamma_{2}^n(s))\right] \diff s} < 2 \bar\rho \delta,
\]
concluding the proof of \eqref{eq:limit-diff-k-eps}. Now, \eqref{eq:estimate-tau-1-n} and \eqref{eq:limit-diff-k-eps} imply that
\begin{equation}
\label{eq:k-gamma2}
\int_{t_0}^{t_0+\tau_1^n}k_\epsilon(s, \gamma_{2}^n(s)) \diff s = \int_{t_0}^{t_0+\bar\rho h_n} k_\epsilon(s, \gamma_{2}^n (s))\diff s + o(h_n).
\end{equation}

Define $F: [0, \tau^n] \to \mathbbm R_+$ by $F(t)=\int_{t_0}^{t_0 + t} k_\epsilon(s,\gamma_{2}^n(s)) \diff s$, then obviously $F$ is continuous and increasing, which implies that $F^{-1}$ is well-defined on the range of $F$. Since $\dot F(t) = k(t, \gamma_{2}^n(t))$, $F$ is $K_{\max}$-Lipschitz continuous and, since $\frac{\diff}{\diff t}F^{-1}(t) = \frac{1}{\dot{F}(F^{-1}(t))}$, we also deduce that $F^{-1}$ is $\frac{2}{K_{\min}}$-Lipschitz continuous.
Therefore, by \eqref{eq:k-gamma2}, we deduce that
\begin{equation*}
\tau_1^n = F^{-1}(F(\bar\rho h_n) + o(h_n)) = \bar\rho h_n + o(h_n) \leq \rho h_n + o(h_n).
\end{equation*}
This concludes the proof of the claim in the case $\abs{\bar x_2^n - x_0} \leq \rho \abs{x_1^n - x_0}$, since $\tau_1^n = \tau^n$ in that case.

Otherwise, we have $\bar\rho = \rho$ and $\abs{x_{3}^n - x_{0}} = \rho \abs{x_{1}^n - x_{0}}$, and thus, from \eqref{eq:dist-x2-x0}, we get
\[
\abs*{\Bar{x}_{2}^n-x_{0}} < \rho (\abs*{ x_{1}^n-x_{0}}+\abs*{\Bar{x}_{2}^n- x_{1}^n}) = \abs*{x_{3}^n - x_{0}} + \rho\abs*{\Bar{x}_{2}^n- x_{1}^n}.
\]
On the other hand, since $x_{3}^n$ belongs to the segment $(x_0, \bar x_2^n)$, we have $\abs*{\Bar{x}_{2}^n - x_{0}} = \abs*{\Bar{x}_{2}^n - x_{3}^n}\allowbreak +\abs*{x_{3}^n - x_{0}}$, hence the inequality $\abs*{\Bar{x}_{2}^n - x_{3}^n} \le \rho \abs*{\Bar{x}_{2}^n- x_{1}^n}$ holds. Suppose $\tau_2^n$ is the time the trajectory $\gamma_{2}^n$ takes to go from $x_{3}^n$ to $\Bar{x}_{2}^n$, i.e., $\gamma_{2}^n(t_0 + \tau_1^n + \tau_2^n) = \Bar{x}_{2}^n$, and note that $\tau^n = \tau_1^n + \tau_2^n$. As before, we compare the times between $\abs*{ \Bar{x}_{2}^n - x_{3}^n}$ and $\abs*{ \Bar{x}_{2}^n- x_{1}^n}$. Let $\beta \leq \rho$ be such that $\abs*{\Bar{x}_{2}^n - x_{3}^n}=\beta \abs*{ \Bar{x}_{2}^n- x_{1}^n}$. Proceeding similarly to \eqref{eq:estimate-tau-1-n}, we get
\begin{equation*}
    \begin{aligned}
    & \int_{0}^{\tau_2^n} k_\epsilon(s+t_0+\tau_1^n,\gamma_{2}^n(s+t_0+\tau_1^n)) \diff s = \abs*{\Bar{x}_{2}^n - x_{3}^n} \\
		{} = {} & \beta \abs*{ \Bar{x}_{2}^n- x_{1}^n} = \beta \int_{t_1^n}^{t_2^n} k_\epsilon(s, \bar\gamma_1^n(s)) \diff s \\
		{} = {} & \int_{0}^{\beta h_n} k_\epsilon\left(\frac{s}{\beta}+t_0+h_n,\Bar{\gamma}_{1}^n\left(\frac{s}{\beta}+t_0+h_n\right)\right) \diff s \\
		{} = {} & \int_{0}^{\beta h_n} k_\epsilon(s+t_0+\tau_1^n,\gamma_{2}^n(s+t_0+\tau_1^n)) \diff s\\
    &  +\int_{0}^{\beta h_n} \left[k_\epsilon\left(\frac{s}{\beta}+t_0+h_n,\Bar{\gamma}_{1}^n\left(\frac{s}{\beta}+t_0+h_n\right)\right)-k_\epsilon(s+t_0+\tau_1^n,\gamma_{2}^n(s+t_0+\tau_1^n))\right] \diff s.
    \end{aligned}
\end{equation*}
Proceeding as in the proof of \eqref{eq:limit-diff-k-eps}, we can show that the last integral in the above expression is an $o(h_n)$ as $n \to +\infty$, and thus
\[
\int_{t_0}^{t_0 + \tau_2^n} k_\epsilon(s+\tau_1^n,\gamma_{2}^n(s+\tau_1^n)) \diff s = \int_{t_0}^{t_0 + \beta h_n} k_\epsilon(s+\tau_1^n,\gamma_{2}^n(s+\tau_1^n)) \diff s + o(h_n).
\]
Defining $F(t) = \int_{t_0}^{t_0 + t} k_\epsilon(s+\tau_1^n,\gamma_{2}^n(s+\tau_1^n)) \diff s$ and arguing similarly to above, we deduce that $\tau_2^n = \beta h_n + o(h_n)$. Therefore the time $\tau^n$ to reach $\Bar{x}_{2}^n$ from $x_{0}$ satisfies
\[
\tau^n = (\rho+\beta) h_n + o(h_n) \leq 2 \rho h_n + o(h_n). \qedhere
\]
\end{proof}

Let us now compare the trajectories $\Bar{\gamma}_{1}^n$ and $\gamma_{1}^n$ on $[t_1^n, t_2^n]$. Let $\delta_{1}^n(t)= \gamma_{1}^n(t)-\Bar{\gamma}_{1}^n(t)$. Hence, from the ODEs satisfied by the trajectories $\Bar{\gamma}_{1}^n$ and $\gamma_{1}^n$, we have 
\begin{equation*}
    \begin{aligned}
    \delta_{1}^n(t)&=\int_{t_1^n}^{t} \Big[k_\epsilon(s,\gamma_{1}^n(s))u_{1}^n(s)-k_\epsilon(s,\Bar{\gamma}_{1}^n(s))\Bar{u}_{1}^n \Big]\diff s\\
    &=\int_{t_1^n}^{t} \Big[k_\epsilon(s,\gamma_{1}^n(s))-k_\epsilon(s,\Bar{\gamma}_{1}^n(s))\Big]u_{1}^n(s) \diff s
    +\int_{t_1^n}^{t}k_\epsilon(s,\Bar{\gamma}_{1}^n(s))(u_{1}^n(s)-\Bar{u}_{1}^n) \diff s.
    \end{aligned}
\end{equation*}
Since $u_{1}^n$ is the optimal control associated with $\gamma_1^n$, by Proposition~\ref{prop:consequences-of-Pontryagin-for-optimal-trajectories-of-ocp-epsilon}, it is Lipschitz continuous in $[t_1^n, t_1^n + \varphi_\epsilon(t_1^n, x_1^n)]$ and its Lipschitz constant is independent of $n$. Therefore, denoting by $L_\epsilon > 0$ the Lipschitz constant of $k_\epsilon$ with respect to its second variable (which is independent of the first variable) and $C > 0$ the Lipschitz constant of $u_1^n$, we have
\[
\abs*{\delta_{1}^n(t)} \le L_\epsilon \int_{t_1^n}^{t} \abs*{ \delta_{1}^n(s)} \diff s + C K_{\max} \int_{t_1^n}^{t} \abs*{s-t_1^n} \diff s,
\]
and hence, by using Grönwall's inequality,
\[
\abs*{\delta_{1}^n(t)} \le C K_{\max} \frac{(t-t_1^n)^2}{2} e^{L_\epsilon(t-t_1^n)}.
\]
In particular, if we set $t=t_1^n+h_n$, then
\[
\abs*{x_{2}^n- \Bar{x}_{2}^n} \le C K_{\max} \frac{h_n^2}{2} e^{L_\epsilon h_n} = O(h_n^2).
\]

Let $u_{3}^n = \frac{ x_{2}^n- \Bar{x}_{2}^n}{\abs*{ x_{2}^n- \Bar{x}_{2}^n}}$ (with the convention $x_3^n = 0$ if $x_2^n = \bar x_2^n$) and $\gamma_{3}^n$ be the solution of
\begin{equation}
    \left \{
    \begin{aligned}
    \Dot{\gamma}_{3}^n(t)&=k_\epsilon(t,\gamma_{3}^n(t)) u_{3}^n
    \\  
    \gamma_{3}^n(t_0 + \tau^n)&=\Bar{x}_{2}^n.       
    \end{aligned} \right.
\end{equation}
Using the lower bound $\frac{1}{2} K_{\min}$ on $k_\epsilon$ and the fact that $\abs*{x_{2}^n- \Bar{x}_{2}^n} = O(h_n^2)$, one can easily deduce that the time $\sigma^n$ from $\Bar{x}_{2}^n$ to $ x_{2}^n$ along $\gamma_3^n$ (i.e., the value $\sigma^n > 0$ such that $\gamma_3^n(t_0 + \tau^n + \sigma^n) = x_2^n$) satisfies $\sigma^n = O(h_n^2)$.

We have thus constructed two ways to go from $x_0$ to $x_2^n$. The first one is to choose the path containing $x_{0}$, $x_{1}^n$, and $x_{2}^n$, which corresponds to the concatenation of the trajectories $\gamma_0$ on $[t_0, t_1^n]$ and $\gamma_1^n$ on $[t_1^n, t_2^n]$, and the second one is the path containing $x_{0}$, $\Bar{x}_{2}^n$, and $x_{2}^n$, which corresponds to the concatenation of the trajectories $\gamma_2^n$ on $[t_0, t_0 + \tau^n]$ and $\gamma_3^n$ on $[t_0 + \tau^n, t_0 + \tau^n + \sigma^n]$. Letting $T_1^n$ and $T_2^n$ be the times for going from $x_0$ to $x_2^n$ along these two paths, respectively, we have, by construction and the claim, that $T_1^n = 2 h_n$ and $T_2^n = \tau^n + \sigma^n \leq 2\rho h_n + o(h_n)$. Hence, since $\rho < 1$, we have, for $n$ large enough, that $T_2^n < T_1^n$.

By definition of $\gamma_0$, we have $x_1^n = \gamma_0(t_1^n) = x_0 + h_n k(t_0, x_0) u_0 + o(h_n)$ and, using \eqref{eq:u0-in-W} and the fact that $\varphi_\epsilon$ is Lipschitz continuous (Proposition~\ref{prop:varphi-eps-loc-Lip}), we deduce that
\[
\varphi(t_0, x_{0}) = \varphi_\epsilon(t_1^n, x_{1}^n) + h_n + o(h_n).
\]
Moreover, since $\gamma_{1}^n \in \Opt_\epsilon(k_\epsilon, t_1^n, x_{1}^n)$, we have from \eqref{eq:DPP-epsilon} that $\varphi_\epsilon(t_1^n, x_1^n) = \varphi_\epsilon(t_2^n, x_2^n) + h_n$, yielding that
\[
\varphi(t_0, x_{0}) = \varphi(t_2^n, x_{2}^n) + T_1^n + o(h_n).
\]
On the other hand, since the path from $x_0$ to $x_2^n$ going through $\bar x_2^n$ is an admissible trajectory for $k_\epsilon$, we have, by \eqref{eq:DPP-epsilon}, that
$\varphi(t_0, x_{0}) \le T_2^n + \varphi_\epsilon(t_0 + T_2^n, x_{2}^n)$. Hence
\begin{equation}\label{time contr}
    \varphi_\epsilon(t_2^n, x_{2}^n) + T_1^n + o(h_n) \le T_2^n + \varphi_\epsilon(t_0 + T_2^n, x_{2}^n).
\end{equation}
We also know that $t_0 + T_2^n < t_0 + T_1^n = t_2^n$ for $n$ large enough. Therefore\footnote{Strictly speaking, Proposition~\ref{PropOCP-2}\ref{PropOCP-LowerBound} only applies to $\varphi$, and not to $\varphi_\epsilon$. However, one can still get its conclusion by arguing as follows. Let $\widehat\OCP$ be defined as in the proof of Proposition~\ref{prop:varphi-eps-loc-Lip} with $\eta = \frac{\epsilon}{2}$ and consider its value function $\widehat\varphi$. By the arguments provided in that proof, $\widehat\varphi$ satisfies Proposition~\ref{PropOCP-2}\ref{PropOCP-LowerBound} and $\widehat\varphi$ and $\varphi_\epsilon$ coincide in $\bar\Omega_{\epsilon/2}$. The conclusion then follows since all points involved here belong to $\bar\Omega_{\epsilon/2}$ for $n$ large enough. Note that the constant $c > 0$ depends on $\epsilon$.}, by Proposition~\ref{PropOCP-2}\ref{PropOCP-LowerBound}, there exists a constant $c > 0$ such that
\begin{equation*}
    \varphi_\epsilon(t_2^n,x_{2}^n) > \varphi_\epsilon(t_0 +T_2^n,x_{2}^n) + (c-1) (t_2^n - t_0 - T_2^n) = \varphi_\epsilon(t_0+T_2^n, x_{2}^n) + (c-1) (T_1^n - T_2^n),
\end{equation*}
and, using \eqref{time contr}, we get $(c-1)(T_1^n-T_2^n)+T_1^n+o(h_n)\le T_2^n$, which leads to 
\[
2 h_n + o(h_n) = T_1^n+o(h_n) \le T_2^n \leq 2\rho h_n + o(h_n).
\]
Divide above inequality by $h_n$ to observe that
\[
2+o(1) \leq 2\rho+o(1).
\]
Finally by letting $n \to +\infty$, we conclude that $\rho\ge 1$, which is a contradiction. Therefore Case~2 will never happen and this ends the proof.
\end{proof}

We conclude this section with a technical result showing that, for $x \in \partial\Omega \setminus \Gamma$, if all directions of $\mathcal W(t, x)$ point to the inside of the domain, then no optimal trajectories starting at time $0$ are close to $x$ at time $t$. For that purpose, we introduce 

\begin{proposition}
\label{prop:no-opt-in-neighborhood}
Consider the optimal control problem $\OCP(k)$ and assume that \HypoOmega, \ref{HypoOCP-k-Bound}, and \ref{HypoOCP-k-Lip} hold. Let $x \in \partial\Omega \setminus \Gamma$, $t > 0$, and assume that there exists $w \in \mathcal W(t, x)$ such that $w \cdot \mathbf n(x) < 0$. Then there exists a neighborhood $N$ of $x$ in $\bar\Omega$ such that, for every $x_0 \in \bar\Omega$ and $\gamma \in \Opt(k, 0, x_0)$, we have $\gamma(t) \notin N$.
\end{proposition}

\begin{proof}
Assume, to obtain a contradiction, that there exist sequences $(x_{0, n})_{n \in \mathbbm N}$ and $(\gamma_n)_{n \in \mathbbm N}$ with $x_{0, n} \in \bar\Omega$ and $\gamma_n \in \Opt(k, 0, x_{0, n})$ for every $n \in \mathbbm N$ and such that $\gamma_n(t) \to x$ as $n \to +\infty$. Since $\bar\Omega$ is compact and $(\gamma_n)_{n \in \mathbbm N}$ is a sequence of functions which are $K_{\max}$-Lipschitz continuous, applying Arzelà--Ascoli Theorem, we deduce that there exist $x_0 \in \bar\Omega$ and $\gamma \in \Lip_{K_{\max}}(\bar\Omega)$ such that, up to extracting subsequences (which we still denote by $(x_{0, n})_{n \in \mathbbm N}$ and $(\gamma_n)_{n \in \mathbbm N}$ for simplicity), we have, as $n \to +\infty$, that $x_{0, n} \to x_0$ and $\gamma_n \to \gamma$ uniformly on compact time intervals. By straightforward arguments based on the continuity of the value function, we deduce that $\gamma \in \Opt(k, 0, x_0)$, and in addition we have $\gamma(t) = x$.

Recall that, by Corollary~\ref{coro:smooth}, we have $\gamma \in \mathbf C^1([0, T]; \bar\Omega)$, where $T = \varphi(0, x_0)$ and $\varphi$ is the value function of $\OCP(k)$. Since $t > 0$ and $x \notin \Gamma$, we have $t \in (0, T)$, and thus $\gamma$ is differentiable at $t$. Moreover, by Proposition~\ref{PropUSingleElement} and Theorem~\ref{thm Ut_0,x_0}, the set $\mathcal W(t, x)$ contains exactly one element, which we denote by $w_0 \in \mathbbm S^{d-1}$, and thus, by assumption, we have $w_0 \cdot \mathbf n(x) < 0$. By Definition~\ref{def:U} and Theorem~\ref{thm Ut_0,x_0}, we deduce also that $\dot\gamma(t) = k(t, x) w_0$.

Let $\alpha: h \mapsto d^{\pm}_{\partial\Omega}(\gamma(t + h))$ be defined on an open neighborhood of $0$ in $\mathbbm R$. Then $\alpha(0) = d^{\pm}_{\partial\Omega}(x) = 0$ and $\dot\alpha(0) = k(t, x) \mathbf n(x) \cdot w_0 < 0$. In particular, there exists $h_0 \in [-t, 0)$ such that, for every $h \in [h_0, 0)$, we have $\alpha(h) > 0$, and thus $\gamma(t + h) \notin \bar\Omega$, which contradicts the fact that $\gamma$ takes values in $\bar\Omega$. This contradiction yields the conclusion.
\end{proof}

\subsection{The normalized gradient}
\label{sec:normalized-grad}

Now that Theorem~\ref{thm Ut_0,x_0} characterizes the set of optimal directions $\mathcal U(t_0, x_0)$ as the set $\mathcal W(t_0, x_0)$ of directions of maximal descent of $\varphi$, we provide the following definition, which is motivated by Proposition~\ref{PropWNormalizedGrad}.

\begin{definition}
\label{def:normalized-gradient}
Consider the optimal control problem $\OCP(k)$ and its value function $\varphi$ and assume that \HypoOmega, \ref{HypoOCP-k-Bound}, and \ref{HypoOCP-k-Lip} hold. Let $(t_0,x_0)\in \mathbbm R_+\times \bar{\Omega}$. If $\mathcal{W}(t_0,x_0)$ contains exactly one element, we denote this element by $-\widehat{\nabla\varphi}(t_0,x_0)$, and call $\widehat{\nabla\varphi}(t_0,x_0)$ the \emph{normalized gradient} of $\varphi$ at $(t_0, x_0)$.
\end{definition}

As a consequence of Proposition~\ref{PropUSingleElement} and Theorem~\ref{thm Ut_0,x_0}, we immediately obtain the following characterization of optimal controls.

\begin{corollary}\label{coro normalized}
Consider the optimal control problem $\OCP(k)$ and its value function $\varphi$ and assume that \HypoOmega, \ref{HypoOCP-k-Bound}, and \ref{HypoOCP-k-Lip} hold. Let $(t_0,x_0)\in \mathbbm R_+ \times \bar{\Omega}$, $\gamma\in \Opt(k, t_0, x_0)$, and $T = \varphi(t_0, x_0)$. Then for every $t \in (t_0, t_0 + T)$, $\varphi$ admits a normalized gradient at $(t,\gamma(t))$ and
\[
\Dot{\gamma}(t)= -k(t,\gamma(t)) \widehat{\nabla \varphi}(t,\gamma(t)).
\]
\end{corollary}

Combining the above result with Corollary~\ref{coro:smooth}, we deduce that, for every optimal trajectory $\gamma$, the map $t \mapsto \widehat{\nabla \varphi}(t, \gamma(t))$ is Lipschitz continuous as long as $\gamma(t)\in \bar{\Omega} \setminus \Gamma$ and $t$ is larger than the initial time of $\gamma$. However, this provides no information on the regularity of $(t, x) \mapsto \widehat{\nabla \varphi}(t, x)$. We are interested in proving the continuity of this map on its domain of definition. For that purpose, we first prove the upper semi-continuity of the set-valued map $\mathcal U$ (we refer to \cite[Definition~1.4.1]{AubinFranskowska} for the definition of upper semi-continuity for set-valued maps).

\begin{proposition}\label{U upper semi}
Consider the optimal control problem $\OCP(k)$ and assume that \HypoOmega, \ref{HypoOCP-k-Bound}, and \ref{HypoOCP-k-Lip} hold. Let $\mathcal{U}:\mathbbm R_+\times \bar{\Omega} \rightrightarrows \mathbbm S^{d-1}$ be the set valued map introduced in Definition~\ref{def:U}. Then $\mathcal{U}$ is upper semi-continuous.
\end{proposition}

\begin{proof}
Since $\mathbbm S^{d-1}$ is a compact set, it suffices to show that $\mathcal{U}$ has a closed graph (see, e.g., \cite[Proposition~1.4.8]{AubinFranskowska}).
Let $(t_n, x_n)_{n \in \mathbbm N}$ be a sequence in $\mathbbm R_+ \times \bar\Omega$ converging as $n \to +\infty$ to some $(t_0, x_0) \in \mathbbm R_+ \times \bar\Omega$, $(\bar{u}_{n})_{n\in \mathbbm N}$ be a sequence such that $\bar{u}_{n} \in \mathcal{U}(t_n, x_n)$ for every $n \in \mathbbm N$ and $\bar u_n \to \bar u_0$ as $n \to +\infty$ for some $\bar{u}_{0}\in \mathbbm S^{d-1}$. Since $\bar{u}_{n} \in \mathcal{U}(t_n, x_n)$ for every $n \in \mathbbm N$, there exists a sequence $(\gamma_n)_{n \in \mathbbm N}$ of optimal trajectories with $\gamma_{n} \in \Opt(k, t_n, x_n)$ for every $n \in \mathbbm N$ and a corresponding sequence of the associated optimal controls $(u_{n})_{n \in \mathbbm N}$ with $u_n \in \Lip([t_n, t_n + \varphi(t_n, x_n)]; \mathbbm S^{d-1})$ and $u_{n}(t_n) = \bar{u}_{n}$ for every $n \in \mathbbm N$. From Corollary~\ref{coro:smooth}, up to modifying $u_n$ outside of the interval $[t_n, t_n + \varphi(t_n, x_n)]$, the sequences $(\gamma_n)_{n \in \mathbbm N}$ and $(u_n)_{n \in \mathbbm N}$ are sequences of Lipschitz continuous functions with Lipschitz constants independent of $n$, and thus, by Arzelà--Ascoli Theorem, one finds elements $\gamma^* \in \Lip(\bar\Omega)$ and $u^* \in \Lip(\mathbbm S^{d-1})$ such that, up to a subsequence, $\gamma_{n}\to \gamma^*$ and $u_{n}\to u^*$ uniformly on compact time intervals. In particular,
\[
u^*(t_0)=\lim_{n\to \infty} u_{n}(t_n)=\lim_{n\to \infty} \bar{u}_{n}=\bar{u}_{0}
\]
and
\[
\gamma^*(t_0)=\lim_{n\to \infty} \gamma_{n}(t_n)=\lim_{n\to \infty} x_{n}=x_{0}.
\]
By using the dynamic programming principle from Proposition~\ref{PropOCP-1}\ref{PropOCP-DPP} and the Lipschtiz continuity of the value function from Proposition~\ref{varphi is Lipschitz}, one observes that the restriction of $u^*$ to $[t_0, t_0 + \varphi(t_0, x_0)]$ is the optimal control corresponding to the optimal trajectory $\gamma^*$. Hence $\bar{u}_{0} = u^\ast(t_0) \in \mathcal{U}(t_0,x_0)$, concluding the proof that $\mathcal{U}$ has a closed graph.
\end{proof}

On the set of points where a set-valued map is single-valued, upper semi-continuity coincides with standard continuity of single-valued functions. As an immediate consequence of this fact, Proposition~\ref{U upper semi}, and Theorem~\ref{thm Ut_0,x_0}, we have the following result.

\begin{corollary}
\label{coro:norm-grad-continuous}
Consider the optimal control problem $\OCP(k)$ and its value function $\varphi$ and assume that \HypoOmega, \ref{HypoOCP-k-Bound}, and \ref{HypoOCP-k-Lip} hold. Let $\mathcal{W}:\mathbbm R_+\times \bar{\Omega} \rightrightarrows \mathbbm S^{d-1}$ be the set valued map introduced in Definition~\ref{DefW} and $\widehat{\nabla\varphi}$ be the normalized gradient of $\varphi$. Then $\mathcal{W}$ is upper semi-continuous and $\widehat{\nabla\varphi}$ is a continuous function on the set where it is defined.
\end{corollary}

\section{The MFG system}\label{sec MFG system}

We now turn to the characterization of equilibria $Q \in \mathcal P(\mathbf C(\bar\Omega))$ of a mean field game $\MFG(K)$ through a system of PDEs, known as the MFG system, consisting of a continuity equation on the density of agents $t \mapsto m_t$, defined through the relation $m_t=e_{t\#}Q$, coupled with the Hamilton--Jacobi equation \eqref{H-J equation} on the value function $\varphi$ of the optimal control problem $\OCP(k)$, where $k$ is defined from $K$ and $Q$ by setting $k(t, x) = K(m_t, x)$. The main difficulty in this characterization lies within the continuity equation for $m_t$, and more precisely on the characterization of the corresponding velocity field. Corollary~\ref{coro normalized} suggests that such velocity field should be given by the opposite of the normalized gradient of $\varphi$ multiplied by the dynamics $k$.

Corollary~\ref{coro:norm-grad-continuous} states the continuity of the normalized gradient $\widehat{\nabla\varphi}$ on the set of points of $\mathbbm R_+ \times \bar\Omega$ where it is defined. We will now show that, for the purposes of studying the equilibria of mean field games, this set is quite ``large''. More precisely, consider the mean field game $\MFG(K)$ and an equilibrium $Q \in \mathcal P(\mathbf C(\bar \Omega))$ for this game. We will prove that, for every fixed $t > 0$, the set of points $x \in \bar\Omega \setminus \Gamma$ at which $\widehat{\nabla\varphi}$ does not exist has $m_t$ measure zero, where $m_t$ is the evaluation at time $t$ of the equilibrium measure $Q$.

For that purpose, let us introduce the set
\begin{equation}\label{Upsilon}
    \Upsilon=\left\{(t,x)\in \mathbbm R_+^\ast \times (\bar{\Omega} \setminus \Gamma) \suchthat \exists t_0\in [0,t),\, \exists x_{0}\in \bar{\Omega},\, \exists \gamma\in \Opt(k, t_0, x_{0})\text{ s.t.\ } \gamma(t) = x\right\}.
\end{equation}
In other words, $\Upsilon$ contains all points $(t,x)\in \mathbbm R_+^\ast \times (\bar{\Omega} \setminus \Gamma)$ which are strictly between the starting and the final points of an optimal trajectory. In particular, it follows from Corollary~\ref{coro normalized} that $\widehat{\nabla \varphi}(t,x)$ exists for every $(t,x)\in \Upsilon$, and, by Corollary~\ref{coro:norm-grad-continuous}, $\widehat{\nabla\varphi}$ is continuous in $\Upsilon$. We also introduce, for $t > 0$, the set
\begin{equation}
\label{Upsilon_t}
\Upsilon_{t}=\left\{ x\in \bar{\Omega} \setminus \Gamma \suchthat (t,x)\in \Upsilon \right\}.
\end{equation}

\begin{proposition}\label{set grad coninuity}
Consider the mean field game $\MFG(K)$ under the assumptions \HypoOmega, \ref{HypoMFG-K-Bound}, and \ref{HypoMFG-K-Lip}. Let $Q$ be an equilibrium for $\MFG(K)$, set $m_t = e_{t\#} Q$ for $t \geq 0$, define $k: \mathbbm R_+ \times \bar\Omega \to \mathbbm R_+$ by $k(t, x) = K(m_t, x)$, consider the optimal control problem $\OCP(k)$, and let $\Upsilon$ and $\Upsilon_t$ be defined as in \eqref{Upsilon} and \eqref{Upsilon_t}, respectively. Then for every $t>0$, we have $m_t(\bar{\Omega} \setminus (\Gamma \cup \Upsilon_{t})) = 0$.
\end{proposition}

\begin{proof}
Let $\OOpt = \bigcup_{x_0 \in \bar\Omega} \Opt(k, 0, x_0)$. Since $Q$ is an equilibrium of $\MFG(K)$, then $Q(\OOpt) = 1$. For every $t > 0$, from the definition of $\Upsilon$, one has that $\{\gamma \in \OOpt \suchthat \gamma(t) \in \bar{\Omega} \setminus (\Gamma \cup \Upsilon_{t})\} = \varnothing$, and then $m_t(\bar{\Omega}\setminus (\Gamma \cup \Upsilon_{t})) = Q\Bigl(\{ \gamma \in \OOpt \suchthat \gamma(t) \in \bar\Omega \setminus (\Gamma \cup \Upsilon_{t})\}\Bigr)=Q(\varnothing)=0$.
\end{proof}

We are now ready to provide our main result concerning the MFG system of $\MFG(K)$.

\begin{theorem}\label{Thm MFG system}
Consider the mean field game $\MFG(K)$ under assumptions \HypoOmega, \ref{HypoMFG-K-Bound}, and \ref{HypoMFG-K-Lip}. Let $m_0 \in \mathcal{P}(\bar{\Omega})$, $Q$ be an equilibrium of $\MFG(K)$ with initial condition $m_0$, $m_t \in \mathcal P(\bar\Omega)$ be defined for $t > 0$ by $m_t = e_{t\#} Q$, $k: \mathbbm R_+ \times \bar\Omega \to \mathbbm R_+$ be defined from $Q$ and $K$ by $k(t, x) = K(e_{t\#} Q, x)$, $\varphi$ be the value function of $\OCP(k)$, $\mathcal W$ be the set-valued map provided in Definition~\ref{DefW}, and $\widehat{\nabla\varphi}$ be the normalized gradient of $\varphi$ from Definition~\ref{def:normalized-gradient}. For $t > 0$, set
\[
\partial\Omega_t^- = \{x \in \partial\Omega \setminus \Gamma \suchthat \exists w \in \mathcal W(t, x) \text{ such that } w \cdot \mathbf n(x) < 0\}.
\]
Then $(m, \varphi)$ solves the MFG system
\begin{equation}\label{MFGs system}
    \left \{
    \begin{aligned}
    &\partial_t m_t(x) - \diverg\left (m_t(x) K(m_t, x) \widehat{\nabla \varphi}(t, x) \right)=0, & & \text{ in }\mathbbm R_+^*\times (\bar{\Omega} \setminus \Gamma), \\
    &-\partial_t \varphi(t,x) + \abs*{\nabla \varphi(t,x)} K(m_t, x) - 1 = 0, & & \text{ in } \mathbbm R_+\times (\bar{\Omega}\setminus \Gamma),\\
    &m_t(x) = 0, & & \text{ for } t > 0 \text{ and } x \in \partial\Omega_t^-,\\
    &\varphi(t,x) = 0, & & \text{ on } \mathbbm R_+\times \Gamma,\\
    &\nabla\varphi(t,x) \cdot \mathbf n(x) \ge 0, & & \text{ on } \mathbbm R_+\times (\partial\Omega\setminus \Gamma),\\
    &m_t  = m_0, & & \text{ in } \{0\}\times \bar \Omega,\\
    \end{aligned} \right.
\end{equation}
where the first equation is satisfied in the sense of distributions, the second and fifth equations are satisfied in the viscosity senses of Proposition~\ref{PropOCP-HJ} and Theorem~\ref{Thm viscosity boundary cond}, respectively, and the third equation is satisfied in the following sense: for every $t > 0$ and $x \in \partial\Omega_t^-$, there exists a neighborhood $N$ of $x$ such that $m_t(N) = 0$.
\end{theorem}

\begin{proof}
Notice first that, since $K$ satisfies \ref{HypoMFG-K-Bound} and \ref{HypoMFG-K-Lip}, then $k$ satisfies \ref{HypoOCP-k-Bound} and \ref{HypoOCP-k-Lip}. The Hamilton--Jacobi equation on $\varphi$ and its boundary conditions then follow immediately from Proposition~\ref{PropOCP-HJ} and Theorem~\ref{Thm viscosity boundary cond}. The initial condition on $m_t$ follows from its definition.

Let us prove that $m_t$ satisfies the continuity equation in \eqref{MFGs system}. First, thanks to Corollary~\ref{coro:norm-grad-continuous}, $\widehat{\nabla \varphi}$ is continuous on the set $\Upsilon$ defined in \eqref{Upsilon}. Let $\xi\in \mathbf{C}_c^{\infty}(\mathbbm R_+^*\times (\bar{\Omega} \setminus \Gamma); \mathbbm R)$ be a test function. Take $\gamma \in \Opt(k, 0, x_0)$ and let $T=\varphi(0,x_{0})$. By Corollary~\ref{coro normalized}, one has $\Dot{\gamma}(t)= -K(m_t,\gamma(t)) \widehat{\nabla \varphi}(t,\gamma(t))$ for every $t\in (0,T)$. Hence
\begin{equation}\label{test func}
    \frac{\diff}{\diff t}\Big(\xi(t,\gamma(t)) \Big)=\partial_t\,\xi(t,\gamma(t))-\nabla_{x}\,\xi(t,\gamma(t)) \cdot \widehat{\nabla \varphi}(t,\gamma(t))\,K(m_t,\gamma(t)).
\end{equation}
Let us denote the set of all optimal trajectories by $\OOpt$, i.e., $\OOpt = \bigcup_{x_0 \in \bar\Omega} \Opt(k, 0, x)$. Thanks to the continuity of the right-hand side of \eqref{test func} on $\mathbbm R_+^* \times \OOpt$, one can integrate to observe that
\begin{equation*}
    \begin{aligned}
    & \int_{0}^{\infty}\int_{\OOpt}\, \frac{\diff}{\diff t}\Big(\xi(t,\gamma(t)) \Big) \diff Q(\gamma)\diff t \\
    {} = {} &\int_{0}^{\infty}\int_{\OOpt}\,\partial_t\,\xi(t,\gamma(t))\diff Q(\gamma)\diff t\\
    &{} -\int_{0}^{\infty}\int_{\OOpt}\,\nabla_{x}\,\xi(t,\gamma(t)) \cdot \widehat{\nabla \varphi}(t,\gamma(t))\,K(m_t,\gamma(t)) \diff Q(\gamma)\diff t.
    \end{aligned}
\end{equation*}
Since $\xi$ is compactly supported, the left-hand side of the above equality is zero. Hence, by using the Proposition~\ref{set grad coninuity} and the relation between $m_t$ and $Q$, one concludes that
\begin{equation*}
    \int_{0}^{\infty}\int_{\bar{\Omega}}\,\partial_t\xi(t,x)\diff m_t(x)\diff t=\int_{0}^{\infty}\int_{\bar{\Omega}} \nabla_{x}\,\xi(t,x) \cdot \widehat{\nabla \varphi}(t,x)\,K(m_t,x) \diff m_t(x)\diff t,
\end{equation*}
which is precisely the weak formulation of the continuity equation in \eqref{MFGs system}.

Finally, let us prove the boundary condition on $m_t$. Let $t > 0$ and $x \in \partial\Omega_t^-$. Then, using the definition of $m_t$ and Proposition~\ref{prop:no-opt-in-neighborhood}, there exists a neighborhood $N$ of $x$ such that
\[
m_t(N) = Q\left(\{\gamma \in \OOpt \suchthat \gamma(t) \in N\}\right) = Q(\varnothing) = 0,
\]
as required.
\end{proof}

\begin{remark}
Note that both the continuity and the Hamilton--Jacobi equations in \eqref{MFGs system} are satisfied in the spatial domain $\bar\Omega \setminus \Gamma$ and, in particular, they are satisfied up to the boundary, excluding the target set $\Gamma$. The velocity field at a point $(t, x) \in \mathbb R_+^\ast \times (\partial\Omega \setminus \Gamma)$ in the boundary is thus $- K(m_t, x) \widehat{\nabla\varphi}(t, x)$, i.e., the same as in the interior. Note also that the fifth equation in \eqref{MFGs system} shows that such a velocity field either points towards the inside of the domain or is tangent to it. In particular, contrarily to the mean field game system obtained in \cite{CannarsaMean} for mean field games with state constraints, fixed final time, and more general cost functions, no correction to the velocity field is needed in the boundary of the domain in the present model due to the particular structure of the optimization criterion solved by each agent.
\end{remark}

\bibliographystyle{abbrv}
\bibliography{main}

\end{document}